\RequirePackage{fix-cm}
\documentclass[aos,preprint]{imsart}

\usepackage{amsthm,amsmath,amsfonts,amssymb}
\usepackage[numbers,sort&compress]{natbib}
\usepackage[colorlinks,citecolor=blue,urlcolor=blue]{hyperref}
\usepackage{bm}
\usepackage{dsfont}
\usepackage{mathtools}
\usepackage{microtype}
\usepackage{enumitem}
\usepackage{mathrsfs}
\usepackage{xcolor}
\usepackage{tikz}
\usepackage{pgfplots}
\usepackage[normalem]{ulem}

\DeclareFontFamily{U}{rsfs}{\skewchar\font127}
\DeclareFontShape{U}{rsfs}{m}{n}{%
   <5> <5.5> <6> rsfs5
   <7> rsfs7
   <8> <9> <10> <10.95> <12> <14.4> <17.28> <20.74> <24.88> rsfs10
}{}
\DeclareFontShape{T1}{ptm}{m}{scit}{<->ssub * ptm/m/it}{}

\startlocaldefs

\theoremstyle{plain}
\newtheorem{theorem}{Theorem}
\newtheorem{cor}[theorem]{Corollary}
\newtheorem{lemma}[theorem]{Lemma}
\newtheorem{prop}[theorem]{Proposition}

\theoremstyle{definition}
\newtheorem{definition}{Definition}
\newtheorem{remark}{Remark}
\newtheorem{assumption}{Assumption}

\newcommand{\parhead}[1]{\par\smallskip\noindent{\normalfont\bfseries\itshape #1.}\enspace}

\newtheorem*{theorem*}{Theorem}

\newcommand{\de}{\mathsf d}

\newcommand{\cB}{\mathcal{B}}

\newcommand{\cE}{\mathcal{E}}
\newcommand{\cF}{\mathcal{F}}

\newcommand{\cH}{\mathcal{H}}

\newcommand{\cK}{\mathcal{K}}
\newcommand{\cL}{\mathcal{L}}

\newcommand{\cR}{\mathcal{R}}

\newcommand{\cV}{\mathcal{V}}

\newcommand{\cX}{\mathcal{X}}

\newcommand{\bbR}{\mathbb{R}}

\newcommand{\sD}{\mathsf{D}}
 
\newcommand{\Id}{\mathrm{I}}

\newcommand{\dd}{\mathrm{d}}

\newcommand{\eps}{\epsilon}
\newcommand{\normal}{\mathsf{N}}

\newcommand{\E}{\mathbb{E}}
\renewcommand{\P}{\mathbb{P}}

\DeclarePairedDelimiter\bkt{[}{]}     
\makeatletter
\newcommand{\@exstar}[1]{\E \bkt*{#1}}
\newcommand{\@exnostar}[2][]{\E \bkt[#1]{#2}}
\newcommand{\ex}{\@ifstar\@exstar\@exnostar}
\makeatother

\makeatletter
\newcommand{\@prstar}[1]{\P \bkt*{#1}}
\newcommand{\@prnostar}[2][]{\P \bkt[#1]{#2}}
\newcommand{\pr}{\@ifstar\@prstar\@prnostar}
\makeatother

\DeclareMathOperator{\cov}{\mathsf{Cov}}

\DeclareMathOperator{\var}{\mathsf{Var}}

\pgfplotsset{compat=newest} 

\let\originalleft\left
\let\originalright\right
\renewcommand{\left}{\mathopen{}\mathclose\bgroup\originalleft}
\renewcommand{\right}{\aftergroup\egroup\originalright}

\newcommand{\wenxuan}[1]{\textcolor{red}{Wenxuan:~#1}}

\endlocaldefs

\begin{document}
\begin{frontmatter}
\title{Gaussian Comparison Theorems for High-Dimensional Posterior Inference}
\runtitle{Gaussian Comparison for Posterior Inference}

\begin{aug}
\author[A]{\fnms{Wenxuan}~\snm{Zou}\ead[label=e1]{wenxuan.zou@duke.edu}}
\author[B]{\fnms{Galen}~\snm{Reeves}\ead[label=e2]{galen.reeves@duke.edu}}
\address[A]{Department of Physics, Duke University\printead[presep={,\ }]{e1}}
\address[B]{Departments of Statistical Science and Electrical and Computer Engineering,
Duke University\printead[presep={,\ }]{e2}}
\end{aug}

\begin{abstract}
We develop a Gaussian comparison theory for posterior inference in
non-Gaussian high-dimensional models. The framework allows for model
misspecification and does not require the posterior distribution itself to be
approximately Gaussian. Working directly with likelihood processes on
separable function spaces, we establish explicit nonasymptotic free energy
bounds controlled by first two moment discrepancies and local third moments. A
perturbation-and-convexity argument transfers these comparisons to sandwich
bounds for posterior mean squared error and posterior variance, including
one-sided conclusions at nondifferentiable phase transitions. In
high-dimensional product models, these conditions reduce to local moment
controls and a canonical feature-radius scaling. Under weak regularity
requirements, the resulting universality theory applies to a broad class of
likelihood models, ranging from exponential families to more general regular
local log-likelihoods. The framework is illustrated on sparse Bernoulli
hypergraph inference.
\end{abstract}


\end{frontmatter}

\section{Introduction}


A central theme in asymptotic statistics is that complex statistical
models can often be approximated by simpler Gaussian ones. In classical
fixed-dimensional settings, this principle is formalized through local
asymptotic normality and related Gaussian approximation
theories~\cite{lecam:1986,lecam:2000,vaart:1998}. After centering and
rescaling around the true parameter, the likelihood ratio process
becomes asymptotically quadratic, the score becomes Gaussian, and the
statistical experiment is locally approximated by a Gaussian shift
model. These approximations underlie many classical results for
estimators, tests, and posterior distributions, including Laplace
approximations~\cite{kadane:1986} and Bernstein--von Mises
theorems~\cite{vaart:1998,lehmann:1998}.


Modern high-dimensional inference problems exhibit fundamentally
different behavior. The parameter dimension and sample size may grow
simultaneously~\cite{baik:2005,hastie:2022,perry:2018}, and local
Gaussian approximations need not accurately describe either the
likelihood process or the posterior distribution.
In particular, posterior distributions may remain highly non-Gaussian
and multimodal, leading to sharp threshold phenomena in inferential
behavior as signal strength or sample size
vary~\cite{lelarge:2018,lesieur:2017a,barbier:2019,barbier:2022}. Nevertheless, Gaussianity
often persists at the level of posterior observables. Quantities such
as posterior mean squared error and posterior variance asymptotically
coincide with those induced by an associated Gaussian model. The
present work develops a general framework for this phenomenon through
Gaussian comparison of posterior-generating likelihood processes.

We consider statistical inference based on observed data $W$ generated under a law $P_\theta$, where the true parameter $\theta$ belongs to a parameter space $\Theta$. The statistical model is specified by a prior distribution $\pi$ on a candidate space $T \subseteq \Theta$ together with a criterion or log-likelihood function $\ell(t;W)$ indexed by candidate parameters $t \in T$. The criterion induces a random likelihood process $X(t)=\ell(t;W)$, 
which determines the posterior distribution
\begin{equation}\label{eq:intro_posterior}
\pi_X(\dd t)
=
\frac{\exp\{X(t)\}\pi(\dd t)}
{\int_T \exp\{X(s)\}\pi(\dd s)}.
\end{equation}
Importantly, the criterion $\ell$ need not coincide with the true log-likelihood associated with the data-generating law $P_\theta$. Consequently, the framework naturally includes both well-specified Bayesian inference and inference under model mismatch. It also encompasses generalized or Gibbs posterior constructions based on alternative loss or scoring functions.

Our primary objects of interest are posterior observables associated with the likelihood process. For a feature map $\eta \colon \Theta \to \cH$ into a Hilbert space $\cH$, let $\widehat\eta_X = \int_T \eta(t)\,\pi_X(\dd t)$ denote the posterior mean estimator. With $\E_\theta$ denoting expectation under the data-generating law $P_\theta$, the principal observables studied in this paper are the corresponding posterior mean squared error and posterior variance:
\begin{equation}\label{eq:intro_mse_variance}
\begin{aligned}
\cR(\theta)
\coloneqq \E_\theta\big[\|\eta(\theta)-\widehat\eta_X\|^2 \big],
\qquad
\cV(\theta)
\coloneqq
\E_\theta \Big[\int_T\|\eta(t)-\widehat\eta_X\|^2\,\pi_X(\dd t)\Big].
\end{aligned}
\end{equation}
These quantities measure, respectively, posterior estimation accuracy and uncertainty. 

To study these observables, we compare the original likelihood process $X$ with a Gaussian process $Y=(Y(t))_{t\in T}$ indexed by the same parameter space $T$ and having matching mean and covariance structure. The Gaussian process induces a posterior distribution through the same prior $\pi$. In high-dimensional settings, many non-Gaussian likelihood processes decompose into a large number of weak local contributions. Under suitable scaling, their effect on posterior observables is captured by the associated Gaussian process, even when the likelihood processes $X$ and $Y$ remain far apart in distribution and the induced posterior distributions remain highly non-Gaussian.

\subsection{Overview of the approach and main results}

We study posterior observables by first analyzing the free energy, the
$P_\theta$-averaged log-normalizing constant associated with the likelihood
process $X$:
\begin{equation}\label{eq:intro_free_energy}
F(\theta)
=
\E_\theta
\Bigl[
\log
\int_T e^{X(t)}\pi(\dd t)
\Bigr].
\end{equation}
The free energy is a global functional of the likelihood process, summarizing
the aggregate posterior behavior rather than the value of the likelihood at a
single parameter. The nonasymptotic comparison proceeds in two stages, each
built around a different comparison principle. We then study its consequences
along asymptotic sequences of inference problems.

\parhead{First stage: free energy comparison}
Let $Y=(Y(t))_{t\in T}$ be any Gaussian process on the same index set, with law
$Q_\theta$, and define its free energy by
\begin{equation}
\label{eq:intro_gaussian_free_energy}
F^{\mathrm G}(\theta)
=
\E_{\theta}
\Bigl[
\log
\int_T e^{Y(t)}\pi(\dd t)
\Bigr].
\end{equation}
Comparing the free energy of a non-Gaussian model with that of a Gaussian
surrogate is a recurring strategy for establishing universality in
high-dimensional inference~\cite{korada:2011,deshpande:2015a,krzakala:2016,lesieur:2017a,lelarge:2018,reeves:2019ab,guionnet2025estimating,guionnet:2025,mergny:2024,zou:2026}.
Our first main result is an explicit, nonasymptotic bound on the free energy
gap $\lvert F(\theta)-F^{\mathrm G}(\theta)\rvert$: under regularity
assumptions on $\Theta$, $T$, and on the processes $X$ and $Y$, we show that
(see Theorem~\ref{thm:fe_comparison})
\begin{equation}\label{eq:intro_fe_bound}
\begin{aligned}
\bigl|
F(\theta)-F^{\mathrm G}(\theta)
\bigr|
&\le
\bigl\|
\E_\theta[X]-\E_\theta[Y]
\bigr\|_\infty
\\
&\qquad +\bigl\|
\cov_\theta(X)-\cov_\theta(Y)
\bigr\|_\infty
+
\E_\theta
\bigl\|
X-\E_\theta[X]
\bigr\|_\infty^3,
\end{aligned}
\end{equation}
where $\|\cov_\theta(X)-\cov_\theta(Y)\|_\infty=\sup_{s,t\in T}|\cov_\theta(X(s),X(t))-\cov_\theta(Y(s),Y(t))|$.
The bound separates two sources of error: the mean and covariance terms
measure how well $Y$ approximates $X$ in its first two moments, while the
third-moment term quantifies the non-Gaussianity of~$X$. When $X$ is itself
Gaussian, the third-moment term can be removed, and the bound reduces to a
comparison of the first two moments alone. The proof builds on a Gaussian
interpolation in the spirit of Stein's
method~\cite{guerra:2003,talagrand:2011,rollin:2013}, which controls the
interpolated derivative via integration by parts, exact on the Gaussian side
and approximate on the non-Gaussian side. The main technical novelty is to
carry out this argument for infinite-dimensional likelihood processes, using
the Fr\'{e}chet differentiability of the log-partition map.

To illustrate the advantage of this function-space approach, note that in
many of the high-dimensional inference problems mentioned above, the
log-likelihood process $X$ and its Gaussian surrogate $Y$ take the form
\[
X(t)=\langle W,\eta(t)\rangle+R(t),
\qquad
Y(t)=\langle W^{\mathrm G},\eta(t)\rangle,
\qquad t\in T,
\]
where $W\in\bbR^d$ is a data vector, $\eta:T\to\bbR^d$ is a feature map,
$W^{\mathrm G}\in\bbR^d$ is a Gaussian vector with the same first and second
moments as $W$, and $R$ is a remainder process.
Existing comparison arguments operate in the finite-dimensional space of~$W$
and therefore cannot accommodate the remainder~$R$ directly. A common
workaround is to introduce an intermediate process
$X'(t)=\langle W,\eta(t)\rangle$, compare $X'$ with~$Y$ via
finite-dimensional interpolation, and bound the free energy gap between~$X$
and~$X'$ separately through the Lipschitz continuity of the log-partition map.
This second step requires conditions such as
\[
\E_\theta\sup_{t\in T}|R(t)|\le\varepsilon,
\]
together with uniform
boundedness~\cite{krzakala:2016,lesieur:2017a,lelarge:2018,guionnet2025estimating,guionnet:2025,mergny:2024}
or Bernstein-type
conditions~\cite{deshpande:2015a,reeves:2019ab,zou:2026} on~$R$.
By contrast, applying~\eqref{eq:intro_fe_bound} directly to the full
process~$X$ controls the remainder~$R$ only through its contribution to the
mean, covariance, and third-moment terms:
\[
\sup_{t\in T}|\E_\theta R(t)|,
\qquad
\sup_{t\in T}\E_\theta|R(t)|^2,
\qquad
\E_\theta\sup_{t\in T}|R(t)|^3.
\]
Since~$R$ is typically small in a suitable scaling, all three terms are
small, requiring only the corresponding moment conditions. This provides
a unified framework that accommodates a broader class of inference problems
without requiring finite-dimensional decompositions or strong tail assumptions
on the remainder.

\parhead{Second stage: comparing posterior observables}
The first-stage comparison of $F(\theta)$ and $F^{\mathrm G}(\theta)$ does not
directly extend to posterior mean squared error or posterior variance, since
these quantities depend on the posterior law rather than only its normalizing
constant. We obtain the corresponding comparison through a perturbation
argument. Specifically, let $F(\theta,\beta)$ be a scalar perturbation of the
non-Gaussian free energy. On the Gaussian side, let $Y_\beta$ be a perturbed
Gaussian process with $Y_0=Y$, and denote its free energy by
$F^{\mathrm G}(\theta,\beta)$. Both perturbations are indexed by $\beta$ near
zero, with $F(\theta,0)=F(\theta)$ and
$F^{\mathrm G}(\theta,0)=F^{\mathrm G}(\theta)$.

Let $\Phi(\theta)$ denote either $\cR(\theta)$ or $\cV(\theta)$ in
\eqref{eq:intro_mse_variance}. We define
$\Phi^{\mathrm G}(\theta,\beta)$ by replacing $X$ with $Y_\beta$ in the
corresponding posterior functional, and write
$\Phi^{\mathrm G}(\theta)=\Phi^{\mathrm G}(\theta,0)$. The perturbations are
chosen so that, schematically,
\[
\partial_\beta F(\theta,0)
=
\Phi(\theta)+\delta(\theta),
\qquad
\partial_\beta F^{\mathrm G}(\theta,\beta)
=
\Phi^{\mathrm G}(\theta,\beta)+\delta^{\mathrm G}(\theta,\beta),
\]
where the remainders are explicit and
$\delta^{\mathrm G}(\theta,0)=0$.

The key step to recover posterior observables from the free energy comparison
is an explicit, nonasymptotic sandwich bound.
By convexity of $\beta\mapsto F(\theta,\beta)$ and
$\beta\mapsto F^{\mathrm G}(\theta,\beta)$, the derivative at zero is bounded
by difference quotients evaluated at $\pm\eps$. Combined with the first-stage
free energy comparison, this gives
\begin{align}\label{eq:intro_sandwich}
\Phi^{\mathrm G}(\theta,-\eps)-\Delta(\theta,\eps)
\;\;\lesssim\;\;
\Phi(\theta)
\;\;\lesssim\;\;
\Phi^{\mathrm G}(\theta,\eps)+\Delta(\theta,\eps),
\end{align}
where $\Delta(\theta,\eps)$ is an explicit error term that combines the free
energy comparison error of the first stage with the derivative approximation errors.
Optimizing over $\eps$ yields a quantitative bound on the gap between the
posterior observable and its Gaussian counterpart.

\parhead{Asymptotic consequences}
We consider a sequence of inference problems indexed by a dimension parameter
$n\to\infty$, with
$\theta=(\theta_1,\ldots,\theta_n)\in\Theta_n$. The parameter and candidate
spaces, prior, data-generating laws, likelihood processes, and their
perturbations may all depend on $n$. We use the normalized quantities
\begin{align}
F_n(\theta,\beta)
&\coloneqq
\frac1n F(\theta,\beta),
&
F_n^{\mathrm G}(\theta,\beta)
&\coloneqq
\frac1n F^{\mathrm G}(\theta,\beta),
\label{eq:normalized_free_energies}
\\
\Phi_n(\theta)
&\coloneqq
\frac1n\Phi(\theta),
&
\Phi_n^{\mathrm G}(\theta,\beta)
&\coloneqq
\frac1n\Phi^{\mathrm G}(\theta,\beta).
\label{eq:normalized_posterior_observables}
\end{align}
The unperturbed quantities correspond to $\beta=0$.

For the asymptotic theory, we specialize $Y$ to the concrete Gaussian location
family of Definition~\ref{def:gaussian_loc}. We identify structural and moment
conditions on $X$ under which its local first- and second-order behavior
determines a corresponding Gaussian location process $Y$ and the normalized
comparison errors satisfy
\[
\bigl|F_n(\theta)-F_n^{\mathrm G}(\theta)\bigr|\longrightarrow0,
\qquad
\Delta_n(\theta,\eps)\longrightarrow \eps,
\]
where $\Delta_n(\theta,\eps)$ denotes the normalized error in the
sandwich~\eqref{eq:intro_sandwich}. The first limit transfers any limiting
Gaussian free energy to the non-Gaussian model.


Beyond these qualitative conclusions, our bounds explicitly quantify the free
energy discrepancy and the sandwich error $\Delta_n(\theta,\eps)$, and hence
yield convergence rates whenever the constituent errors vanish at controlled
rates. Suppose, in addition, that
$\cF^{\mathrm G}(\theta,\beta)\coloneqq
\lim_{n\to\infty}F_n^{\mathrm G}(\theta,\beta)$ exists for $\beta$ in a
neighborhood of zero and is differentiable at zero. The
sandwich~\eqref{eq:intro_sandwich} then gives, 
\[
\lim_{n\to\infty}\Phi_n(\theta)
=
\lim_{n\to\infty}\Phi_n^{\mathrm G}(\theta)
=
\partial_\beta \cF^{\mathrm G}(\theta,0).
\]
Thus, asymptotic universality of the posterior observables reduces to
identifying the limiting Gaussian free energy and its derivative at the origin.
The finite-sample sandwich bound~\eqref{eq:intro_sandwich}, however, requires
neither the existence nor the differentiability of this limit. If the limiting
Gaussian free energy exists but is not differentiable at the origin, its
one-sided derivatives still provide asymptotic lower and upper bounds. Such
nondifferentiability is characteristic of a phase transition and may manifest
itself as a jump in the posterior mean squared error or posterior variance.

\subsection{Relation to prior work}


Our approach draws on ideas from statistical physics and information theory,
particularly the study of Gibbs measures, free energies, and
high-dimensional Bayesian
inference~\cite{talagrand:2003,mezard:2009}. Closely related comparison
principles have also been developed in statistics and machine learning. We
review the literature most relevant to our results.
\parhead{Universality in high-dimensional inference}
Perturbative identities relating derivatives of log-partition functionals to
observables such as energy, overlap, and mean squared error play a central role
in the analysis of Gibbs measures and Bayesian inference. A substantial body
of work has used related ideas to study universality in high-dimensional
inference~\cite{korada:2010,korada:2011,deshpande:2015a,krzakala:2016,lesieur:2017a,lelarge:2018,reeves:2019ab,guionnet2025estimating,guionnet:2025,mergny:2024}.
A recurring principle is that sufficiently regular non-Gaussian observation
models share their asymptotic behavior with Gaussian models having the same
local second-order structure, suggesting asymptotically equivalent Gaussian
descriptions at the level of the log-partition functional.

Our results differ from this earlier work in several respects. The comparison
framework is formulated directly in terms of posterior-generating likelihood
processes, separating the comparison principle from model-specific structure.
The results are nonasymptotic and apply on infinite-dimensional function
spaces, with the approximation error controlled by explicit third-moment
conditions. The framework further converts comparison of log-partition
functionals into comparison of posterior observables, including posterior mean
squared error and posterior variance. While related perturbative arguments
have been developed for particular models, our framework provides a unified
route from Gaussian comparison of free energies to inference on posterior
observables.

\parhead{Comparison principles and high-dimensional Gaussian approximation}
In the high-dimensional regime considered here, the parameter dimension
grows with the sample size, and Gaussian approximation is often formulated
by comparing $\E[f(X)]$ and $\E[f(Y)]$ for selected test functionals~$f$,
rather than by approximating the full distributions of~$X$ and~$Y$.
One standard approach is the Lindeberg replacement method, which exploits
independence to replace components successively
\cite{chatterjee:2006,korada:2011}. In random matrix theory, entrywise
replacement underlies the Four Moment Theorem and Green-function comparison
methods for local spectral universality
\cite{tao:2011,erdos-yau-yin:2012}.
A second approach is Stein's method, which uses couplings or interpolation
identities to compare expectations of smooth test functions
\cite{chatterjee-meckes:2008,reinert-rollin:2009,rollin:2013}. A prominent
high-dimensional implementation is the work of Chernozhukov, Chetverikov,
and Kato
\cite{chernozhukov:2013,chernozhukov:2017}, where a closely related
Stein-interpolation argument is combined with smoothing and
anti-concentration to approximate maxima and probabilities of
high-dimensional sets.

Our proof combines these two approaches and carries out the comparison
directly in $C(T)$, the Banach space of real-valued continuous functions on
the compact metric space~$T$, equipped with the supremum norm.
\parhead{Gaussian equivalence in machine learning}
Our emphasis on posterior observables is also related to recent Gaussian
equivalence and universality results in high-dimensional statistics and
machine learning~\cite{dobriban:2018,hastie:2022,goldt:2022,hu:2023,
montanari:2022,han:2023}, including work on random feature models, neural
networks, and high-dimensional empirical risk minimization. These works
primarily concern estimators and prediction functionals, whereas our
comparison acts on the likelihood process and the posterior quantities it
induces. Methodologically, these works are related to ours through their use
of Lindeberg-type replacement or interpolation arguments and, in some cases,
Stein-type Gaussian approximation~\cite{hu:2023,montanari:2022,han:2023}.


\parhead{Beyond posterior contraction}
This work also differs from classical Bayesian asymptotics, which
primarily characterizes posterior concentration around the true parameter
through contraction rates~\cite{ghosal:2017,ghosal:2000convergence} and Bernstein--von
Mises phenomena, including extensions to infinite-dimensional
models~\cite{castillo:2013,castillo:2014} to misspecified
models~\cite{kleijn:2012}, and to generalized
posteriors~\cite{miller:2021}. In the high-dimensional regimes considered here,
the posterior need not concentrate at a single point, and posterior
uncertainty may remain macroscopic asymptotically. We instead characterize
this residual uncertainty through Gaussian comparison of posterior
observables. 

A stronger direction is to pass from posterior observables to
the asymptotic posterior distribution itself.
Universality at this level has been studied in
high-dimensional spin-glass models and relies on additional symmetry and
concentration properties~\cite{carmona:2006,panchenko:2013book,chen:2019spin}. Related
work transfers this perspective to high-dimensional statistical inference
through multioverlap concentration~\cite{barbier:2022strong,
barbier:2022logconcave}. We do not pursue this stronger direction here.

\subsection{Organization}

Section~\ref{sec:general_frame} introduces the general framework, including
likelihood processes, posterior observables, and their Gaussian location
counterparts. Section~\ref{sec:comparison} develops the nonasymptotic
comparison principles for free energies and posterior observables through a
perturbation argument. Section~\ref{sec:universality} applies these principles
to high-dimensional product models, first for exponential-family
log-likelihoods and then for the broader class of regular local
log-likelihoods. Finally, Section~\ref{sec:app_sparse_bernoulli} specializes
the theory to sparse Bernoulli hypergraph models.



\section{Problem Formulation}~\label{sec:general_frame}

Let $\Theta$ be a compact parameter space, and let $T\subseteq\Theta$ be a
compact metric space equipped with a prior distribution $\pi$. The space
$\cX\coloneqq C(T)$, endowed with the norm
$\|x\|_\infty\coloneqq\sup_{t\in T}|x(t)|$, is a separable Banach space. We
regard the likelihood process $X=(X(t))_{t\in T}$ as a $\cX$-valued random
variable on a common sample space $\Omega$, and write
$X(t,\omega)=X(\omega)(t)$. Compactness of $T$ and continuity of the sample
paths imply
\[
0
<
\int_T \exp\{X(t)\}\pi(\dd t)
\le
\exp\{\|X\|_\infty\}
<
\infty.
\]
Thus the posterior distribution in~\eqref{eq:intro_posterior} is well defined,
and the logarithm appearing in the free energy~\eqref{eq:intro_free_energy} is
finite for every sample path;
finiteness of the free energy itself is discussed in
Theorem~\ref{thm:fe_comparison}.
All Gaussian processes $Y$ considered below are
also taken to have continuous sample paths. Their posterior normalizing
constants and the logarithms defining the Gaussian free energies
in~\eqref{eq:intro_gaussian_free_energy} are therefore finite pathwise.

Let $(P_\theta)_{\theta\in\Theta}$ denote the data-generating family on
$\Omega$, and write $\E_{P_\theta}$ for expectation under $P_\theta$. The
statistical model represented by $X$ need not coincide with this family,
allowing for model misspecification. Under correct specification, $T=\Theta$
and there exists a reference probability measure $\mu$ such that
\[
X(t,\omega)
=
\log\frac{\dd P_t}{\dd\mu}(\omega),
\qquad
t\in T.
\]

\subsection{Feature Maps and Posterior Observables}

Throughout, a feature map is understood to be a continuous map
$\eta:\Theta\to\cH$, where $\cH$ is a separable real Hilbert space. This representation provides a geometric embedding of the candidate space and allows one to quantify posterior uncertainty and estimation error. Write the posterior mean estimator associated with $\eta$ as
\[
\widehat\eta_X
\coloneqq
\int_T \eta(t)\,\pi_X(\dd t).
\]

The posterior variance $\cV(\theta)$ and posterior mean squared error
$\cR(\theta)$ defined in~\eqref{eq:intro_mse_variance} admit the following
representations under the product posterior measure. For each
$\theta\in\Theta$,
\begin{equation}\label{eq:pv}
\begin{aligned}
\cV(\theta)
&=
\iint_{T^2}
\Big[
\|\eta(t)\|_\cH^2
-
\langle \eta(t),\eta(s)\rangle_\cH
\Big]
\,\E_{P_\theta}[\pi_X^{\otimes2}]
(\dd t,\dd s),
\end{aligned}
\end{equation}
and
\begin{equation}\label{eq:mse}
\begin{aligned}
\cR(\theta)
&=
\iint_{T^2}
\Big[
\|\eta(\theta)\|_\cH^2
+
\langle \eta(t),\eta(s)\rangle_\cH
-
2\langle \eta(\theta),\eta(t)\rangle_\cH
\Big]
\,\E_{P_\theta}[\pi_X^{\otimes2}]
(\dd t,\dd s).
\end{aligned}
\end{equation}
More generally, let $(\phi_\theta)_{\theta\in\Theta}$ be a family of test functions with $\phi_\theta:T^2\to\bbR$. We define the associated \emph{posterior observables} by
\begin{align}\label{eq:Phi_X}
\Phi(\theta)
\coloneqq
\iint_{T^2}
\phi_\theta(t,s)\,
\E_{P_\theta}[\pi_X^{\otimes2}]
(\dd t,\dd s).
\end{align}

\subsection{Gaussian Log-Likelihood Processes}

Let $(Q_\theta)_{\theta\in\Theta}$ be a family of probability measures on $\Omega$, and let $Y=(Y(t))_{t\in T}$ be a Gaussian random variable taking values in $\cX=C(T)$. For each $\theta\in\Theta$, the law of $Y$ is induced by $Q_\theta$. The Gaussian posterior distribution induced by $Y$ is
\begin{equation}\label{eq:gaussian_posterior}
\pi_Y(\dd t)
=
\frac{\exp\{Y(t)\}\pi(\dd t)}
{\int_T \exp\{Y(s)\}\pi(\dd s)}.
\end{equation}
For a family of test functions $(\phi_\theta)_{\theta\in\Theta}$ with $\phi_\theta:T^2\to\bbR$, we define the Gaussian posterior observables by
\begin{align}\label{eq:Phi_Y}
\Phi^{\mathrm G}(\theta)
\coloneqq
\int_{T^2}
\phi_\theta(t,s)\,
\E_{Q_\theta}[\pi_Y^{\otimes2}]
(\dd t,\dd s).
\end{align}

We focus on a class of Gaussian log-likelihood processes associated with a given feature map, which are directly related to the posterior observables of interest.
For a given feature map $\eta:\Theta\to\cH$ into a real separable Hilbert space $\cH$, we construct a Gaussian process $Y=(Y(t))_{t\in T}$ of the form
\begin{align}\label{eq:Y_eta}
	Y(t,\omega) = \langle \xi(\omega),\eta(t) \rangle + m_\theta(t),\qquad \xi\sim \normal(0,\Sigma_\theta).
\end{align}
Here $\Sigma_\theta:\cH\to\cH$ is a covariance operator, and $m_\theta:T\to\bbR$ is a deterministic mean function. Consequently,
\begin{align}\label{eq:Y_eta1}
\E_{Q_\theta}[Y(t)]
=
m_\theta(t),
\qquad
\cov_{Q_\theta}(Y(s),Y(t))
=
\langle \Sigma_\theta\,\eta(s),\eta(t)\rangle_{\cH}, \qquad s,t\in T.
\end{align}
The practical significance of the above construction becomes apparent in the Gaussian location setting, which motivates the following family.

\begin{definition}[Gaussian location family]~\label{def:gaussian_loc}
	A Gaussian location family is specified by a bounded linear operator $A:\cH\to\cH$, a self-adjoint bounded operator $K:\cH\to\cH$, and a covariance operator $\Sigma:\cH\to\cH$. The associated Gaussian log-likelihood process is the Gaussian process $Y=(Y(t))_{t\in T}$ whose mean function and covariance kernel are given by
		\begin{align*}
		\E_{Q_\theta}[Y(t)]
		&=
		\langle A\,\eta(\theta),\eta(t)\rangle_{\cH}
		-\frac12
		\langle K\,\eta(t),\eta(t)\rangle_{\cH},
		\\
		\cov_{Q_\theta}(Y(s),Y(t))
		&=
		\langle \Sigma\,\eta(s),\eta(t)\rangle_{\cH}.
		\end{align*}
	Equivalently, there exists $\xi\sim N(0,\Sigma)$ such that one may write
	\begin{equation}\label{eq:gaussian_location_xi_representation}
	Y(t)
	=
	\langle \xi,\eta(t)\rangle_{\cH}
	+
	\langle A\,\eta(\theta),\eta(t)\rangle_{\cH}
	-\frac12
	\langle K\,\eta(t),\eta(t)\rangle_{\cH}.
	\end{equation}
\end{definition}

The term \emph{location} refers to the fact that $A$, $\Sigma$, and $K$ are
independent of $\theta$. The preceding definition can be realized by a Gaussian
observation model. Under the data-generating law $Q_\theta$, suppose that the observation is
\[
Z(\omega)
=
A\eta(\theta)+\xi(\omega),
\qquad
\xi\sim\normal(0,\Sigma),
\]
while the statistical model is given, possibly under misspecification, by
\[
\{\normal(K\eta(t),K):t\in T\}.
\]
This interpretation requires that $K$ be a covariance operator and that the
compatibility conditions stated in
Appendix~\ref{app:gaussian_loglik_derivation} hold; the construction is
justified there via the Cameron--Martin formula. The case $A=K=\Sigma$
corresponds to the well-specified setting.

\section{Gaussian Comparison via Perturbation}~\label{sec:comparison}

This section develops the nonasymptotic comparison principles for free
energies and posterior observables outlined in the introduction. We first
introduce perturbations that represent posterior observables as derivatives of
perturbed free energies. We then establish a free energy comparison theorem
and combine it with convexity to obtain explicit sandwich bounds between these
observables and their Gaussian counterparts.


\subsection{Perturbation Mechanism}\label{subsec:perturbation_mechanism}

Let $\cB\subset\bbR$ be an open interval containing $0$, and let
$(X_\beta)_{\beta\in \cB}$ be a family of $\cX$-valued random variables on
$\Omega$ such that $X_0=X$.
We assume that, for every $\omega\in\Omega$, the map
$\beta\mapsto X_\beta(\omega)$ is differentiable as an $\cX$-valued map, and
write
\[
\dot X_h(t)
\coloneqq
\frac{\partial}{\partial\beta}X_\beta(t)\big|_{\beta=h},
\qquad
h\in \cB.
\]
Thus $\dot X_h$ is again a $\cX$-valued random variable for every
$h\in\cB$. We refer to $(X_\beta)_{\beta\in \cB}$ as a
\emph{perturbation} of $X$, and define the corresponding perturbed free
energy by
\[
F(\theta,\beta)
\coloneqq
\E_{P_\theta}
\left[
\log
\int_T e^{X_\beta(t)}\,\pi(\dd t)
\right].
\]
By construction, $F(\theta,0)=F(\theta)$. Similarly, let
$(Y_\beta)_{\beta\in \cB}$ be a perturbation of $Y$ in the sense above,
with $Y_0=Y$. We additionally require that, for every $\beta\in\cB$,
$Y_\beta$ is Gaussian and $(Y_\beta,\dot Y_\beta)$ is jointly Gaussian.
Define the corresponding perturbed Gaussian free energy by
\begin{equation}\label{eq:perturbed_gaussian_free_energy}
F^{\mathrm G}(\theta,\beta)
\coloneqq
\E_{Q_\theta}
\left[
\log\int_T \exp\{Y_\beta(t)\}\,\pi(\dd t)
\right].
\end{equation}

The derivatives of perturbed free energies are closely related to the
Gibbs product measures generated by the perturbed processes. To make this
relation explicit, for a perturbation $(X_\beta)_{\beta\in \cB}$ of $X$ and
$s,t\in T$, we define
\begin{align}\label{eq:psi_X}
\psi_{\theta,\beta}(s,t)
=
\E_{P_\theta}[\dot X_\beta(t)]
+
\cov_{P_\theta}(X_\beta(t),\dot X_\beta(t))
-
\cov_{P_\theta}(X_\beta(s),\dot X_\beta(t))
.
\end{align}
Similarly, for a perturbation $(Y_\beta)_{\beta\in \cB}$ of $Y$ and
$s,t\in T$, we define
\begin{align}\label{eq:psi_Y}
	\psi^{\mathrm G}_{\theta,\beta}(s,t)
	=
	\E_{Q_\theta}[\dot Y_\beta(t)]
	+
	\cov_{Q_\theta}(Y_\beta(t),\dot Y_\beta(t))
	-
	\cov_{Q_\theta}(Y_\beta(s),\dot Y_\beta(t))
	.
\end{align}
The next two lemmas express the free energy derivatives in terms of these
functions.

\begin{lemma}\label{lem:deri_FG}
	Suppose that
	$\E_{Q_\theta}[\sup_{\beta\in\cB}\|\dot Y_\beta\|_\infty]<\infty$.
	Then for all $\theta\in\Theta$ and $\beta\in\cB$,
	\begin{equation}\label{eq:gaussian_free_energy_derivative}
	\partial_\beta F^{\mathrm G}(\theta,\beta)
	=
	\iint_{T^2}
	\psi^{\mathrm G}_{\theta,\beta}(s,t)
	\,\E_{Q_\theta}[\pi_{Y_\beta}^{\otimes2}]
	(\dd s,\dd t).
	\end{equation}
\end{lemma}

\begin{proof}
The proof, based on Gaussian integration by parts on the separable Banach
space $\cX=C(T)$, is given in Appendix~\ref{app:proof_deri_FG}.
\end{proof}

\begin{lemma}\label{lem:deri_F}
	Suppose that $X_\beta=\sum_{m=1}^M X_{\beta,m}$
	where $\{(X_{\beta,m},\dot X_{\beta,m})\}_{m=1}^M$ are independent
	under $P_\theta$, that
	$\E_{P_\theta}[\sup_{\beta\in\cB}\|\dot X_\beta\|_\infty]<\infty$,
	and that for every $m\in[M]$ and $\beta\in\cB$,
	$\E_{P_\theta}\|X_{\beta,m}\|_\infty^3<\infty$ and
	$\E_{P_\theta}\|\dot X_{\beta,m}\|_\infty^3<\infty$.
	Then for all $\theta\in\Theta$ and $\beta\in\cB$,
	\begin{align}
		\partial_\beta F(\theta,\beta)
		&=
		\iint_{T^2} 	\psi_{\theta,\beta}(s,t) \,
		\E_{P_\theta}[\pi_{X_\beta}^{\otimes 2}](\dd s,\dd t)
		+
		\sum_{m=1}^M R_{m,\beta},
	\end{align}
	where for each $m\in[M]$, $R_{m,\beta}$ is a remainder process with,
	\[
	|R_{m,\beta}|
	\le
	3
	\Bigl(
	\E_{P_\theta}
	\|X_{\beta,m}-\E_\theta[X_{\beta,m}]\|_\infty^3
	\Bigr)^{2/3}
	\Bigl(
	\E_{P_\theta}
	\|\dot X_{m,\beta}-\E_{P_\theta}[\dot X_{m,\beta}]\|_\infty^3
	\Bigr)^{1/3}.
	\]
\end{lemma}

\begin{proof}
The proof, based on an approximate integration-by-parts argument on the separable Banach
space $\cX=C(T)$, is given in Appendix~\ref{app:proof_deri_F}.
\end{proof}

The independent-component decomposition in Lemma~\ref{lem:deri_F}
anticipates the high-dimensional product structures considered in
Section~\ref{sec:universality}: it produces a sum of local remainders that can
vanish after normalization. The statement also includes the case $M=1$. We
next use a Gaussian linear perturbation to illustrate how these derivative
identities generate the posterior mean squared error and posterior variance.

\parhead{Gaussian linear perturbations generating posterior observables}
Fix a feature map $\eta:\Theta\to\cH$, and let $Y$ be the Gaussian process
specified by~\eqref{eq:Y_eta}--\eqref{eq:Y_eta1}. We consider the following
perturbation of $Y$:
\begin{align}\label{eq:Y_beta}
	Y_{\beta}(t) \coloneqq Y(t) + \beta\,\widetilde{Y}(t),
\end{align}
where $\widetilde{Y} = (\widetilde{Y}(t))_{t\in T}$ is a $\cX$-valued Gaussian process on $\Omega$.
To specify $\widetilde{Y}$, let $H_\theta:\cH\to\cH$ be a self-adjoint bounded linear operator satisfying
\begin{align}\label{eq:H_Sigma}
	\operatorname{Range}(H_\theta)
	\subseteq
	\operatorname{Range}(\Sigma_\theta).
\end{align}
By the Douglas factorization theorem, there exists a bounded linear operator $B_\theta:\cH\to\cH$ such that $H_\theta = \Sigma_\theta B_\theta$. In finite dimensions one may take $B_\theta = \Sigma_\theta^{+}H_\theta$, where $\Sigma_\theta^{+}$ denotes the Moore--Penrose pseudoinverse. For $a \in \{0,-1\}$, define
\begin{align}\label{eq:Y_tilde}
\widetilde{Y}(t)
\coloneqq
\langle \xi(\omega),\,B_\theta\, \eta(t) \rangle
+
a \langle H_\theta[\eta(\theta) - \eta(t)],\, \eta(\theta) - \eta(t) \rangle,
\end{align}
where $\xi$ is the same as in~\eqref{eq:Y_eta}. We further specify a test function related to such perturbation,
\begin{align}\label{eq:phi_theta_lin}
	\phi_\theta(s,t)
	=
	a\,\langle H_\theta[\eta(\theta)-\eta(t)],\,\eta(\theta)-\eta(t)\rangle
	+
	\langle H_\theta\eta(t),\eta(t)\rangle
	-
	\langle H_\theta\eta(s),\eta(t)\rangle.
\end{align}

\begin{lemma}\label{lem:linear_perturb_G}
Consider the Gaussian perturbation $(Y_\beta)_{\beta\in\cB}$ defined
by~\eqref{eq:Y_beta}--\eqref{eq:Y_tilde}, with perturbed free energy
$F^{\mathrm G}(\theta,\beta)$ defined in
\eqref{eq:perturbed_gaussian_free_energy}. For the test function $\phi_\theta$
in~\eqref{eq:phi_theta_lin}, identity~\eqref{eq:gaussian_free_energy_derivative}
holds with
\[
\psi_{\theta,\beta}^{\mathrm G}(s,t)
=
\phi_\theta(s,t)
+
\beta\,\langle B_\theta^*\Sigma_\theta B_\theta[\eta(t)-\eta(s)],\eta(t)\rangle.
\]
In particular, at $\beta=0$,
\[
\partial_\beta F^{\mathrm G}(\theta,0)
=
\iint_{T^2} \phi_\theta(s,t)\,\E_{Q_\theta}[\pi_Y^{\otimes2}](\dd s,\dd t).
\]
\end{lemma}

\begin{proof}
See Appendix~\ref{app:proof_linear_perturb_G}.
\end{proof}

At $\beta=0$, Lemma~\ref{lem:linear_perturb_G} identifies the free energy
derivative with the $H_\theta$-weighted Gaussian posterior mean squared error
and posterior variance. Let
$\widehat\eta_Y\coloneqq\int_T\eta(t)\,\pi_Y(\dd t)$. For $a=-1$,
\[
-\partial_\beta F^{\mathrm G}(\theta,0)
=
\E_{Q_\theta} \bigl[ \langle H_\theta[\eta(\theta) - \widehat\eta_Y],\, \eta(\theta) - \widehat\eta_Y \rangle \bigr],
\]
whereas for $a=0$,
\[
\partial_\beta F^{\mathrm G}(\theta,0)
=
\E_{Q_\theta} \Bigl[\, \int_T \langle H_\theta[\eta(t) - \widehat\eta_Y],\, \eta(t) - \widehat\eta_Y \rangle \,\pi_Y(\dd t) \,\Bigr].
\]

When $H_\theta=\Id$, these expressions are the Gaussian counterparts of
$\cR(\theta)$ and $\cV(\theta)$ in~\eqref{eq:mse} and~\eqref{eq:pv},
respectively. Thus $\phi_\theta$ in~\eqref{eq:phi_theta_lin} encodes both
observables: $a$ selects the observable, while $H_\theta$ specifies its
direction. This class of test functions is the main target of the posterior
comparison below.

An analogous perturbation can be constructed for a non-Gaussian likelihood
process $X$. Its precise form, however, depends on the local structure of
$X$. We therefore defer the construction and the control of its approximation
errors to Section~\ref{sec:universality}, where the relevant structure is
specified.

\subsection{Comparison Theorems}
We now establish nonasymptotic Gaussian comparison bounds for free energies
and posterior observables. Here $X$ denotes the possibly non-Gaussian
likelihood process under $P_\theta$, whereas $Y$ denotes the Gaussian
log-likelihood process under $Q_\theta$. To streamline notation, $\E_\theta$
and $\cov_\theta$ are taken under $P_\theta$ for expressions involving $X$ and under $Q_\theta$ for expressions involving $Y$. 

\begin{theorem}[Free Energy Comparison]\label{thm:fe_comparison}
	Suppose that $X=\sum_{m=1}^M X_m$, where, for every
	$\theta\in\Theta$, the processes $X_1,\ldots,X_M$ are independent under
	$P_\theta$. Then, for every $\theta\in\Theta$,
		\begin{align*}
		\bigl|
		F(\theta)-F^{\mathrm G}(\theta)
		\bigr|
		&\le
		\bigl\|
		\E_\theta[X]-\E_\theta[Y]
		\bigr\|_\infty
		\\
		&\quad+
		\bigl\|
		\cov_\theta(X)-\cov_\theta(Y)
	\bigr\|_\infty
	+
	\sum_{m=1}^M
	\E_\theta
	\bigl\|
		X_m-\E_\theta[X_m]
		\bigr\|_\infty^3,
		\end{align*}
	where $\|\cov_\theta(X)-\cov_\theta(Y)\|_\infty=\sup_{s,t\in T}|\cov_\theta(X(s),X(t))-\cov_\theta(Y(s),Y(t))|$.
\end{theorem}

\begin{proof}[Proof sketch]
The proof interpolates between $X$ and $Y$ via a Guerra path and
exploits the Fr\'{e}chet differentiability of the log-partition map
on $C(T)$.
Approximate integration by parts, which is exact for the Gaussian component
and approximate for the non-Gaussian one, controls the interpolated derivative
by the mean, covariance, and third-moment terms.
Integrating along the path yields the result; see
Appendix~\ref{app:proof_fe_comparison}.
\end{proof}


\begin{remark}
The independence assumption in Theorem~\ref{thm:fe_comparison} anticipates
the high-dimensional product models of Section~\ref{sec:universality}.
The theorem may be viewed as a $C(T)$-valued extension of the
independent-coordinate comparison in
\cite[Section~3.2, Theorem~3.1]{rollin:2013}, specialized to the
log-partition functional. Since the Stein-coupling framework
of~\cite{rollin:2013} also accommodates dependent structures, we expect
analogous comparisons to hold for dependent $C(T)$-valued processes under
suitable coupling conditions.
\end{remark}

We now compare the posterior observables $\Phi(\theta)$ and
$\Phi^{\mathrm G}(\theta)$ in~\eqref{eq:Phi_X} and~\eqref{eq:Phi_Y} associated
with a target test function $\phi_\theta:T^2\to\bbR$. The argument combines
the perturbation mechanism of Section~\ref{subsec:perturbation_mechanism} with
the free energy comparison in Theorem~\ref{thm:fe_comparison} and proceeds in
two steps.
\begin{enumerate}
	\item \emph{Choice of perturbation.} Identify perturbations of $X$ and $Y$
	whose induced test functions $\psi_{\theta,\beta}$ and
	$\psi_{\theta,\beta}^{\mathrm G}$ are close to $\phi_\theta$ for small
	$\beta$. The Gaussian linear perturbation in
	Lemma~\ref{lem:linear_perturb_G} gives an explicit construction satisfying
	$\psi_{\theta,0}^{\mathrm G}=\phi_\theta$.
	\item \emph{Convexity argument.} Lemmas~\ref{lem:deri_FG}
	and~\ref{lem:deri_F} express the derivatives of the perturbed free energies
	through these induced test functions. Convexity then converts the free
	energy comparison into a comparison of the derivatives, and hence of the
	corresponding posterior observables.
\end{enumerate}
To carry out the second step, we define the perturbed Gaussian functional
\begin{equation}\label{eq:Phi_G_beta}
\Phi^{\mathrm G}(\theta,\beta)
\coloneqq
\int_{T^2}
\phi_\theta(t,s)\,
\E_{\theta}[\pi_{Y_\beta}^{\otimes2}](\dd t,\dd s),
\qquad \beta\in\cB.
\end{equation}
In addition, for any function $\phi:T^2\to\bbR$, we define
$\|\phi\|_\infty
\coloneqq
\sup_{(s,t)\in T^2} |\phi(s,t)|$.

\begin{theorem}[Comparison of Posterior Observables]\label{thm:Phi_comparison}
	Fix a target test function $\phi_\theta:T^2\to\bbR$.
	Let $(X_\beta)_{\beta\in\cB}$ and $(Y_\beta)_{\beta\in\cB}$ be
	perturbations of $X$ and $Y$, with posterior functional
	$\Phi(\theta)$ defined in~\eqref{eq:Phi_X}, perturbed Gaussian functional
	$\Phi^{\mathrm G}(\theta,\beta)$ defined in~\eqref{eq:Phi_G_beta}, and
	induced test functions $\psi_{\theta,\beta}$ and
	$\psi_{\theta,\beta}^{\mathrm G}$ defined by~\eqref{eq:psi_X} and
	\eqref{eq:psi_Y}.
	Suppose that
	\begin{enumerate}[label=\textup{(\alph*)}]
		\item $X_\beta=\sum_{m=1}^M X_{\beta,m}$,
		where the pairs $\{(X_{\beta,m},\dot X_{\beta,m})\}_{m=1}^M$ are
		independent under $P_\theta$, and $(Y_\beta,\dot Y_\beta)$ is jointly
		Gaussian under $Q_\theta$;
		\item $\E_{P_\theta}[\sup_{\beta\in\cB}\|\dot X_\beta\|_\infty]<\infty$
		and
		$\E_{Q_\theta}[\sup_{\beta\in\cB}\|\dot Y_\beta\|_\infty]<\infty$;\vspace{0.2em}
		\item $\beta\mapsto F(\theta,\beta)$ and
		$\beta\mapsto F^{\mathrm G}(\theta,\beta)$ are convex on $\cB$.
	\end{enumerate}
	Then for all $\theta\in\Theta$ and any $\eps>0$ such that
	$\pm\eps\in\cB$,
	\begin{equation}
		\Phi^{\mathrm G}(\theta,-\eps)
		- \Delta_1(\theta,-\eps) -
		\frac{\Delta_2(\theta,-\eps)}{\eps}
		\le \Phi(\theta) \le
		\Phi^{\mathrm G}(\theta,\eps)
		+ \Delta_1(\theta,\eps) +
		\frac{\Delta_2(\theta,\eps)}{\eps}.
	\end{equation}
	where $\Delta_1(\theta,\cdot),\Delta_2(\theta,\cdot)\colon \cB \to \bbR_+$ are given by
	\begin{equation}\label{eq:posterior_comparison_delta1}
		\begin{aligned}
			\Delta_1(\theta,\beta)
			&=
			\|\phi_\theta -\psi_{\theta,0} \|_\infty
			+
			\|\phi_\theta -\psi^{\mathrm{G}}_{\theta,\beta} \|_\infty
			\\
			&\quad
			+
			3\sum_{m=1}^M
			\Bigl(
			\E_\theta\|X_{m} - \E_\theta[X_{m}]\|_\infty^3
			\Bigr)^{\frac23}
			\times
				\Bigl(
				\E_\theta\|\dot X_{0,m} - \E_\theta[\dot X_{0,m}]\|_\infty^3
				\Bigr)^{\frac13},
		\end{aligned}
	\end{equation}
	\begin{equation}\label{eq:posterior_comparison_delta2}
		\begin{aligned}
			\Delta_2(\theta,\beta)
			&=
			\bigl\|
			\E_\theta[X]-\E_\theta[Y]
			\bigr\|_\infty
			+
			\bigl\|
			\E_\theta[X_\beta]-\E_\theta[Y_\beta]
			\bigr\|_\infty
			\\
			&\quad
			+
			\bigl\|
			\cov_\theta(X)-\cov_\theta(Y)
			\bigr\|_\infty
			+
			\bigl\|
			\cov_\theta(X_\beta)-\cov_\theta(Y_\beta)
			\bigr\|_\infty
			\\
			&\quad
			+
			\sum_{m=1}^M
			\E_\theta
			\bigl\|
			X_m-\E_\theta[X_m]
			\bigr\|_\infty^3
			+
			\sum_{m=1}^M
			\E_\theta
			\bigl\|
			X_{\beta,m}
			-\E_\theta[X_{\beta,m}]
			\bigr\|_\infty^3.
		\end{aligned}
	\end{equation}
\end{theorem}

\begin{proof}
The proof is given in Appendix~\ref{app:proof_Phi_comparison}.
\end{proof}

\section{Universality in High-Dimensional Inference}\label{sec:universality}

We now apply the nonasymptotic comparison results of
Section~\ref{sec:comparison} to sequences of high-dimensional inference
problems. Our aim is to identify structural conditions under which the
normalized errors in Theorems~\ref{thm:fe_comparison}
and~\ref{thm:Phi_comparison} vanish. Under independence and suitable scaling
conditions, the normalized free energy and posterior observables are
asymptotically determined by their Gaussian counterparts.

The parameter takes the form $\theta=(\theta_1,\ldots,\theta_n)\in\Theta_n$,
where the parameter space $\Theta_n$ depends on $n$. The candidate space is
denoted by $T_n$ and may also vary with $n$. All objects introduced earlier,
including the prior $\pi$, the data-generating laws $P_\theta$ and $Q_\theta$,
the processes $X$ and $Y$, and their perturbations, are likewise allowed to
depend on $n$. Throughout, we use the normalized free energies and posterior
observables introduced in~\eqref{eq:normalized_free_energies}--\eqref{eq:normalized_posterior_observables}.

We begin with the product and feature structures used throughout the section
and the associated Gaussian comparison model. We then consider two classes of
non-Gaussian likelihood processes. Section~\ref{subsec:exp_family} uses
exponential family log-likelihoods to illustrate the comparison mechanism,
whereas Section~\ref{subsec:regular_local_log_likelihoods} develops the more
general regular local formulation.

\parhead{Product structure}
An important ingredient underlying universality in our framework is
independence across local contributions. We formalize this independence by
imposing compatible product structures on the data-generating process and the
statistical model. Recall that all randomness is defined on a common sample
space $\Omega$. For each $n$, let the common sample space and the
data-generating laws have the product form
\begin{align}\label{eq:product_data}
\Omega
=
\prod_{m=1}^{M_n}\Omega_m,
\qquad 
P_\theta(\dd\omega)
=
\bigotimes_{m=1}^{M_n}P_{\theta,m}(\dd\omega_m),
\qquad \theta\in\Theta_n.
\end{align}
Here $M_n$ is the number of local components and is allowed to grow with $n$.
Under $P_\theta$, the coordinates
$\omega=(\omega_1,\ldots,\omega_{M_n})$ are independent. Correspondingly, we
assume the log-likelihood process decomposes as
\begin{align}\label{eq:X_product}
X(t,\omega)
&=
\sum_{m=1}^{M_n}X_m(t,\omega_m),
\qquad t\in T_n .
\end{align}

Associated with each local process $X_m$ is a feature map
$\eta_m:\Theta_n\to\cH_m$, where $\cH_m$ is a real separable Hilbert space.
These local maps induce a global feature map $\eta:\Theta_n\to\cH$ that is
compatible with the preceding decomposition and is given by
\begin{align}\label{eq:product_feature_map}
\cH
\coloneqq
\bigoplus_{m=1}^{M_n}\cH_m,
\qquad 
\eta(u)
\coloneqq
\bigl(\eta_1(u),\ldots,\eta_{M_n}(u)\bigr),
\qquad u\in\Theta_n,
\end{align}
with inner product
\[
\langle h,g\rangle_\cH
=
\sum_{m=1}^{M_n}\langle h_m,g_m\rangle_{\cH_m}.
\]
Local operators are lifted in the same way: if
$L_m:\cH_m\to\cH_m$, then $L=\bigoplus_m L_m$ acts on $\cH$ and satisfies
\[
\langle L\eta(s),\eta(t)\rangle_\cH
=
\sum_{m=1}^{M_n}
\langle L_m\eta_m(s),\eta_m(t)\rangle_{\cH_m}.
\]
We refer to~\eqref{eq:product_data}, \,\eqref{eq:X_product},
and~\eqref{eq:product_feature_map} collectively as the
\emph{product structure}. 
The product feature map in~\eqref{eq:product_feature_map} fits the feature-map
framework introduced in Section~\ref{sec:general_frame}. Thus all definitions
and comparison results developed above apply to the direct-sum representation
in~\eqref{eq:product_feature_map} without modification.
For consistency with the preceding notation, the dependence of $P_\theta$,
$X$, and $\eta$ on $n$ is left implicit throughout.

\parhead{Feature radius}
We now record the feature scale induced by the product structure. This scale
will play a central role in the universality statements below. Define the
\emph{feature radius}
\begin{align}\label{eq:feature_radius}
r_n
\coloneqq
\max_{1\le m\le M_n}
\sup_{u\in\Theta_n}
\|\eta_m(u)\|_{\cH_m}.
\end{align}
Consider the test functions $\phi_\theta$ in~\eqref{eq:phi_theta_lin} with
$H_\theta=\Id$, which generate the posterior mean squared error~\eqref{eq:mse}
and posterior variance~\eqref{eq:pv} introduced in
Section~\ref{sec:general_frame}. These test functions are controlled by the
elementary bound
\[
\|\eta(u)-\eta(v)\|_{\cH}^2
\le
4M_n r_n^2,
\qquad u,v\in\Theta_n,
\]
and hence both $\Phi_n(\theta)$ and $\Phi_n^{\mathrm G}(\theta)$ have size at
most of order $M_n r_n^2/n$. Thus the natural extensive scale is
\[
r_n^2\asymp \frac{n}{M_n},
\qquad\text{equivalently}\qquad
r_n\asymp \sqrt{\frac{n}{M_n}} .
\]
On this scale, the normalized posterior observables are of order one, while
the smallness of $r_n$ is the local condition that drives the universality
errors to zero.
The use of a single maximal radius is only for notational simplicity;
heterogeneous local feature scales can be treated similarly.
 
\parhead{Gaussian comparison}
The Gaussian counterpart used in this section is the Gaussian location family
of Definition~\ref{def:gaussian_loc}. In the present product setting, we take
its operators to have the block-diagonal form
\begin{align}\label{eq:gaussian_product_operators}
A=\bigoplus_{m=1}^{M_n}A_m,
\qquad
\Sigma=\bigoplus_{m=1}^{M_n}\Sigma_m,
\qquad
K=\bigoplus_{m=1}^{M_n}K_m,
\end{align}
The associated Gaussian log-likelihood process can then be written as
\begin{align}\label{eq:Y_product}
Y(t)
=
\sum_{m=1}^{M_n}Y_m(t),
\qquad t\in T_n,
\end{align}
where $Y_1,\ldots,Y_{M_n}$ are independent Gaussian processes under
$Q_\theta$. For each $m\in[M_n]$, the local process has mean and covariance
kernel
\begin{align}\label{eq:Y_local_moments}
\E_{Q_\theta}[Y_m(t)]
&=
\langle A_m\eta_m(\theta),\eta_m(t)\rangle_{\cH_m}
-
\frac12
\langle K_m\eta_m(t),\eta_m(t)\rangle_{\cH_m},
\\
\cov_{Q_\theta}(Y_m(s),Y_m(t))
&=
\langle \Sigma_m\eta_m(s),\eta_m(t)\rangle_{\cH_m}.
\label{eq:Y_local_covariance}
\end{align}
Summing these local moments recovers the mean and covariance kernel of the
Gaussian location family associated with the operators
in~\eqref{eq:gaussian_product_operators}.

The decomposition~\eqref{eq:Y_product} also has a direct observation-model
interpretation. Each $Y_m$ may be viewed as arising from an independent
Gaussian observation
\begin{align}\label{eq:Y_local_observation}
Z_m
=
A_m\eta_m(\theta)+\xi_m,
\qquad
\xi_m\sim\normal(0,\Sigma_m),
\end{align}
paired with the local Gaussian model
$\{\normal(K_m\eta_m(t),K_m):t\in T_n\}$. This interpretation requires that
each $K_m$ be a covariance operator and that the local analogues of the
compatibility conditions stated in
Appendix~\ref{app:gaussian_loglik_derivation} hold.

\begin{remark}
The product Gaussian location model above includes the Gaussian counterparts
of many standard high-dimensional problems, including spiked matrix
models~\cite{barbier:2016a,krzakala:2016,lelarge:2018,miolane:2017a,behne:2022,guionnet2025estimating},
finite-rank matrix and tensor
models~\cite{lesieur:2017a,reeves:2020,mourrat:2020,luneau:2021,chen:2022_finite-rank},
and extensions with nonlinear feature maps or covariate
information~\cite{rossetti:2025,rossetti2025fundamental}. This breadth also
motivates the extensive scaling $r_n^2\asymp n/M_n$: it keeps the normalized
Gaussian free energy $F_n^{\mathrm G}$ and posterior observables
$\Phi_n^{\mathrm G}(\theta)$ on the nontrivial scale where phase transitions
occur. A rank-one spiked matrix model provides a concrete example. It has
$M_n\asymp n^2$ local observations and scalar feature maps
$\eta_{ij}(u)=u_i u_j/\sqrt n$, so uniformly bounded coordinates give
$r_n\asymp n^{-1/2}\asymp\sqrt{n/M_n}$, exactly the canonical scaling above.
This is also the critical scaling associated with the Baik--Ben
Arous--P\'{e}ch\'{e} transition, to which the corresponding phase transition in
the spiked matrix model is closely
related~\cite{johnstone:2001a,baik:2005,perry:2018}.
\end{remark}

\subsection{Exponential Family Log-Likelihoods}\label{subsec:exp_family}
Exponential family log-likelihoods provide a transparent illustration of the
general universality mechanism. Their cumulant structure explicitly identifies
the local first- and second-order behavior, and hence the Gaussian counterpart,
while third-order cumulants control the remaining comparison error. This class
therefore shows concretely how the process-level comparison bounds of
Section~\ref{sec:comparison} translate into conditions on the feature radius.
To keep this mechanism transparent, we work in a homogeneous product setting.
Let $\Omega_0$ be a measurable space and let $\cH_0=\bbR^d$. For each $n$ and
each $m\le M_n$, take
\[
\Omega_m=\Omega_0,\quad \Omega=\Omega_0^{M_n}\,;
\qquad
\cH_m=\cH_0,
\qquad
\cH=\bigoplus_{m=1}^{M_n}\cH_0 .
\]
Let $\lambda$ and $\gamma$ be probability measures on $\Omega_0$. Their product
measures $\lambda^{\otimes M_n}$ and $\gamma^{\otimes M_n}$ serve as the
model-side and data-generating references, respectively. Let
$W_0,W_0^\star:\Omega_0\to\cH_0$ be measurable functions and independent of $n$.

For each $m\le M_n$, let the local model log-likelihood process be given by the
exponential-family representation
\begin{align}\label{eq:expf_X}
X_m(t,\omega_m)
=
\langle W_0(\omega_m),\eta_m(t)\rangle_{\cH_0}
-
\Lambda(\eta_m(t)),
\end{align}
where $W_0$ is the sufficient statistic and $\Lambda$ is the model
log-partition function,
\begin{align}\label{eq:expf_Lambda}
\Lambda(u)
\coloneqq
\log
\int
\exp\{\langle W_0(\omega),u\rangle_{\cH_0}\}
\,\lambda(\dd\omega).
\end{align}
The local data-generating law $P_{\theta,m}$ is defined relative to $\gamma$ by
\begin{align}\label{eq:expf_P}
\log
\frac{\dd P_{\theta,m}}{\dd \gamma}(\omega_m)
=
\langle W_0^\star(\omega_m),\eta_m(\theta)\rangle_{\cH_0}
-
\Gamma(\eta_m(\theta)),
\end{align}
where $\Gamma$ is the data-generating log-partition function,
\begin{align}\label{eq:expf_Gamma}
\Gamma(u)
\coloneqq
\log
\int
\exp\{\langle W_0^\star(\omega),u\rangle_{\cH_0}\}
\,\gamma(\dd\omega).
\end{align}
To encode the joint cumulants of $W_0$ and $W_0^\star$ under $\gamma$, define
\begin{align}\label{eq:expf_Xi}
\Xi(v,u)
\coloneqq
\log
\int
\exp\{
\langle W_0(\omega),v\rangle_{\cH_0}
+
\langle W_0^\star(\omega),u\rangle_{\cH_0}
\}
\,\gamma(\dd\omega).
\end{align}
Thus $\Lambda$ and $\Gamma$ are the cumulant generating functions of the model
and data-generating sufficient statistics, respectively, whereas $\Xi$
describes their joint fluctuations under $\gamma$. The well-specified case is
obtained when $\lambda=\gamma$ and $W_0=W_0^\star$, in which case
$\Lambda=\Gamma$.

\begin{assumption}[Local exponential-family regularity]\label{assump:expfam_regular}
Let
\begin{align*}
D_\Lambda&=\{u\in\cH_0:\Lambda(u)<\infty\},
\qquad
D_\Gamma=\{u\in\cH_0:\Gamma(u)<\infty\},
\\
D_\Xi&=\{(v,u)\in\cH_0\times\cH_0:\Xi(v,u)<\infty\}.
\end{align*}
There exists $\rho>0$ such that
\[
\{u\in\cH_0:\|u\|_{\cH_0}<\rho\}
\subset D_\Lambda\cap D_\Gamma,
\qquad
\{(v,u):\|v\|_{\cH_0}<\rho,\ \|u\|_{\cH_0}<\rho\}
\subset D_\Xi .
\]
In addition,
\[
\int W_0(\omega)\,\lambda(\dd\omega)
=
\int W_0(\omega)\,\gamma(\dd\omega).
\]
\end{assumption}
The Gaussian counterpart is the Gaussian location family of
Definition~\ref{def:gaussian_loc} specialized to the product structure
above. Define the fixed local operators
\begin{align}\label{eq:expf_gaussian_operators}
A_0
&\coloneqq
\nabla_{vu}^2\Xi(0,0),
&
\Sigma_0
&\coloneqq
\nabla_{vv}^2\Xi(0,0),
&
K_0
&\coloneqq
\nabla^2\Lambda(0).
\end{align}
Its product form is obtained by setting
\[
A_m=A_0,\qquad
\Sigma_m=\Sigma_0,\qquad
K_m=K_0,
\qquad m\le M_n,
\]
in~\eqref{eq:Y_local_moments}. Here $A_0$ is the local mean response to the
data-generating tilt, $\Sigma_0$ is the reference covariance of the model
statistic, and $K_0$ is the quadratic term in the model log-partition
function. Assumption~\ref{assump:expfam_regular} ensures that the derivatives
in~\eqref{eq:expf_gaussian_operators} exist and are finite, so these operators
are well defined.
Moreover, assumption~\ref{assump:expfam_regular} provides the local expansion and
moment bounds needed to apply Theorem~\ref{thm:fe_comparison}, yielding the
following result.

\begin{theorem}[Free energy universality for exponential families]\label{thm:exp_free_energy_univ}
Consider the product structure
\eqref{eq:product_data}--\eqref{eq:product_feature_map}, with local
exponential-family likelihood processes and data-generating laws specified by
\eqref{eq:expf_X} and~\eqref{eq:expf_P}, respectively. Let $Y$ be the
associated product Gaussian comparison process
\eqref{eq:Y_product}--\eqref{eq:Y_local_covariance}, with common local
operators given by~\eqref{eq:expf_gaussian_operators}. Suppose
Assumption~\ref{assump:expfam_regular} holds, $M_n=\omega(n)$, and
$r_n\lesssim\sqrt{n/M_n}$. Then
\[
\sup_{\theta\in\Theta_n}\,
\bigl|
F_n(\theta)-F_n^{\mathrm G}(\theta)
\bigr|
=
O(r_n).
\]
\end{theorem}

\begin{proof}
See Appendix~\ref{app:proof_exp_free_energy}.
\end{proof}

\parhead{Posterior universality}
We next derive the corresponding universality result for posterior mean
squared error and posterior variance. In the present product setting, the
Gaussian location counterpart admits the exact linear perturbation of
Lemma~\ref{lem:linear_perturb_G}, while the exponential-family likelihood
admits a corresponding non-Gaussian perturbation. These constructions allow
Theorem~\ref{thm:Phi_comparison} to be applied to the free energy comparison in
Theorem~\ref{thm:exp_free_energy_univ}.

We first specialize the Gaussian perturbation
\eqref{eq:Y_beta}--\eqref{eq:Y_tilde} to the product Gaussian location family
with representation~\eqref{eq:Y_product} and local moments
\eqref{eq:Y_local_moments}. In the homogeneous exponential-family setting,
the covariance operator of $Y$ is independent of $\theta$ and is determined by
$\Sigma_0$ in~\eqref{eq:expf_gaussian_operators}.
Let $H_0$ be a symmetric $d\times d$ matrix satisfying
$\operatorname{span}(H_0)\subseteq\operatorname{span}(\Sigma_0)$, and set
\begin{align}\label{eq:posterior_product_operators}
\Sigma\coloneqq\bigoplus_{m=1}^{M_n}\Sigma_0, 
\qquad
H\coloneqq\bigoplus_{m=1}^{M_n}H_0 .
\end{align}
These global operators depend on $n$ through the number of local components;
this dependence is suppressed in the notation. By construction,
$\operatorname{span}(H)\subseteq\operatorname{span}(\Sigma)$, so
Lemma~\ref{lem:linear_perturb_G} applies with $H_\theta=H$. The resulting test
function in~\eqref{eq:phi_theta_lin} is
\begin{align}\label{eq:exp_phi_theta}
\phi_\theta(s,t)
=
a\langle H[\eta(\theta)-\eta(t)],\eta(\theta)-\eta(t)\rangle_{\cH}
+\langle H\eta(t),\eta(t)\rangle_{\cH}
-\langle H\eta(s),\eta(t)\rangle_{\cH}.
\end{align}
Under this convention, $a=-1$ and $a=0$ generate the $H$-weighted versions of
the posterior mean squared error~\eqref{eq:mse} and posterior
variance~\eqref{eq:pv}, respectively.

For the exponential-family likelihood process, we construct the corresponding
linear perturbation through the sufficient statistics. Let
\[
W
\coloneqq
\bigl(W_0(\omega_1),\ldots,W_0(\omega_{M_n})\bigr)\in\cH .
\]
Define a perturbation $(X_\beta)_{\beta\in\cB}$ of $X$, in the sense of
Section~\ref{sec:comparison}, by
\begin{align}\label{eq:exp_X_beta_pert}
X_\beta(t)
\coloneqq
X(t)+\beta\widetilde X(t),
\qquad t\in T_n,
\end{align}
where $a=-1$ and $a=0$ select the MSE and variance perturbations, respectively,
and
\begin{align}\label{eq:exp_X_tilde}
\widetilde X(t)
\coloneqq
\left\langle
H\Sigma^+
\bigl(W-\E_\theta [W]\bigr),
\eta(t)
\right\rangle
+
a\,\langle H[\eta(\theta)-\eta(t)],\eta(\theta)-\eta(t)\rangle_{\cH}.
\end{align}
This is the exponential-family analogue of~\eqref{eq:Y_tilde}. Since
$\dot X_\beta=\widetilde X$, Lemma~\ref{lem:deri_F} and~\eqref{eq:psi_X}
yield the induced test function
\begin{align}\label{eq:exp_psi_theta0}
\psi_{\theta,0}(s,t)
=
\E_\theta[\widetilde X(t)]
+
\cov_\theta(X(t),\widetilde X(t))
-
\cov_\theta(X(s),\widetilde X(t)).
\end{align}
Let the block covariance operator of the sufficient statistics under the local
laws $P_{\theta,m}$ be
\begin{align}\label{eq:exp_Sigma_W_theta}
\Sigma_W(\theta)
\coloneqq
\bigoplus_{m=1}^{M_n}
\nabla_{vv}^2\Xi(0,\eta_m(\theta)).
\end{align}
The Gaussian and non-Gaussian test functions then satisfy
\begin{align}\label{eq:exp_psi_remainder}
\psi_{\theta,0}(s,t)
=
\phi_\theta(s,t)
+
\left\langle
H\Sigma^+
\bigl(\Sigma_W(\theta)-\Sigma\bigr)
\eta(t),
\eta(t)-\eta(s)
\right\rangle .
\end{align}

\begin{theorem}[Posterior universality for exponential families]\label{thm:exp_posterior_univ}
Work under the assumptions of Theorem~\ref{thm:exp_free_energy_univ}, and let
$H$ and $\phi_\theta$ be specified by
\eqref{eq:posterior_product_operators}--\eqref{eq:exp_phi_theta}. Let $\Phi_n$
and $\Phi_n^{\mathrm G}$ be the normalized posterior functionals in
\eqref{eq:normalized_posterior_observables}, where
$\Phi^{\mathrm G}(\theta,\beta)$ is generated by the Gaussian perturbation
\eqref{eq:Y_beta}--\eqref{eq:Y_tilde} with the product operators in
\eqref{eq:posterior_product_operators}. Then, for every fixed sufficiently
small $\eps>0$, uniformly over $\theta\in\Theta_n$,
\begin{align*}
\Phi_n^{\mathrm G}(\theta,-\eps)
-
C\left(\eps+\frac{r_n}{\eps}\right)
\le\,
\Phi_n(\theta)
\,\le
\Phi_n^{\mathrm G}(\theta,\eps)
+
C\left(\eps+\frac{r_n}{\eps}\right),
\end{align*}
where $C<\infty$ is independent of $n$, $\theta$, and $\eps$.
\end{theorem}

\begin{proof}[Proof sketch]
Apply Theorem~\ref{thm:Phi_comparison} to the Gaussian perturbation
\eqref{eq:Y_beta}--\eqref{eq:Y_tilde}, specialized by
\eqref{eq:posterior_product_operators}, and the exponential-family perturbation
\eqref{eq:exp_X_beta_pert}--\eqref{eq:exp_X_tilde}. The estimates used in the
proof of Theorem~\ref{thm:exp_free_energy_univ} also control the perturbed free
energy gap, yielding the stated sandwich bounds. Details are given in
Appendix~\ref{app:proof_exp_posterior}.
\end{proof}

The operator $H$ specifies the direction
along which posterior
observable is measured; $H=\Id$ recovers the normalized posterior mean squared
error and posterior variance. The condition $M_n/n\to\infty$ holds automatically
in the classical local asymptotic normality regime, where the parameter
dimension is fixed and the number of observations tends to infinity.

\subsection{Regular Local Log-Likelihoods}
\label{subsec:regular_local_log_likelihoods}

The exponential-family setting above makes the universality mechanism
explicit, but the comparison requires only a local second-order representation
of the log-likelihood process and a compatible first-order representation of
the data-generating likelihood ratio. Imposing these conditions directly
yields a broader class of non-Gaussian models to which the same free energy and
posterior comparisons apply.

We retain the product structure
\eqref{eq:product_data}--\eqref{eq:product_feature_map}. In this subsection
we display the $n$-dependence of the local objects, since the local variables,
operators, and laws may scale with $n$. The additional regularity condition is
local: the $m$th term in the decomposition~\eqref{eq:X_product} is written
here as $X_m^{(n)}$ and admits the representation
\begin{equation}
\begin{aligned}\label{eq:regular_X_local}
X_m^{(n)}(t,\omega_m)
&=
\langle U_m^{(n)}(\omega_m),\eta_m(t)\rangle_{\cH_m}
\\
&\qquad\qquad
-\frac12
\langle
V_m^{(n)}(\omega_m)\eta_m(t),
\eta_m(t)
\rangle_{\cH_m}
+
R_m^{(n)}(t,\omega_m),
\end{aligned}
\end{equation}
Here $U_m^{(n)}$ is a random element of $\cH_m$,
$V_m^{(n)}$ is a random self-adjoint operator on $\cH_m$, and
$R_m^{(n)}$ is a remainder term. Let
$
P_0
\coloneqq
\bigotimes_{m=1}^{M_n}P_{0,m}^{(n)}
$
be a null product measure, independent of $\theta$. The local
data-generating laws are specified relative to $P_{0,m}^{(n)}$ by
\begin{align}\label{eq:regular_likelihood_ratio}
\log
\frac{\dd P_{\theta,m}^{(n)}}{\dd P_{0,m}^{(n)}}(\omega_m)
&=
\langle S_m^{(n)}(\omega_m),\eta_m(\theta)\rangle_{\cH_m}
+
\Delta_m^{(n)}(\theta,\omega_m),
\qquad \theta\in\Theta_n .
\end{align}
Here $S_m^{(n)}$ is a random element of $\cH_m$ and
$\Delta_m^{(n)}$ is a remainder term. Together with
$P_\theta=\bigotimes_{m=1}^{M_n}P_{\theta,m}^{(n)}$, this gives a
compatible local representation of the statistical model and the
data-generating family. In this formulation, regularity is imposed through the
moment bounds in Assumption~\ref{assump:regular_local_bounds}, which control
the local quantities appearing in
\eqref{eq:regular_X_local}--\eqref{eq:regular_likelihood_ratio}.

The scale used in these bounds is the feature radius $r_n$ in
\eqref{eq:feature_radius}. In the moment bounds below, $\E_0$ and
$\E_\theta$ denote expectation under $P_{0,m}^{(n)}$ and
$P_{\theta,m}^{(n)}$, respectively. We first collect the second moments,
\begin{align}\label{eq:regular_leading_moment}
\mathsf L_{m,n}
&\coloneqq
\E_0\|U_m^{(n)}\|_{\cH_m}^2
+
\E_0\|S_m^{(n)}\|_{\cH_m}^2 .
\end{align}
We next define the local error quantities. For $\theta\in\Theta_n$, write
$\Delta_m^{(n)}(\theta)=\Delta_m^{(n)}(\theta,\omega_m)$ and define
\begin{equation}\label{eq:regular_moment_U}
\cE^U_{m,n}(\theta)
\coloneqq
r_n\Bigl(
\E_0\|U_m^{(n)}\|_{\cH_m}^3
+
\E_\theta\|U_m^{(n)}\|_{\cH_m}^3
\Bigr),
\end{equation}
\begin{equation}\label{eq:regular_moment_S}
\cE^S_{m,n}(\theta)
\coloneqq
r_n\Bigl(
\E_0\|S_m^{(n)}\|_{\cH_m}^3
+
\E_\theta\|S_m^{(n)}\|_{\cH_m}^3
\Bigr),
\end{equation}
\begin{equation}\label{eq:regular_moment_V}
\begin{aligned}
\cE^V_{m,n}(\theta)
&\coloneqq
r_n\sqrt{\E_0\|V_m^{(n)}\|_{\mathrm{op}}^2}
+
r_n^4\Bigl(
\E_0\|V_m^{(n)}\|_{\mathrm{op}}^3
+
\E_\theta\|V_m^{(n)}\|_{\mathrm{op}}^3
\Bigr),
\end{aligned}
\end{equation}
\begin{equation}\label{eq:regular_moment_R}
\begin{aligned}
\cE^R_{m,n}(\theta)
&\coloneqq
\sup_{t\in T_n}
\Biggl\{
r_n^{-2}|\E_0R_m^{(n)}(t)|
+
r_n^{-1}\sqrt{\E_0|R_m^{(n)}(t)|^2}
\\
&\qquad\qquad{}
+
r_n^{-2}\E_0|R_m^{(n)}(t)|^3
\Biggr\}
+
r_n^{-2}\E_\theta\left[
\sup_{t\in T_n}|R_m^{(n)}(t)|^3
\right],
\end{aligned}
\end{equation}
\begin{equation}\label{eq:regular_moment_Delta}
\begin{aligned}
\cE^\Delta_{m,n}(\theta)
&\coloneqq
r_n^{-1}\sqrt{\E_0|\Delta_m^{(n)}(\theta)|^2}
+
r_n^{-2}\Bigl(
\E_0|\Delta_m^{(n)}(\theta)|^3
+
\E_\theta|\Delta_m^{(n)}(\theta)|^3
\Bigr).
\end{aligned}
\end{equation}

\begin{assumption}[local moments bounds]\label{assump:regular_local_bounds}
For the local representations
\eqref{eq:regular_X_local}--\eqref{eq:regular_likelihood_ratio}, let
$\delta_n(\theta)$ be a nonnegative deterministic quantity, indexed by
$n$ and $\theta\in\Theta_n$. The following bounds are assumed to hold
uniformly over $m\le M_n$ and $\theta\in\Theta_n$:
\[
\begin{aligned}
&\qquad\qquad\qquad\E_0[U_m^{(n)}]
=0,
\qquad
\mathsf L_{m,n}\lesssim1,
\\
&\cE^U_{m,n}(\theta)
+
\cE^S_{m,n}(\theta)
+
\cE^V_{m,n}(\theta)
+
\cE^R_{m,n}(\theta)
+
\cE^\Delta_{m,n}(\theta)
\lesssim
\delta_n(\theta).
\end{aligned}
\]
\end{assumption}

The five quantities in~\eqref{eq:regular_moment_U}--\eqref{eq:regular_moment_Delta}
are the local moment bounds on $U_m^{(n)},V_m^{(n)},R_m^{(n)},S_m^{(n)}$, and
$\Delta_m^{(n)}$. They collect the higher-order and remainder terms under the
two local laws $P_{0,m}^{(n)}$ and $P_{\theta,m}^{(n)}$ into a single scale
$\delta_n(\theta)$.
The exponential-family construction above is a special case: Taylor expansion
of the local log-partition functions identifies the corresponding
$U_m^{(n)},V_m^{(n)},R_m^{(n)},S_m^{(n)}$, and $\Delta_m^{(n)}$. In that
setting the local laws are homogeneous and the exponential-family regularity
ensures the required moments bounds.

The corresponding Gaussian location family is obtained by matching the local
first- and second-order structure under the null law. Since the spaces
$\cH_m$ need not be finite-dimensional, the local operators are defined
through bilinear forms. Assumption~\ref{assump:regular_local_bounds} ensures
that the following forms are bounded on $\cH_m\times\cH_m$ for fixed $n$; hence, by the
Riesz representation theorem, they define unique bounded operators
$A_m^{(n)},\Sigma_m^{(n)},K_m^{(n)}$ through
\begin{align}\label{eq:regular_operator_forms}
\langle A_m^{(n)} h,g\rangle_{\cH_m}
&\coloneqq
\E_0\bigl[
\langle S_m^{(n)},h\rangle_{\cH_m}
\langle U_m^{(n)},g\rangle_{\cH_m}
\bigr],
\\
\langle \Sigma_m^{(n)} h,g\rangle_{\cH_m}
&\coloneqq
\E_0\bigl[
\langle U_m^{(n)},h\rangle_{\cH_m}
\langle U_m^{(n)},g\rangle_{\cH_m}
\bigr],
\\
\langle K_m^{(n)} h,g\rangle_{\cH_m}
&\coloneqq
\E_0\bigl[
\langle V_m^{(n)} h,g\rangle_{\cH_m}
\bigr].
\label{eq:regular_operator_K}
\end{align}
Then $\Sigma_m^{(n)}$ is nonnegative and self-adjoint, while $K_m^{(n)}$ is
self-adjoint. Substituting these local operators into the product Gaussian
location construction in~\eqref{eq:gaussian_product_operators}--\eqref{eq:Y_local_moments}
gives the associated Gaussian comparison process.

\begin{theorem}[Free energy universality for regular local log-likelihoods]
\label{thm:regular_free_energy_univ}
Consider the product structure
\eqref{eq:product_data}--\eqref{eq:product_feature_map}. Let the local
likelihood processes and product data-generating laws be given by
\eqref{eq:regular_X_local}--\eqref{eq:regular_likelihood_ratio}, and let $Y$
denote the associated Gaussian location counterpart determined by the local
operators in~\eqref{eq:regular_operator_forms}--\eqref{eq:regular_operator_K}.
Suppose that
Assumption~\ref{assump:regular_local_bounds} holds and
$r_n\lesssim\sqrt{n/M_n}$. Then, for every $\theta\in\Theta_n$,
\[
\bigl|
F_n(\theta)-F_n^{\mathrm G}(\theta)
\bigr|
\lesssim
\delta_n(\theta)\{ 1+\delta_n(\theta) \}.
\]
\end{theorem}

\begin{proof}
The proof is given in Appendix~\ref{app:proof_regular_free_energy}.
\end{proof}

To formulate the corresponding comparison of posterior mean squared error and
posterior variance in prescribed feature directions, choose self-adjoint local
operators
$H_m^{(n)}:\cH_m\to\cH_m$, independent of $\theta$, satisfying the product
version of the range condition in~\eqref{eq:H_Sigma},
\[
\operatorname{Range}(H_m^{(n)})
\subseteq
\operatorname{Range}(\Sigma_m^{(n)}),
\qquad m\le M_n .
\]
The Douglas factorization theorem then yields bounded operators
$B_m^{(n)}:\cH_m\to\cH_m$ such that
$H_m^{(n)}=\Sigma_m^{(n)}B_m^{(n)}$. Suppressing the superscript $(n)$ for
the resulting global operators, define
\begin{align}\label{eq:regular_posterior_product_operators}
\Sigma
\coloneqq
\bigoplus_{m=1}^{M_n}\Sigma_m^{(n)},
\qquad
H
\coloneqq
\bigoplus_{m=1}^{M_n}H_m^{(n)},
\qquad
B
\coloneqq
\bigoplus_{m=1}^{M_n}B_m^{(n)} .
\end{align}
The test function $\phi_\theta$ in~\eqref{eq:phi_theta_lin}, with
$H_\theta=H$, generates the posterior observables: $a=-1$
gives the posterior mean squared error, whereas $a=0$ gives the posterior
variance.

\begin{theorem}[Posterior universality for regular local log-likelihoods]
\label{thm:regular_posterior_comparison}
Consider the setting of Theorem~\ref{thm:regular_free_energy_univ}, with the
product operators in~\eqref{eq:regular_posterior_product_operators} and the
test function in~\eqref{eq:phi_theta_lin} with $H_\theta=H$.
Let $\Phi_n$ and $\Phi_n^{\mathrm G}$ denote the normalized posterior
functionals in~\eqref{eq:normalized_posterior_observables}, where
$\Phi^{\mathrm G}(\theta,\beta)$ in~\eqref{eq:Phi_G_beta} is generated by the
Gaussian perturbation
\eqref{eq:Y_beta}--\eqref{eq:Y_tilde} with the product operators in
\eqref{eq:regular_posterior_product_operators}.
Set
\[
\mathfrak h_n
\coloneqq
\max_{m\le M_n}
\left(
\|H_m^{(n)}\|_{\mathrm{op}}
+
\|B_m^{(n)}\|_{\mathrm{op}}
\right),
\qquad
\widetilde\delta_n(\theta)
\coloneqq
\delta_n(\theta)\{1+\delta_n(\theta)\}.
\]
If $\mathfrak h_n=O(1)$, then, for every fixed sufficiently small $\eps>0$ and
every $\theta\in\Theta_n$,
\begin{align*}
\Phi_n^{\mathrm G}(\theta,-\eps)
-
C\left(\eps+\frac{\widetilde\delta_n(\theta)}{\eps}\right)
\le\,
\Phi_n(\theta)
\,\le
\Phi_n^{\mathrm G}(\theta,\eps)
+
C\left(\eps+\frac{\widetilde\delta_n(\theta)}{\eps}\right),
\end{align*}
where $C<\infty$ is independent of $n$, $\theta$, and $\eps$.
\end{theorem}

\begin{proof}
The proof is given in Appendix~\ref{app:proof_regular_posterior}.
\end{proof}

Theorems~\ref{thm:regular_free_energy_univ}
and~\ref{thm:regular_posterior_comparison} contain the exponential-family
results in Theorems~\ref{thm:exp_free_energy_univ}
and~\ref{thm:exp_posterior_univ} as the special case
$\delta_n(\theta)=r_n$. The regular local formulation additionally permits
heterogeneous local components, Hilbert spaces $\cH_m$ that need not be
finite-dimensional, and weaker regularity conditions expressed directly in
terms of local moments and remainders. Most importantly, the local submodels
themselves may vary with $n$.

The sandwich bound in Theorem~\ref{thm:regular_posterior_comparison} controls
$\Phi_n(\theta)$ by the perturbed Gaussian functionals
$\Phi_n^{\mathrm G}(\theta,\pm\eps)$. Recovering the unperturbed Gaussian
observable therefore requires an asymptotic analysis in which $n\to\infty$
and then $\eps\downarrow0$. This passage is governed by the limiting perturbed
Gaussian free energy: one-sided derivatives yield bounds on subsequential
limits, while differentiability at the origin identifies a unique limit. The
following corollary makes this conclusion precise after averaging over
$\theta$.

\begin{cor}\label{cor:regular_posterior_limit}
Suppose the assumptions of Theorem~\ref{thm:regular_posterior_comparison}
hold. Let $(\nu_n)$ be probability measures on $\Theta_n$ satisfying
\[
\int_{\Theta_n}\widetilde\delta_n(\theta)\,\nu_n(\dd\theta)\to0.
\]
Assume that, for $\beta$ in a neighborhood of zero, the averaged Gaussian free
energies
\[
\mathcal F_n^{\mathrm G}(\beta)
\coloneqq
\int_{\Theta_n}F_n^{\mathrm G}(\theta,\beta)\,\nu_n(\dd\theta)
\]
converge pointwise to a finite limit $\mathcal F^{\mathrm G}(\beta)$. Then
\[
\partial_-\mathcal F^{\mathrm G}(0)
\le
\liminf_{n\to\infty}
\int_{\Theta_n}\Phi_n(\theta)\,\nu_n(\dd\theta)
\le
\limsup_{n\to\infty}
\int_{\Theta_n}\Phi_n(\theta)\,\nu_n(\dd\theta)
\le
\partial_+\mathcal F^{\mathrm G}(0).
\]
In particular, if $\mathcal F^{\mathrm G}$ is differentiable at the origin, then
\[
\lim_{n\to\infty}\int_{\Theta_n}\Phi_n(\theta)\,\nu_n(\dd\theta)
=
\lim_{n\to\infty}\int_{\Theta_n}\Phi_n^{\mathrm G}(\theta)\,\nu_n(\dd\theta)
=
(\mathcal F^{\mathrm G})'(0).
\]
\end{cor}

\begin{proof}
The proof is given in
Appendix~\ref{app:proof_regular_posterior_limit}.
\end{proof}

Corollary~\ref{cor:regular_posterior_limit} therefore reduces the asymptotic
analysis of posterior observables to the existence of a limit
for the perturbed Gaussian free energy. The application below illustrates how
this limiting free energy can be identified in a concrete high-dimensional
model.

\section{Application: Sparse Bernoulli Hypergraphs}
\label{sec:app_sparse_bernoulli}

We illustrate the preceding theory with a sparse Bernoulli hypergraph model
whose $n$-dependent local channel fits the regular local log-likelihood
framework of Section~\ref{subsec:regular_local_log_likelihoods}. We recover the
free energy universality previously established by a different argument
in~\cite{zou:2026}, and further obtain universality of the posterior mean
squared error and posterior variance. In the special case $p=2$, $d=1$, and
$J_\alpha=1$, we then combine this comparison with existing Gaussian free
energy limits~\cite{guionnet2025estimating} to analyze the limiting posterior
observables.

We first specify the tensor feature map used by the observation model. Fix
integers $p\ge2$ and $d\ge1$, and let
$\Theta_n=T_n=[-1,1]^{d\times n}$. Thus
$\theta=(\theta_1,\ldots,\theta_n)\in\Theta_n$ has coordinates
$\theta_i\in[-1,1]^d$. To each multi-index
$\alpha=(\alpha_1,\ldots,\alpha_p)\in[n]^p$, we associate the local feature
\begin{equation}\label{eq:sparse_tensor_feature_map}
\eta_\alpha(u)
=
n^{\frac{1-p}{2}}\,
u_{\alpha_1}\otimes\cdots\otimes u_{\alpha_p}
\in\bbR^{d^p},
\qquad u\in\Theta_n.
\end{equation}
Collecting these local tensors gives the product feature map
$\eta(u)=(\eta_\alpha(u))_{\alpha\in[n]^p}$, with $M_n=n^p$ local components.
The normalization in~\eqref{eq:sparse_tensor_feature_map} is chosen so that the
feature radius is on the canonical scale introduced in
\eqref{eq:feature_radius}. We index local observations by the same
multi-indices $\alpha\in[n]^p$. For each $\alpha$, let
$J_\alpha\in\bbR^{d^p}$ be a deterministic coupling vector and write
$z_\alpha(u)=\langle J_\alpha,\eta_\alpha(u)\rangle$. Let
$s_n=M_n/n=n^{p-1}$ and let $\de_n\in(0,s_n)$ denote the average-degree
parameter. Given the signal $\theta\in\Theta_n$, the observations
$G=(G_\alpha)_{\alpha\in[n]^p}$ are conditionally independent with
\begin{equation}\label{eq:bernoulli_hypergraph_data}
G_\alpha
\sim
\operatorname{Bernoulli}
\left(
\frac{\de_n}{s_n}
+
\sqrt{\frac{\de_n}{s_n}\left(1-\frac{\de_n}{s_n}\right)}
\,\sqrt{\lambda_\star}\,z_\alpha(\theta)
\right),
\qquad \alpha\in[n]^p .
\end{equation}
Here $\lambda_\star>0$ is the signal-to-noise ratio of the data-generating
channel. The statistical model may use a different signal-to-noise ratio
$\lambda>0$. With $b_n=\sqrt{(s_n-\de_n)/\de_n}$, its log-likelihood, written
relative to the baseline product measure under which
$G_\alpha\sim\operatorname{Bernoulli}(\de_n/s_n)$ independently over
$\alpha\in[n]^p$, is
\begin{equation}\label{eq:bernoulli_local_likelihood}
L_n(t,G)
\coloneqq
\sum_{\alpha\in[n]^p}
\left[
G_\alpha\log(1+b_n\sqrt{\lambda}\,z_\alpha(t))
+
(1-G_\alpha)\log\left(1-\frac{\sqrt{\lambda}\,z_\alpha(t)}{b_n}\right)
\right].
\end{equation}
where $t\in T_n$ is the candidate parameter. 

The Gaussian comparison model has the same feature map. Its data-generating
law is given by independent scalar Gaussian observations
\begin{equation}\label{eq:bernoulli_gaussian_observation}
W_\alpha
\sim
\normal(\sqrt{\lambda_\star}\,z_\alpha(\theta),1),
\qquad \alpha\in[n]^p,
\end{equation}
while the corresponding statistical model is
$\{\normal(\sqrt{\lambda}\,z_\alpha(t),1):t\in T_n\}$ for each local coordinate
$\alpha$. The scalar case $d=1$ recovers the sparse Bernoulli hypergraph model
considered in~\cite{zou:2026}, written here in the notation of the present
paper.

\parhead{Universality via the presented framework}
We now formulate this model in the framework of the preceding sections. The
local Hilbert space is $\cH_\alpha=\bbR^{d^p}$, and the tensor feature map
\eqref{eq:sparse_tensor_feature_map} gives the product feature structure of
\eqref{eq:product_feature_map}. The likelihood process in
Section~\ref{sec:general_frame} is $X(t)=L_n(t,G)$, or equivalently
$X(t)=\sum_{\alpha\in[n]^p}X_\alpha(t,G_\alpha)$, where
\[
X_\alpha(t,G_\alpha)
=
G_\alpha\log(1+b_n\sqrt{\lambda}\,z_\alpha(t))
+
(1-G_\alpha)\log\left(1-\frac{\sqrt{\lambda}\,z_\alpha(t)}{b_n}\right).
\]
Expanding the log-likelihood \eqref{eq:bernoulli_local_likelihood} in powers
of $z_\alpha(t)$ gives the regular local representation
\eqref{eq:regular_X_local}. The first two local coefficients are
\begin{equation}\label{eq:bernoulli_UV}
U_\alpha^{(n)}(G_\alpha)
=
\sqrt{\lambda}
\left(b_nG_\alpha-\frac{1-G_\alpha}{b_n}\right)J_\alpha,
\qquad
V_\alpha^{(n)}(G_\alpha)
=
\lambda
\left(b_n^2G_\alpha+\frac{1-G_\alpha}{b_n^2}\right)
J_\alpha J_\alpha^\top .
\end{equation}
The data-generating likelihood ratio admits the analogous expansion in the
true local feature $z_\alpha(\theta)$. The explicit remainder and the
associated $S_\alpha^{(n)},\Delta_\alpha^{(n)}$ are recorded in
Appendix~\ref{app:sparse_bernoulli_details}.

For the Gaussian comparison, the statistical model and data-generating process
specified around~\eqref{eq:bernoulli_gaussian_observation} yield the explicit
log-likelihood process
\begin{equation}\label{eq:bernoulli_gaussian_process}
Y(t)
=
\sum_{\alpha\in[n]^p}
\left\{
\sqrt{\lambda}\,W_\alpha z_\alpha(t)
-
\frac{\lambda}{2}z_\alpha(t)^2
\right\}.
\end{equation}
Using the product Gaussian construction
\eqref{eq:gaussian_product_operators}--\eqref{eq:Y_local_moments}, one
verifies from Definition~\ref{def:gaussian_loc} that this is the product
Gaussian location family with local operators
\begin{equation}\label{eq:bernoulli_gaussian_operators}
A_\alpha
=
\sqrt{\lambda\lambda_\star}\,J_\alpha J_\alpha^\top,
\qquad
\Sigma_\alpha
=
K_\alpha
=
\lambda
J_\alpha J_\alpha^\top .
\end{equation}
The coefficients $U_\alpha^{(n)},V_\alpha^{(n)}$ above depend on $n$ through
$b_n$, whereas the Gaussian operators in
\eqref{eq:bernoulli_gaussian_operators} are written without an additional
superscript. To see why the local expansion selects this Gaussian comparison,
note that
\[
\E_0[U_\alpha^{(n)}(U_\alpha^{(n)})^\top]
=
\Sigma_\alpha,
\qquad
\E_0[V_\alpha^{(n)}]
=
K_\alpha,
\]
so the linear and quadratic coefficients in the Bernoulli likelihood match the
covariance and quadratic operators of the Gaussian process above. The complete
verification of the local expansion and the moment bounds is given in
Appendix~\ref{app:sparse_bernoulli_details}.

\begin{theorem}[Universality for sparse Bernoulli hypergraphs]
\label{thm:sparse_bernoulli_univ}
Consider the sparse Bernoulli hypergraph model
\eqref{eq:bernoulli_hypergraph_data}--\eqref{eq:bernoulli_local_likelihood}
and its Gaussian counterpart~\eqref{eq:bernoulli_gaussian_process}. Assume
that $\sup_{\alpha\in[n]^p}\|J_\alpha\|=O(1)$, that
$\de_n\to\infty$ and $s_n-\de_n\to\infty$, and that the Bernoulli
channel is well defined uniformly over $\Theta_n$. Then, for any prior $\pi$ on $T_n$,
\[
\sup_{\theta\in\Theta_n}
\bigl|
F_n(\theta)-F_n^{\mathrm G}(\theta)
\bigr|
\lesssim
\delta_n,
\qquad
\delta_n
\coloneqq
\frac1{\sqrt{\de_n}}
+
\frac1{\sqrt{s_n-\de_n}}.
\]
\end{theorem}

\begin{proof}
The proof is given in Appendix~\ref{app:sparse_bernoulli_details}.
\end{proof}

The scalar case $d=1$ and $J_\alpha=1$ is contained
in~\cite{zou:2026}, where the free energy comparison is proved by a
generic chaining argument and sub-gamma control of the local quadratic
coefficient $V_\alpha^{(n)}$. The present proof uses only the moment
bounds in Assumption~\ref{assump:regular_local_bounds}, here amounting
to second and third moments of $V_\alpha^{(n)}$, and achieves the
same rate. 
The posterior
comparison below is new to~\cite{zou:2026}.

We now make the Gaussian perturbation explicit. First write the Gaussian
process~\eqref{eq:bernoulli_gaussian_process} in the form of
Definition~\ref{def:gaussian_loc}. With the local operators in
\eqref{eq:bernoulli_gaussian_operators},
\begin{equation}\label{eq:sparse_gaussian_location_moments}
\begin{aligned}
\E_{\theta}[Y(t)]
&=
\sum_{\alpha\in[n]^p}
\left\{
\langle A_\alpha\eta_\alpha(\theta),\eta_\alpha(t)\rangle
-
\frac12
\langle K_\alpha\eta_\alpha(t),\eta_\alpha(t)\rangle
\right\},
\\
\cov_{\theta}(Y(s),Y(t))
&=
\sum_{\alpha\in[n]^p}
\langle \Sigma_\alpha\eta_\alpha(s),\eta_\alpha(t)\rangle .
\end{aligned}
\end{equation}
The product Gaussian linear perturbation
\eqref{eq:Y_beta}--\eqref{eq:Y_tilde} uses local matrix directions
$H_\alpha\in\bbR^{d^p\times d^p}$ satisfying the componentwise range condition
in~\eqref{eq:H_Sigma}. Since
$\Sigma_\alpha=\lambda J_\alpha J_\alpha^\top$, this condition restricts the
local direction to the one-dimensional span of $J_\alpha$. Up to a scalar
multiple, the corresponding matrix direction is $J_\alpha J_\alpha^\top$; we
therefore take
\begin{equation}\label{eq:sparse_H_direction}
H_\alpha
=
J_\alpha J_\alpha^\top,
\qquad
H
=
\bigoplus_{\alpha\in[n]^p}H_\alpha .
\end{equation}
For $a\in\{-1,0\}$, corresponding respectively to posterior mean squared error
and posterior variance, define
\begin{equation}\label{eq:sparse_perturbed_gaussian_operators}
\begin{aligned}
A_\alpha^{(\beta)}
&\coloneqq
\bigl(\sqrt{\lambda\lambda_\star}-2\beta a\bigr)H_\alpha,
\quad
\Sigma_\alpha^{(\beta)}
\coloneqq
\frac{(\lambda+\beta)^2}{\lambda}
H_\alpha,
\quad
K_\alpha^{(\beta)}
\coloneqq
\bigl(\lambda-2\beta a\bigr)H_\alpha.
\end{aligned}
\end{equation}
The perturbed Gaussian process $Y_\beta$ generated by
\eqref{eq:Y_beta}--\eqref{eq:Y_tilde} and~\eqref{eq:sparse_H_direction} is
equivalently specified by
\begin{equation}\label{eq:sparse_gaussian_perturbed_moments}
\begin{aligned}
\E_{\theta}[Y_\beta(t)]
&=
\sum_{\alpha\in[n]^p}
\Bigl\{
\langle A_\alpha^{(\beta)}\eta_\alpha(\theta),\eta_\alpha(t)\rangle
\\
&\qquad\qquad
-
\frac12
\langle K_\alpha^{(\beta)}\eta_\alpha(t),\eta_\alpha(t)\rangle
+
\beta a
\langle H_\alpha\eta_\alpha(\theta),\eta_\alpha(\theta)\rangle
\Bigr\},
\\
\cov_{\theta}(Y_\beta(s),Y_\beta(t))
&=
\sum_{\alpha\in[n]^p}
\langle \Sigma_\alpha^{(\beta)}\eta_\alpha(s),\eta_\alpha(t)\rangle .
\end{aligned}
\end{equation}
At $\beta=0$, the operators in~\eqref{eq:sparse_perturbed_gaussian_operators}
reduce to those in~\eqref{eq:bernoulli_gaussian_operators}, and
\eqref{eq:sparse_gaussian_perturbed_moments} recovers
\eqref{eq:sparse_gaussian_location_moments}.

\begin{remark}
	For $\beta$ in a small
neighborhood of zero, the perturbation has the following scalar-channel
interpretation. Let
\[
\mu_\beta
\coloneqq
\sqrt{\lambda-2a\beta},
\qquad
\mu_{\star,\beta}
\coloneqq
\frac{\sqrt{\lambda\lambda_\star}-2a\beta}{\sqrt{\lambda-2a\beta}},
\qquad
\sigma_{\star,\beta}^2
\coloneqq
\frac{\lambda+2\beta+\beta^2/\lambda}{\lambda-2a\beta}.
\]
The data-generating law is
$W_{\alpha,\beta}\sim
\normal(\mu_{\star,\beta}z_\alpha(\theta),\sigma_{\star,\beta}^2)$,
independently over $\alpha$, and the statistical model is
$\{\normal(\mu_\beta z_\alpha(t),1):t\in T_n\}$. Its log-likelihood process
agrees with $Y_\beta(t)$ up to the additive term
$\beta a\sum_\alpha\langle H_\alpha\eta_\alpha(\theta),
\eta_\alpha(\theta)\rangle$, which is independent of $t$.
\end{remark}

Let $Y_\beta^{\mathrm{mse}}$ and $Y_\beta^{\mathrm{var}}$ denote the Gaussian
perturbations above corresponding to $a=-1$ and $a=0$, respectively. Writing
$\widehat\eta_X=\int_{T_n}\eta(t)\pi_X(\dd t)$, define
\begin{equation}\label{eq:sparse_directional_observables}
\begin{aligned}
\cR_n(\theta)
&\coloneqq
\frac1n\E_{\theta}
\left[\|\eta(\theta)-\widehat\eta_X\|_H^2\right],
\\
\cV_n(\theta)
&\coloneqq
\frac1n\E_{\theta}
\left[
\int_{T_n}
\|\eta(t)-\widehat\eta_X\|_H^2\,\pi_X(\dd t)
\right].
\end{aligned}
\end{equation}
For either $Z=Y_\beta^{\mathrm{mse}}$ or $Z=Y_\beta^{\mathrm{var}}$, set
$\widehat\eta_Z=\int_{T_n}\eta(t)\pi_Z(\dd t)$. The corresponding Gaussian
observables are
\begin{equation}\label{eq:sparse_directional_observables_G}
\begin{aligned}
\cR_n^{\mathrm G}(\theta,\beta)
&\coloneqq
\frac1n\E_{\theta}
\left[\|\eta(\theta)-\widehat\eta_{Y_\beta^{\mathrm{mse}}}\|_H^2\right],
\\
\cV_n^{\mathrm G}(\theta,\beta)
&\coloneqq
\frac1n\E_{\theta}
\left[
\int_{T_n}
\|\eta(t)-\widehat\eta_{Y_\beta^{\mathrm{var}}}\|_H^2\,
\pi_{Y_\beta^{\mathrm{var}}}(\dd t)
\right].
\end{aligned}
\end{equation}

\begin{theorem}[Posterior universality for sparse Bernoulli hypergraphs]
\label{thm:sparse_bernoulli_posterior_univ}
Work in the setting of Theorem~\ref{thm:sparse_bernoulli_univ}, and let
$\cR_n,\cV_n,\cR_n^{\mathrm G}$, and $\cV_n^{\mathrm G}$ be defined by
\eqref{eq:sparse_directional_observables}--\eqref{eq:sparse_directional_observables_G}.
There exists a constant $C<\infty$,
independent of $n$, $\theta$, and $\eps$, such that, for every sufficiently
small fixed $\eps>0$, uniformly over $\theta\in\Theta_n$,
\begin{align*}
\cR_n^{\mathrm G}(\theta,\eps)
-C\left(\eps+\frac{\delta_n}{\eps}\right)
&\le
\cR_n(\theta)
\le
\cR_n^{\mathrm G}(\theta,-\eps)
+C\left(\eps+\frac{\delta_n}{\eps}\right),
\\
\cV_n^{\mathrm G}(\theta,-\eps)
-C\left(\eps+\frac{\delta_n}{\eps}\right)
&\le
\cV_n(\theta)
\le
\cV_n^{\mathrm G}(\theta,\eps)
+C\left(\eps+\frac{\delta_n}{\eps}\right).
\end{align*}
\end{theorem}

\begin{proof}
The proof is given in
Appendix~\ref{app:proof_sparse_bernoulli_posterior_univ}.
\end{proof}

\parhead{Limiting analysis}
We now consider the limiting behavior of the posterior observables when $p=2$,
$d=1$, and $J_\alpha=1$ for all $\alpha\in[n]^2$, a setting for which Gaussian
free energy limits are available for product priors~\cite{guionnet2025estimating}. At
dimension $n$, write $\pi_n$ for the prior $\pi$ used throughout to define the
posterior, and assume $\pi_n=\pi_0^{\otimes n}$. We average the true parameter
under $\nu_n=\nu_0^{\otimes n}$, where the probability measures $\pi_0$ and
$\nu_0$ on $[-1,1]$ need not coincide. 

Let $z=(z_{ij})_{i,j\le n}$ have independent standard Gaussian entries. Write
\[
A_\beta=\sqrt{\lambda\lambda_\star}-2a\beta,
\qquad
K_\beta=\lambda-2a\beta,
\qquad
\Sigma_\beta=\lambda+2\beta+\frac{\beta^2}{\lambda}.
\]
Then the perturbed Gaussian free energy can be written as
\begin{equation}\label{eq:sparse_perturbed_gaussian_free_energy}
\begin{aligned}
F_n^{\mathrm G}(\theta,\beta)
=&
\frac{\beta a}{n^2}
\left(\sum_{i=1}^n\theta_i^2\right)^2
+
\frac1n
\E
\log\int_{T_n}
\exp\{g_{n,\beta}(u;\theta,z)\}\,\pi_n(\dd u),
\end{aligned}
\end{equation}
where $\E$ denotes expectation with respect to $z$ and
\begin{equation}\label{eq:sparse_gaussian_hamiltonian}
\begin{aligned}
g_{n,\beta}(u;\theta,z)
=&
\frac{\sqrt{\Sigma_\beta}}{\sqrt n}
\sum_{i,j=1}^n z_{ij}u_i u_j
+
\frac{A_\beta}{n}
\sum_{i,j=1}^n\theta_i\theta_j u_i u_j
-
\frac{K_\beta}{2n}
\sum_{i,j=1}^n u_i^2u_j^2 .
\end{aligned}
\end{equation}
For $a\in\{0,-1\}$, define the product-prior averaged log-partition function
\begin{equation}\label{eq:sparse_averaged_logpartition}
\begin{aligned}
\mathcal G_{n,a}^{\nu}(\beta)
&\coloneqq
\int_{\Theta_n}
\frac1n
\E
\log\int_{T_n}
\exp\{g_{n,\beta}^{(a)}(u;\theta,z)\}\,\pi_n(\dd u)\,
\nu_n(\dd\theta).
\end{aligned}
\end{equation}
Here $g_{n,\beta}^{(a)}$ denotes the Hamiltonian in
\eqref{eq:sparse_gaussian_hamiltonian} with the indicated value of $a$. The
quantity in \eqref{eq:sparse_averaged_logpartition} is precisely the type of
Gaussian free energy studied in~\cite{guionnet2025estimating}. For any
$\Sigma_\beta>0$ and $A_\beta,K_\beta\in\bbR$, the results therein yield its
large-$n$ limit under the product measures $\pi_n$ and $\nu_n$; in the
terminology of~\cite{guionnet2025estimating}, these three parameters are the temperature
parameters of the model. We denote this pointwise limit by
$\mathcal G_a^\nu(\beta)$.

\begin{cor}[Posterior observable limits]\label{cor:sparse_posterior_limits}
Assume the setting of Theorem~\ref{thm:sparse_bernoulli_posterior_univ} with
$p=2$, $d=1$, and $J_\alpha=1$ for all $\alpha\in[n]^2$. Suppose
$\delta_n\to0$, and let $\pi_n=\pi_0^{\otimes n}$ and
$\nu_n=\nu_0^{\otimes n}$ for probability measures $\pi_0$ and $\nu_0$ on
$[-1,1]$, not necessarily equal. Take $\pi_n$ as the candidate prior defining
the posterior observables. Then
\[
\partial_-\mathcal G_0^{\nu}(0)
\le
\liminf_{n\to\infty}
\int_{\Theta_n}\cV_n(\theta)\,\nu_n(\dd\theta)
\le
\limsup_{n\to\infty}
\int_{\Theta_n}\cV_n(\theta)\,\nu_n(\dd\theta)
\le
\partial_+\mathcal G_0^{\nu}(0).
\]
Moreover, define $q^\nu\coloneqq
\bigl(\int_{[-1,1]}x^2\,\nu_0(\dd x)\bigr)^2$ and
$\mathcal M^\nu(\beta)\coloneqq\beta q^\nu-\mathcal G_{-1}^\nu(\beta)$.
Then
\begin{equation*}
\begin{aligned}
\partial_+\mathcal M^\nu(0)
\le
\liminf_{n\to\infty}
\int_{\Theta_n}\cR_n(\theta)\,\nu_n(\dd\theta)
\le
\limsup_{n\to\infty}
\int_{\Theta_n}\cR_n(\theta)\,\nu_n(\dd\theta)
\le
\partial_-\mathcal M^\nu(0).
\end{aligned}
\end{equation*}
\end{cor}

\begin{proof}
The proof is given in Appendix~\ref{app:proof_sparse_posterior_limits}.
\end{proof}

\begin{remark}
The Gaussian free energy limits of~\cite{guionnet2025estimating} cover both
$\lambda\ne\lambda_\star$ and the matched case $\lambda=\lambda_\star$. In the matched setting, combining our framework with known
finite-rank tensor limits~\cite{chen:2022_finite-rank,rossetti:2025,rossetti2025fundamental} extends the analysis to fixed $p\ge1$ and finite
$d$. The averaged posterior mean squared error and variance then both equal the
MMSE, yielding MMSE universality.
\end{remark}

\begin{remark}
Proposition~2.4 of~\cite{guionnet2025estimating} proves free energy
universality for smooth $n$-independent channels. The regular local
log-likelihood results of Section~\ref{subsec:regular_local_log_likelihoods}
recover this conclusion and extend it to a broader class of channels. Their large-deviation
argument yields universality of rate functions for certain order parameters,
but not of posterior observables.
\end{remark}

\bibliographystyle{imsart-number}
\bibliography{long_names,library,universality,more_refs,more_refs1,more_refs2}

\begin{thebibliography}{63}

\bibitem{albiac:2006banach}
\begin{bbook}[author]
\bauthor{\bsnm{Albiac},~\bfnm{Fernando}\binits{F.}} \AND
  \bauthor{\bsnm{Kalton},~\bfnm{Nigel~J.}\binits{N.~J.}}
(\byear{2006}).
\btitle{Topics in Banach Space Theory}.
\bseries{Graduate Texts in Mathematics}
\bvolume{233}.
\bpublisher{Springer}.
\end{bbook}
\endbibitem

\bibitem{baik:2005}
\begin{barticle}[author]
\bauthor{\bsnm{Baik},~\bfnm{Jinho}\binits{J.}},
  \bauthor{\bsnm{Arous},~\bfnm{G\'erard~Ben}\binits{G.~B.}} \AND
  \bauthor{\bsnm{P\'{e}ch\'{e}},~\bfnm{Sandrine}\binits{S.}}
(\byear{2005}).
\btitle{Phase transition of the largest eigenvalue for nonnull complex sample
  covariance matrices}.
\bjournal{Annals of Probability}
\bvolume{33}
\bpages{1643--1697}.
\end{barticle}
\endbibitem

\bibitem{barbier:2016a}
\begin{binproceedings}[author]
\bauthor{\bsnm{Barbier},~\bfnm{Jean}\binits{J.}},
  \bauthor{\bsnm{Dia},~\bfnm{Mohamad}\binits{M.}},
  \bauthor{\bsnm{Macris},~\bfnm{Nicolas}\binits{N.}},
  \bauthor{\bsnm{Krzakala},~\bfnm{Florent}\binits{F.}},
  \bauthor{\bsnm{Lesieur},~\bfnm{Thibault}\binits{T.}} \AND
  \bauthor{\bsnm{Zdeborov\'a},~\bfnm{Lenka}\binits{L.}}
(\byear{2016}).
\btitle{Mutual information for symmetric rank-one matrix estimation: {A} proof
  of the replica formula}.
In \bbooktitle{Advances in Neural Information Processing Systems (NIPS)}
\bvolume{29}
\bpages{424--432}.
\end{binproceedings}
\endbibitem

\bibitem{barbier:2022}
\begin{binproceedings}[author]
\bauthor{\bsnm{Barbier},~\bfnm{Jean}\binits{J.}},
  \bauthor{\bsnm{Hou},~\bfnm{Tianqi}\binits{T.}},
  \bauthor{\bsnm{Mondelli},~\bfnm{Marco}\binits{M.}} \AND
  \bauthor{\bsnm{S\'aenz},~\bfnm{Manuel}\binits{M.}}
(\byear{2022}).
\btitle{The price of ignorance: how much does it cost to forget noise structure
  in low-rank matrix estimation?}
In \bbooktitle{Advances in Neural Information Processing Systems}
\bvolume{35}.
\end{binproceedings}
\endbibitem

\bibitem{barbier:2019}
\begin{barticle}[author]
\bauthor{\bsnm{Barbier},~\bfnm{Jean}\binits{J.}},
  \bauthor{\bsnm{Krzakala},~\bfnm{Florent}\binits{F.}},
  \bauthor{\bsnm{Macris},~\bfnm{Nicolas}\binits{N.}},
  \bauthor{\bsnm{Miolane},~\bfnm{L\'eo}\binits{L.}} \AND
  \bauthor{\bsnm{Zdeborov\'a},~\bfnm{Lenka}\binits{L.}}
(\byear{2019}).
\btitle{Optimal Errors and Phase Transitions in High-Dimensional Generalized
  Linear Models}.
\bjournal{Proceedings of the National Academy of Sciences}
\bvolume{116}
\bpages{5451-5460}.
\end{barticle}
\endbibitem

\bibitem{barbier:2022strong}
\begin{barticle}[author]
\bauthor{\bsnm{Barbier},~\bfnm{Jean}\binits{J.}} \AND
  \bauthor{\bsnm{Panchenko},~\bfnm{Dmitry}\binits{D.}}
(\byear{2022}).
\btitle{Strong replica symmetry in high-dimensional optimal {B}ayesian
  inference}.
\bjournal{Communications in Mathematical Physics}
\bvolume{393}
\bpages{1199--1239}.
\bdoi{10.1007/s00220-022-04387-w}
\end{barticle}
\endbibitem

\bibitem{barbier:2022logconcave}
\begin{barticle}[author]
\bauthor{\bsnm{Barbier},~\bfnm{Jean}\binits{J.}},
  \bauthor{\bsnm{Panchenko},~\bfnm{Dmitry}\binits{D.}} \AND
  \bauthor{\bsnm{S\'aenz},~\bfnm{Manuel}\binits{M.}}
(\byear{2022}).
\btitle{Strong replica symmetry for high-dimensional disordered log-concave
  {G}ibbs measures}.
\bjournal{Information and Inference: A Journal of the IMA}
\bvolume{11}
\bpages{1079--1108}.
\bdoi{10.1093/imaiai/iaab027}
\end{barticle}
\endbibitem

\bibitem{behne:2022}
\begin{binproceedings}[author]
\bauthor{\bsnm{Behne},~\bfnm{Joshua~K.}\binits{J.~K.}} \AND
  \bauthor{\bsnm{Reeves},~\bfnm{Galen}\binits{G.}}
(\byear{2022}).
\btitle{Fundamental limits for rank-one matrix estimation with groupwise
  heteroskedasticity}.
In \bbooktitle{Proceedings of The 25th International Conference on Artificial
  Intelligence and Statistics}.
\end{binproceedings}
\endbibitem

\bibitem{bogachev:1998gaussian}
\begin{bbook}[author]
\bauthor{\bsnm{Bogachev},~\bfnm{Vladimir~Igorevich}\binits{V.~I.}}
(\byear{1998}).
\btitle{Gaussian measures}
\bvolume{62}.
\bpublisher{American Mathematical Soc.}
\end{bbook}
\endbibitem

\bibitem{carmona:2006}
\begin{binproceedings}[author]
\bauthor{\bsnm{Carmona},~\bfnm{Philippe}\binits{P.}} \AND
  \bauthor{\bsnm{Hu},~\bfnm{Yueyun}\binits{Y.}}
(\byear{2006}).
\btitle{Universality in Sherrington--Kirkpatrick's spin glass model}.
In \bbooktitle{Annales de l'Institut Henri Poincare (B) Probability and
  Statistics}
\bvolume{42}
\bpages{215--222}.
\bpublisher{Elsevier}.
\end{binproceedings}
\endbibitem

\bibitem{castillo:2013}
\begin{barticle}[author]
\bauthor{\bsnm{Castillo},~\bfnm{Isma{\"e}l}\binits{I.}} \AND
  \bauthor{\bsnm{Nickl},~\bfnm{Richard}\binits{R.}}
(\byear{2013}).
\btitle{Nonparametric {B}ernstein--von {M}ises theorems in {G}aussian white
  noise}.
\bjournal{Annals of Statistics}
\bvolume{41}
\bpages{1999--2028}.
\end{barticle}
\endbibitem

\bibitem{castillo:2014}
\begin{barticle}[author]
\bauthor{\bsnm{Castillo},~\bfnm{Isma{\"e}l}\binits{I.}} \AND
  \bauthor{\bsnm{Nickl},~\bfnm{Richard}\binits{R.}}
(\byear{2014}).
\btitle{On the {B}ernstein--von {M}ises phenomenon for nonparametric {B}ayes
  procedures}.
\bjournal{Annals of Statistics}
\bvolume{42}
\bpages{1941--1969}.
\end{barticle}
\endbibitem

\bibitem{chatterjee:2006}
\begin{barticle}[author]
\bauthor{\bsnm{Chatterjee},~\bfnm{Sourav}\binits{S.}}
(\byear{2006}).
\btitle{A Generalization of the Lindeberg Principle}.
\bjournal{The Annals of Probability}
\bvolume{34}
\bpages{2061--2076}.
\end{barticle}
\endbibitem

\bibitem{chatterjee-meckes:2008}
\begin{barticle}[author]
\bauthor{\bsnm{Chatterjee},~\bfnm{Sourav}\binits{S.}} \AND
  \bauthor{\bsnm{Meckes},~\bfnm{Elizabeth}\binits{E.}}
(\byear{2008}).
\btitle{Multivariate normal approximation using exchangeable pairs}.
\bjournal{ALEA Latin American Journal of Probability and Mathematical
  Statistics}
\bvolume{4}
\bpages{257--283}.
\end{barticle}
\endbibitem

\bibitem{chen:2022_finite-rank}
\begin{barticle}[author]
\bauthor{\bsnm{Chen},~\bfnm{Hongbin}\binits{H.}},
  \bauthor{\bsnm{Mourrat},~\bfnm{Jean-Christophe}\binits{J.-C.}} \AND
  \bauthor{\bsnm{Xia},~\bfnm{Jiaming}\binits{J.}}
(\byear{2022}).
\btitle{Statistical inference of finite-rank tensors}.
\bjournal{Annales Henri Lebesgue}
\bvolume{5}
\bpages{1161--1189}.
\end{barticle}
\endbibitem

\bibitem{chen:2019spin}
\begin{barticle}[author]
\bauthor{\bsnm{Chen},~\bfnm{Yu-Ting}\binits{Y.-T.}}
(\byear{2019}).
\btitle{Universality of {G}hirlanda--{G}uerra identities and spin distributions
  in mixed {$p$}-spin models}.
\bjournal{Annales de l'Institut Henri Poincar\'e, Probabilit\'es et
  Statistiques}
\bvolume{55}
\bpages{528--550}.
\bdoi{10.1214/18-AIHP890}
\end{barticle}
\endbibitem

\bibitem{chernozhukov:2013}
\begin{barticle}[author]
\bauthor{\bsnm{Chernozhukov},~\bfnm{Victor}\binits{V.}},
  \bauthor{\bsnm{Chetverikov},~\bfnm{Denis}\binits{D.}} \AND
  \bauthor{\bsnm{Kato},~\bfnm{Kengo}\binits{K.}}
(\byear{2013}).
\btitle{Gaussian approximations and multiplier bootstrap for maxima of sums of
  high-dimensional random vectors}.
\bjournal{Annals of Statistics}
\bvolume{41}
\bpages{2786--2819}.
\end{barticle}
\endbibitem

\bibitem{chernozhukov:2017}
\begin{barticle}[author]
\bauthor{\bsnm{Chernozhukov},~\bfnm{Victor}\binits{V.}},
  \bauthor{\bsnm{Chetverikov},~\bfnm{Denis}\binits{D.}} \AND
  \bauthor{\bsnm{Kato},~\bfnm{Kengo}\binits{K.}}
(\byear{2017}).
\btitle{Central limit theorems and bootstrap in high dimensions}.
\bjournal{Annals of Probability}
\bvolume{45}
\bpages{2309--2352}.
\end{barticle}
\endbibitem

\bibitem{deshpande:2015a}
\begin{bmisc}[author]
\bauthor{\bsnm{Deshpande},~\bfnm{Yash}\binits{Y.}},
  \bauthor{\bsnm{Abbe},~\bfnm{Emmanuel}\binits{E.}} \AND
  \bauthor{\bsnm{Montanari},~\bfnm{Andrea}\binits{A.}}
(\byear{2015}).
\btitle{Asymptotic Mutual Information for the Two-Groups Stochastic Block
  Model}.
\bnote{[Online]. {A}vailable \url{https://arxiv.org/abs/1507.08685}}.
\end{bmisc}
\endbibitem

\bibitem{dobriban:2018}
\begin{barticle}[author]
\bauthor{\bsnm{Dobriban},~\bfnm{Edgar}\binits{E.}} \AND
  \bauthor{\bsnm{Wager},~\bfnm{Stefan}\binits{S.}}
(\byear{2018}).
\btitle{High-dimensional asymptotics of prediction: Ridge regression and
  classification}.
\bjournal{Annals of Statistics}
\bvolume{46}
\bpages{247--279}.
\end{barticle}
\endbibitem

\bibitem{erdos-yau-yin:2012}
\begin{barticle}[author]
\bauthor{\bsnm{Erd{\H{o}}s},~\bfnm{L{\'a}szl{\'o}}\binits{L.}},
  \bauthor{\bsnm{Yau},~\bfnm{Horng-Tzer}\binits{H.-T.}} \AND
  \bauthor{\bsnm{Yin},~\bfnm{Jun}\binits{J.}}
(\byear{2012}).
\btitle{Bulk universality for generalized {W}igner matrices}.
\bjournal{Probability Theory and Related Fields}
\bvolume{154}
\bpages{341--407}.
\bdoi{10.1007/s00440-011-0390-3}
\end{barticle}
\endbibitem

\bibitem{ghosal:2000convergence}
\begin{barticle}[author]
\bauthor{\bsnm{Ghosal},~\bfnm{Subhashis}\binits{S.}},
  \bauthor{\bsnm{Ghosh},~\bfnm{Jayanta~K}\binits{J.~K.}} \AND
  \bauthor{\bsnm{Van Der~Vaart},~\bfnm{Aad~W}\binits{A.~W.}}
(\byear{2000}).
\btitle{Convergence rates of posterior distributions}.
\bjournal{Annals of Statistics}
\bpages{500--531}.
\end{barticle}
\endbibitem

\bibitem{ghosal:2017}
\begin{bbook}[author]
\bauthor{\bsnm{Ghosal},~\bfnm{Subhashis}\binits{S.}} \AND
  \bauthor{\bparticle{van~der} \bsnm{Vaart},~\bfnm{Aad~W.}\binits{A.~W.}}
(\byear{2017}).
\btitle{Fundamentals of Nonparametric {B}ayesian Inference}.
\bpublisher{Cambridge University Press}.
\end{bbook}
\endbibitem

\bibitem{goldt:2022}
\begin{binproceedings}[author]
\bauthor{\bsnm{Goldt},~\bfnm{Sebastian}\binits{S.}},
  \bauthor{\bsnm{Loureiro},~\bfnm{Bruno}\binits{B.}},
  \bauthor{\bsnm{Reeves},~\bfnm{Galen}\binits{G.}},
  \bauthor{\bsnm{Krzakala},~\bfnm{Florent}\binits{F.}},
  \bauthor{\bsnm{M{\'e}zard},~\bfnm{Marc}\binits{M.}} \AND
  \bauthor{\bsnm{Zdeborov{\'a}},~\bfnm{Lenka}\binits{L.}}
(\byear{2022}).
\btitle{The {G}aussian equivalence of generative models for learning with
  shallow neural networks}.
In \bbooktitle{Mathematical and Scientific Machine Learning (MSML)}.
\bseries{Proceedings of Machine Learning Research}
\bvolume{145}
\bpages{426--471}.
\end{binproceedings}
\endbibitem

\bibitem{guerra:2003}
\begin{barticle}[author]
\bauthor{\bsnm{Guerra},~\bfnm{Francesco}\binits{F.}}
(\byear{2003}).
\btitle{Broken Replica Symmetry Bounds in the Mean Field Spin Glass Model}.
\bjournal{Communications in Mathematical Physics}
\bvolume{233}
\bpages{1--12}.
\end{barticle}
\endbibitem

\bibitem{guionnet2025estimating}
\begin{barticle}[author]
\bauthor{\bsnm{Guionnet},~\bfnm{Alice}\binits{A.}},
  \bauthor{\bsnm{Ko},~\bfnm{Justin}\binits{J.}},
  \bauthor{\bsnm{Krzakala},~\bfnm{Florent}\binits{F.}} \AND
  \bauthor{\bsnm{Zdeborov{\'a}},~\bfnm{Lenka}\binits{L.}}
(\byear{2025}).
\btitle{Estimating Rank-One Matrices with Mismatched Prior and Noise:
  Universality and Large Deviations}.
\bjournal{Communications in Mathematical Physics}
\bvolume{406}
\bpages{9}.
\end{barticle}
\endbibitem

\bibitem{guionnet:2025}
\begin{barticle}[author]
\bauthor{\bsnm{Guionnet},~\bfnm{Alice}\binits{A.}},
  \bauthor{\bsnm{Ko},~\bfnm{Justin}\binits{J.}},
  \bauthor{\bsnm{Krzakala},~\bfnm{Florent}\binits{F.}} \AND
  \bauthor{\bsnm{Zdeborov{\'a}},~\bfnm{Lenka}\binits{L.}}
(\byear{2025}).
\btitle{Low-rank matrix estimation with inhomogeneous noise}.
\bjournal{Information and Inference: A Journal of the IMA}
\bvolume{14}
\bpages{iaaf010}.
\end{barticle}
\endbibitem

\bibitem{hairer:2025spde}
\begin{bmisc}[author]
\bauthor{\bsnm{Hairer},~\bfnm{Martin}\binits{M.}}
(\byear{2025}).
\btitle{An Introduction to Stochastic PDEs}.
\bnote{Lecture notes}.
\end{bmisc}
\endbibitem

\bibitem{han:2023}
\begin{barticle}[author]
\bauthor{\bsnm{Han},~\bfnm{Qiyang}\binits{Q.}} \AND
  \bauthor{\bsnm{Shen},~\bfnm{Yandi}\binits{Y.}}
(\byear{2023}).
\btitle{Universality of regularized regression estimators in high dimensions}.
\bjournal{The Annals of Statistics}
\bvolume{51}
\bpages{1799--1823}.
\end{barticle}
\endbibitem

\bibitem{hastie:2022}
\begin{barticle}[author]
\bauthor{\bsnm{Hastie},~\bfnm{Trevor}\binits{T.}},
  \bauthor{\bsnm{Montanari},~\bfnm{Andrea}\binits{A.}},
  \bauthor{\bsnm{Rosset},~\bfnm{Saharon}\binits{S.}} \AND
  \bauthor{\bsnm{Tibshirani},~\bfnm{Ryan~J.}\binits{R.~J.}}
(\byear{2022}).
\btitle{Surprises in High-Dimensional Ridgeless Least Squares Interpolation}.
\bjournal{The Annals of Statistics}
\bvolume{50}
\bpages{949--986}.
\end{barticle}
\endbibitem

\bibitem{hu:2023}
\begin{barticle}[author]
\bauthor{\bsnm{Hu},~\bfnm{Hong}\binits{H.}} \AND
  \bauthor{\bsnm{Lu},~\bfnm{Yue~M.}\binits{Y.~M.}}
(\byear{2023}).
\btitle{Universality laws for high-dimensional learning with random features}.
\bjournal{IEEE Transactions on Information Theory}
\bvolume{69}
\bpages{1932--1964}.
\end{barticle}
\endbibitem

\bibitem{johnstone:2001a}
\begin{barticle}[author]
\bauthor{\bsnm{Johnstone},~\bfnm{Iain~M.}\binits{I.~M.}}
(\byear{2001}).
\btitle{On the distribution of the largest eigenvalue in principal components
  analysis}.
\bjournal{Annals of Statistics}
\bvolume{29}
\bpages{295--327}.
\end{barticle}
\endbibitem

\bibitem{kleijn:2012}
\begin{barticle}[author]
\bauthor{\bsnm{Kleijn},~\bfnm{Bas J.~K.}\binits{B.~J.~K.}} \AND
  \bauthor{\bparticle{van~der} \bsnm{Vaart},~\bfnm{Aad~W.}\binits{A.~W.}}
(\byear{2012}).
\btitle{The {B}ernstein-{V}on-{M}ises theorem under misspecification}.
\bjournal{Electronic Journal of Statistics}
\bvolume{6}
\bpages{354--381}.
\end{barticle}
\endbibitem

\bibitem{korada:2010}
\begin{barticle}[author]
\bauthor{\bsnm{Korada},~\bfnm{Satish~Babu}\binits{S.~B.}} \AND
  \bauthor{\bsnm{Macris},~\bfnm{Nicolas}\binits{N.}}
(\byear{2010}).
\btitle{Tight Bounds on the Capicty of Binary Input Random {CDMA} Systems}.
\bjournal{IEEE Transactions on Information Theory}
\bvolume{56}
\bpages{5590-5613}.
\end{barticle}
\endbibitem

\bibitem{korada:2011}
\begin{barticle}[author]
\bauthor{\bsnm{Korada},~\bfnm{Satish~Babu}\binits{S.~B.}} \AND
  \bauthor{\bsnm{Montanari},~\bfnm{Andrea}\binits{A.}}
(\byear{2011}).
\btitle{Applications of the {L}indeberg Principle in Communications and
  Statistical Learning}.
\bjournal{IEEE Transactions on Information Theory}
\bvolume{57}
\bpages{2011}.
\end{barticle}
\endbibitem

\bibitem{kreyszig1991introductory}
\begin{bbook}[author]
\bauthor{\bsnm{Kreyszig},~\bfnm{Erwin}\binits{E.}}
(\byear{1991}).
\btitle{Introductory functional analysis with applications}.
\bpublisher{John Wiley \& Sons}.
\end{bbook}
\endbibitem

\bibitem{krzakala:2016}
\begin{binproceedings}[author]
\bauthor{\bsnm{Krzakala},~\bfnm{Florent}\binits{F.}},
  \bauthor{\bsnm{Xu},~\bfnm{Jiaming}\binits{J.}} \AND
  \bauthor{\bsnm{Zdeborov\'{a}},~\bfnm{Lenka}\binits{L.}}
(\byear{2016}).
\btitle{Mutual Information in Rank-One Matrix Estimation}
In \bbooktitle{Proceedings of the IEEE Information Theory Workshop (ITW)}.
\end{binproceedings}
\endbibitem

\bibitem{lecam:1986}
\begin{bbook}[author]
\bauthor{\bsnm{Le~Cam},~\bfnm{Lucien}\binits{L.}}
(\byear{1986}).
\btitle{Asymptotic Methods in Statistical Decision Theory}.
\bpublisher{Springer}.
\end{bbook}
\endbibitem

\bibitem{lecam:2000}
\begin{bbook}[author]
\bauthor{\bsnm{Le~Cam},~\bfnm{Lucien}\binits{L.}} \AND
  \bauthor{\bsnm{Yang},~\bfnm{Grace~Lo}\binits{G.~L.}}
(\byear{2000}).
\btitle{Asymptotics in Statistics: Some Basic Concepts}.
\bpublisher{Springer-Verlag}, \baddress{New York}.
\end{bbook}
\endbibitem

\bibitem{lehmann:1998}
\begin{bbook}[author]
\bauthor{\bsnm{Lehmann},~\bfnm{E.~L.}\binits{E.~L.}} \AND
  \bauthor{\bsnm{Casella},~\bfnm{George}\binits{G.}}
(\byear{1998}).
\btitle{Theory of Point Estimation},
\bedition{Second} ed.
\bpublisher{Springer}.
\end{bbook}
\endbibitem

\bibitem{lelarge:2018}
\begin{barticle}[author]
\bauthor{\bsnm{Lelarge},~\bfnm{Marc}\binits{M.}} \AND
  \bauthor{\bsnm{Miolane},~\bfnm{L\'eo}\binits{L.}}
(\byear{2019}).
\btitle{Fundamental limits of symmetric low-rank matrix estimation}.
\bjournal{Probability Theory and Related Fields}
\bvolume{173}
\bpages{859--929}.
\bdoi{10.1007/s00440-018-0845-x}
\end{barticle}
\endbibitem

\bibitem{lesieur:2017a}
\begin{bmisc}[author]
\bauthor{\bsnm{Lesieur},~\bfnm{Thibault}\binits{T.}},
  \bauthor{\bsnm{Miolane},~\bfnm{L/'eo}\binits{L.}},
  \bauthor{\bsnm{Lelarge},~\bfnm{Marc}\binits{M.}},
  \bauthor{\bsnm{Krzakala},~\bfnm{Florent}\binits{F.}} \AND
  \bauthor{\bsnm{Zdeborov\'a},~\bfnm{Lenka}\binits{L.}}
(\byear{2017}).
\btitle{Statistical and computational phase transitions in spiked tensor
  estimation}.
\bnote{[Online]. {A}vailable \url{https://arxiv.org/abs/1701.08010}}.
\end{bmisc}
\endbibitem

\bibitem{luneau:2021}
\begin{barticle}[author]
\bauthor{\bsnm{Luneau},~\bfnm{Cl{\'e}ment}\binits{C.}},
  \bauthor{\bsnm{Barbier},~\bfnm{Jean}\binits{J.}} \AND
  \bauthor{\bsnm{Macris},~\bfnm{Nicolas}\binits{N.}}
(\byear{2021}).
\btitle{Mutual information for low-rank even-order symmetric tensor
  estimation}.
\bjournal{Information and Inference: A Journal of the IMA}
\bvolume{10}
\bpages{1167--1207}.
\end{barticle}
\endbibitem

\bibitem{mergny:2024}
\begin{barticle}[author]
\bauthor{\bsnm{Mergny},~\bfnm{Pierre}\binits{P.}},
  \bauthor{\bsnm{Ko},~\bfnm{Justin}\binits{J.}},
  \bauthor{\bsnm{Krzakala},~\bfnm{Florent}\binits{F.}} \AND
  \bauthor{\bsnm{Zdeborov{\'a}},~\bfnm{Lenka}\binits{L.}}
(\byear{2024}).
\btitle{Fundamental limits of non-linear low-rank matrix estimation}.
\bjournal{arXiv preprint arXiv:2403.04234}.
\end{barticle}
\endbibitem

\bibitem{mezard:2009}
\begin{bbook}[author]
\bauthor{\bsnm{M\'{e}zard},~\bfnm{Mark}\binits{M.}} \AND
  \bauthor{\bsnm{Montanari},~\bfnm{Andrea}\binits{A.}}
(\byear{2009}).
\btitle{Information, physics, and computation}.
\bpublisher{Oxford University Press}.
\end{bbook}
\endbibitem

\bibitem{miller:2021}
\begin{barticle}[author]
\bauthor{\bsnm{Miller},~\bfnm{Jeffrey~W.}\binits{J.~W.}}
(\byear{2021}).
\btitle{Asymptotic normality, concentration, and coverage of generalized
  posteriors}.
\bjournal{Journal of Machine Learning Research}
\bvolume{22}
\bpages{1--53}.
\end{barticle}
\endbibitem

\bibitem{miolane:2017a}
\begin{bmisc}[author]
\bauthor{\bsnm{Miolane},~\bfnm{L\'eo}\binits{L.}}
(\byear{2017}).
\btitle{Fundamental limits of low-rank matrix estimation: the non-symmetric
  case}.
\bnote{[Online]. {A}vailable \url{https://arxiv.org/abs/1702.00473}}.
\end{bmisc}
\endbibitem

\bibitem{montanari:2022}
\begin{binproceedings}[author]
\bauthor{\bsnm{Montanari},~\bfnm{Andrea}\binits{A.}} \AND
  \bauthor{\bsnm{Saeed},~\bfnm{Basil~N.}\binits{B.~N.}}
(\byear{2022}).
\btitle{Universality of Empirical Risk Minimization}.
In \bbooktitle{Proceedings of the Thirty-Fifth Conference on Learning Theory}.
\bseries{Proceedings of Machine Learning Research}
\bvolume{178}
\bpages{4310--4312}.
\bpublisher{PMLR}.
\end{binproceedings}
\endbibitem

\bibitem{mourrat:2020}
\begin{barticle}[author]
\bauthor{\bsnm{Mourrat},~\bfnm{Jean-Christophe}\binits{J.-C.}}
(\byear{2020}).
\btitle{Hamilton-{J}acobi equations for finite-rank matrix inference}.
\bjournal{The Annals of Applied Probability}
\bvolume{30}
\bpages{2234-2260}.
\end{barticle}
\endbibitem

\bibitem{panchenko:2013book}
\begin{bbook}[author]
\bauthor{\bsnm{Panchenko},~\bfnm{Dmitry}\binits{D.}}
(\byear{2013}).
\btitle{The {S}herrington--{K}irkpatrick Model}.
\bseries{Springer Monographs in Mathematics}.
\bpublisher{Springer}, \baddress{New York}.
\bdoi{10.1007/978-1-4614-6289-7}
\end{bbook}
\endbibitem

\bibitem{perry:2018}
\begin{barticle}[author]
\bauthor{\bsnm{Perry},~\bfnm{Amelia}\binits{A.}},
  \bauthor{\bsnm{Wein},~\bfnm{Alexander~S.}\binits{A.~S.}},
  \bauthor{\bsnm{Bandeira},~\bfnm{Afonso~S.}\binits{A.~S.}} \AND
  \bauthor{\bsnm{Moitra},~\bfnm{Ankur}\binits{A.}}
(\byear{2018}).
\btitle{Optimality and sub-optimality of {PCA} {I}: {S}piked random matrix
  models}.
\bjournal{Annals of Statistics}
\bvolume{46}
\bpages{2416--2451}.
\end{barticle}
\endbibitem

\bibitem{reeves:2020}
\begin{barticle}[author]
\bauthor{\bsnm{Reeves},~\bfnm{Galen}\binits{G.}}
(\byear{2020}).
\btitle{Information-Theoretic Limits for the Matrix Tensor Product}.
\bjournal{IEEE Journal on Selected Areas in Information Theory}
\bvolume{1}
\bpages{777--798}.
\end{barticle}
\endbibitem

\bibitem{reeves:2019ab}
\begin{bmisc}[author]
\bauthor{\bsnm{Reeves},~\bfnm{Galen}\binits{G.}},
  \bauthor{\bsnm{Mayya},~\bfnm{Vaishakhi}\binits{V.}} \AND
  \bauthor{\bsnm{Volfovsky},~\bfnm{Alexander}\binits{A.}}
(\byear{2019}).
\btitle{The Geometry of Community Detection via the MMSE Matrix}.
\bnote{[Online]. {A}vailable \url{https://arxiv.org/pdf/1907.02496.pdf}}.
\end{bmisc}
\endbibitem

\bibitem{reinert-rollin:2009}
\begin{barticle}[author]
\bauthor{\bsnm{Reinert},~\bfnm{Gesine}\binits{G.}} \AND
  \bauthor{\bsnm{R{\"o}llin},~\bfnm{Adrian}\binits{A.}}
(\byear{2009}).
\btitle{Multivariate normal approximation with {Stein}'s method of exchangeable
  pairs under a general linearity condition}.
\bjournal{The Annals of Probability}
\bvolume{37}
\bpages{2150--2173}.
\bdoi{10.1214/09-AOP467}
\end{barticle}
\endbibitem

\bibitem{rollin:2013}
\begin{barticle}[author]
\bauthor{\bsnm{R\"{o}llin},~\bfnm{Adrian}\binits{A.}}
(\byear{2013}).
\btitle{Stein's method in high dimensions with applications}.
\bjournal{Annales de l'IHP Probabilit{\'e}s et statistiques}
\bvolume{49}
\bpages{529--549}.
\end{barticle}
\endbibitem

\bibitem{rossetti:2025}
\begin{barticle}[author]
\bauthor{\bsnm{Rossetti},~\bfnm{Riccardo}\binits{R.}} \AND
  \bauthor{\bsnm{Reeves},~\bfnm{Galen}\binits{G.}}
(\byear{2025}).
\btitle{Statistical Limits for Finite-Rank Tensor Estimation}.
\bjournal{arXiv preprint arXiv:2506.06749}.
\end{barticle}
\endbibitem

\bibitem{rossetti2025fundamental}
\begin{binproceedings}[author]
\bauthor{\bsnm{Rossetti},~\bfnm{Riccardo}\binits{R.}} \AND
  \bauthor{\bsnm{Reeves},~\bfnm{Galen}\binits{G.}}
(\byear{2025}).
\btitle{Fundamental Limits for High-Dimensional Factor Regression Models}.
In \bbooktitle{2025 IEEE International Symposium on Information Theory (ISIT)}
\bpages{1--6}.
\bpublisher{IEEE}.
\end{binproceedings}
\endbibitem

\bibitem{talagrand:2003}
\begin{bbook}[author]
\bauthor{\bsnm{Talagrand},~\bfnm{Michel}\binits{M.}}
(\byear{2003}).
\btitle{Spin Glasses: {A} Challenge for Mathematicians}.
\bpublisher{Springer-Verlag}.
\end{bbook}
\endbibitem

\bibitem{talagrand:2011}
\begin{bbook}[author]
\bauthor{\bsnm{Talagrand},~\bfnm{Michel}\binits{M.}}
(\byear{2011}).
\btitle{Mean Field Models for Spin Glasses, Volumn I: Basic Examples}.
\bpublisher{Springer}, \baddress{Berlin, Heidelberg}.
\end{bbook}
\endbibitem

\bibitem{tao:2011}
\begin{barticle}[author]
\bauthor{\bsnm{Tao},~\bfnm{Terence}\binits{T.}} \AND
  \bauthor{\bsnm{Vu},~\bfnm{Van}\binits{V.}}
(\byear{2011}).
\btitle{Random matrices: Universality of local eigenvalue statistics}.
\bjournal{Acta Math}
\bvolume{206}
\bpages{127--204}.
\end{barticle}
\endbibitem

\bibitem{kadane:1986}
\begin{barticle}[author]
\bauthor{\bsnm{Tierney},~\bfnm{Luke}\binits{L.}} \AND
  \bauthor{\bsnm{Kadane},~\bfnm{Joseph~B.}\binits{J.~B.}}
(\byear{1986}).
\btitle{Accurate Approximations for Posterior Moments and Marginal Densities}.
\bjournal{Journal of the American Statistical Association}
\bvolume{81}
\bpages{82--86}.
\end{barticle}
\endbibitem

\bibitem{vaart:1998}
\begin{bbook}[author]
\bauthor{\bparticle{van~der} \bsnm{Vaart},~\bfnm{A.~W.}\binits{A.~W.}}
(\byear{1998}).
\btitle{Asymptotic Statistics}.
\bpublisher{Cambridge University Press}.
\end{bbook}
\endbibitem

\bibitem{zou:2026}
\begin{binproceedings}[author]
\bauthor{\bsnm{Zou},~\bfnm{Wenxuan}\binits{W.}} \AND
  \bauthor{\bsnm{Reeves},~\bfnm{Galen}\binits{G.}}
(\byear{2026}).
\btitle{Free Energy Universality in Tensor Estimation via Generic Chaining}.
In \bbooktitle{2026 IEEE International Symposium on Information Theory (ISIT)
  Workshop}.
\bnote{To appear; arXiv:2605.30636}.
\end{binproceedings}
\endbibitem

\end{thebibliography}

\begin{appendix}
\section{Proofs for Section~\ref{sec:comparison}}
\label{app:proofs_section3}

This appendix proves the perturbation identities and comparison results
stated in Section~\ref{sec:comparison}. We proceed in two stages.
Section~\ref{app:proof_prelim} first develops the calculus of Fr\'echet
derivatives on $\cX$ and proves a two-element approximate
integration-by-parts identity for a generic Fr\'echet-differentiable
functional. It then introduces the log-partition functional $f$, establishes
its smoothness properties, and records the exact two-element Gaussian
integration-by-parts identity needed below; the latter identity is specific
to $f$. Sections~\ref{app:proof_deri_FG}--\ref{app:proof_fe_comparison} then
apply these tools to establish, in turn,
Lemma~\ref{lem:deri_FG}, Lemma~\ref{lem:deri_F}, and
Lemma~\ref{lem:linear_perturb_G}, followed by
Theorem~\ref{thm:fe_comparison}.

\subsection{Preliminary Lemmas}
\label{app:proof_prelim}

Throughout, $\cX=C(T;\bbR)$ denotes the real Banach space of continuous
functions on the compact metric space $T$ of
Section~\ref{sec:general_frame}, equipped with the supremum norm
$\|x\|_\infty=\sup_{t\in T}|x(t)|$. Since $T$ is metrizable, $\cX$ is
separable by~\cite[Theorem~4.1.3]{albiac:2006banach}. By the
Riesz--Markov--Kakutani representation
theorem, the topological dual $\cX^*$ is identified with the space of finite
signed regular Borel measures on $T$, with duality pairing
$\langle x,\nu\rangle=\int_T x(t)\,\nu(\dd t)$ for $x\in\cX$, $\nu\in\cX^*$.

We use the following notation for Fr\'echet derivatives. Let $V$ and $W$ be
separable Banach spaces, and let $G:W\to V$ be $k$ times Fr\'echet
differentiable at $x\in W$. Then $D^kG(x)$ is a bounded $k$-linear map from
$W^k$ to $V$, with operator norm
\begin{equation}\label{eq:frechet_derivative_operator_norm}
\|D^k G(x)\|_{\mathrm{op}}
\coloneqq
\sup_{\substack{\|h_1\|_W\le1,\ldots,\|h_k\|_W\le1}}
\bigl\|D^k G(x)[h_1,\ldots,h_k]\bigr\|_V.
\end{equation}
We use the convention that $D^0$ is the identity operator, so that
$D^0G=G$. When $W=\cX$ and $V=\bbR$, so that $G:\cX\to\bbR$, the derivative
$DG(x)$ at each $x\in\cX$ is an element of $\cX^*$, and is therefore itself a
finite signed regular Borel measure on $T$; we write
\[
DG(x)[h]
=
\int_T h(t)\,(DG(x))(\dd t),
\qquad h\in\cX,
\]
for the corresponding integral representation of the directional derivative.
Background on Gaussian measures on $\cX$, including the Cameron--Martin space
and the abstract Gaussian integration-by-parts formula, is recorded in
Appendix~\ref{app:gaussian_measures_banach}.

For a bounded bilinear form $B$ on $\cX\times\cX$ and a covariance kernel
$K_{12}$ on $T\times T$, we write $\langle B,K_{12}\rangle$ for the
contraction of $B$ with $K_{12}$, defined by
$\langle B,K_{12}\rangle\coloneqq\E[B(V,W)]$ for any centered random elements
$V,W$ of $\cX$ with $\cov(V(s),W(t))=K_{12}(s,t)$.

The following two-element approximate integration-by-parts identity applies
to a generic Fr\'echet-differentiable functional and will be used for the
non-Gaussian terms below.

\begin{lemma}[Two-element approximate Gaussian integration by parts]
\label{lem:approx_ibp_two}
Let $X_1,X_2$ be random elements of $\cX$ with means $m_1,m_2\in\cX$,
cross-covariance kernel $K_{12}(s,t)=\cov(X_1(s),X_2(t))$, and
$\E\|X_1\|_\infty^3,\E\|X_2\|_\infty^3<\infty$. Let $F:\cX\to\bbR$ be
three-times continuously Fr\'echet differentiable with
$M_3(F)\coloneqq\sup_{z\in\cX}\|\sD^3F(z)\|_{\mathrm{op}}<\infty$. Then
\[
\E\bigl[\sD F(X_1)[X_2-m_2]\bigr]
=
\bigl\langle\E[\sD^2F(X_1)],K_{12}\bigr\rangle
+
R,
\]
where
$|R|\le\tfrac32M_3(F)\bigl(\E\|X_1-m_1\|_\infty^3\bigr)^{2/3}\bigl(\E\|X_2-m_2\|_\infty^3\bigr)^{1/3}$.
\end{lemma}

\begin{proof}
Write $\bar X_1\coloneqq X_1-m_1$, $\bar X_2\coloneqq X_2-m_2$. Let $X_1'$ be
an independent copy of $X_1$, independent of $(X_1,X_2)$, write
$\bar X_1'\coloneqq X_1'-m_1$, and set
$Z_{r,s}\coloneqq m_1+rs\bar X_1+(1-r)\bar X_1'$ for $r,s\in[0,1]$, so that
$Z_{1,s}=m_1+s\bar X_1$
and $Z_{0,s}=X_1'$ for every $s\in[0,1]$.

Before applying the fundamental theorem of calculus, we record the required
integrability. The bound on $D^3F$ and the mean-value inequality give, for
every $x\in\cX$,
\begin{align*}
\|\sD^2F(x)\|_{\mathrm{op}}
&\le
\|\sD^2F(m_1)\|_{\mathrm{op}}
+M_3(F)\|x-m_1\|_\infty,\\
\|\sD F(x)\|_{\mathrm{op}}
&\le
\|\sD F(m_1)\|_{\mathrm{op}}
+\|\sD^2F(m_1)\|_{\mathrm{op}}\|x-m_1\|_\infty
+\frac12M_3(F)\|x-m_1\|_\infty^2.
\end{align*}
Together with the assumed third moments, H\"older's inequality shows that
$\E\|\sD^2F(X_1)\|_{\mathrm{op}}<\infty$ and
$\E|\sD F(X_1)[\bar X_2]|<\infty$. The same bounds justify all subsequent
exchanges of expectation with the $r$- and $s$-integrals.

Since $F$ is twice continuously Fr\'echet differentiable and
$\bar X_2$ does not depend on $z$,
\[
\sD F(X_1)[\bar X_2]-\sD F(m_1)[\bar X_2]
=
\int_0^1\sD^2F(Z_{1,s})[\bar X_1,\bar X_2]\,\dd s.
\]
Taking expectations and using $\E[\bar X_2]=0$, so that
$\sD F(m_1)[\E\bar X_2]=0$,
\[
\E\bigl[\sD F(X_1)[\bar X_2]\bigr]
=
\int_0^1
\E\bigl[\sD^2F(Z_{1,s})[\bar X_1,\bar X_2]\bigr]\,\dd s.
\]
Differentiating $r\mapsto\sD^2F(Z_{r,s})[\bar X_1,\bar X_2]$ and integrating
from $r=0$ to
$r=1$, using $Z_{0,s}=X_1'$,
\[
\sD^2F(Z_{1,s})[\bar X_1,\bar X_2]
=
\sD^2F(X_1')[\bar X_1,\bar X_2]
+
\int_0^1
\sD^3F(Z_{r,s})[s\bar X_1-\bar X_1',\bar X_1,\bar X_2]\,\dd r,
\]
and the first term on the right does not depend on $s$. Since $X_1'$ is
independent of $(\bar X_1,\bar X_2)$ and has the same law as $X_1$,
\[
\E\bigl[\sD^2F(X_1')[\bar X_1,\bar X_2]\bigr]
=
\E\bigl[\langle\sD^2F(X_1'),K_{12}\rangle\bigr]
=
\bigl\langle\E[\sD^2F(X_1)],K_{12}\bigr\rangle.
\]
Combining the displays above establishes the asserted identity with
\[
R
=
\int_0^1\int_0^1
\E\bigl[\sD^3F(Z_{r,s})[s\bar X_1-\bar X_1',\bar X_1,\bar X_2]\bigr]\,\dd r\,\dd s.
\]

For the bound, $\|\sD^3F(z)\|_{\mathrm{op}}\le M_3(F)$ for every $z$ gives
\begin{align*}
|R|
&\le
M_3(F)
\int_0^1\int_0^1
\E\bigl[\bigl(s\|\bar X_1\|_\infty+\|\bar X_1'\|_\infty\bigr)\|\bar X_1\|_\infty\|\bar X_2\|_\infty\bigr]\,\dd r\,\dd s
\\
&=
M_3(F)
\left(
\frac12\E\bigl[\|\bar X_1\|_\infty^2\|\bar X_2\|_\infty\bigr]
+
\E\bigl[\|\bar X_1'\|_\infty\bigr]\,\E\bigl[\|\bar X_1\|_\infty\|\bar X_2\|_\infty\bigr]
\right),
\end{align*}
using independence of $X_1'$ from $(\bar X_1,\bar X_2)$. By H\"older's
inequality, $\E[\|\bar X_1\|_\infty^2\|\bar X_2\|_\infty]\le(\E\|\bar X_1\|_\infty^3)^{2/3}(\E\|\bar X_2\|_\infty^3)^{1/3}$,
and, by Cauchy--Schwarz followed by the power-mean inequality,
$\E[\|\bar X_1\|_\infty\|\bar X_2\|_\infty]\le(\E\|\bar X_1\|_\infty^2)^{1/2}(\E\|\bar X_2\|_\infty^2)^{1/2}
\le(\E\|\bar X_1\|_\infty^3)^{1/3}(\E\|\bar X_2\|_\infty^3)^{1/3}$,
while $\E[\|\bar X_1'\|_\infty]\le(\E\|\bar X_1\|_\infty^3)^{1/3}$ since $X_1'$
has the same law as $X_1$; both terms above are therefore bounded by
$(\E\|\bar X_1\|_\infty^3)^{2/3}(\E\|\bar X_2\|_\infty^3)^{1/3}$, so
\begin{align*}
|R|
&\le
\Bigl(\frac12+1\Bigr)M_3(F)
\bigl(\E\|X_1-m_1\|_\infty^3\bigr)^{2/3}
\bigl(\E\|X_2-m_2\|_\infty^3\bigr)^{1/3}
\\
&=
\frac32M_3(F)
\bigl(\E\|X_1-m_1\|_\infty^3\bigr)^{2/3}
\bigl(\E\|X_2-m_2\|_\infty^3\bigr)^{1/3}.
\end{align*}
\end{proof}

Having established the generic approximate identity, we now introduce the
log-partition functional $f:\cX\to\bbR$ defined by
\begin{equation}\label{eq:log_partition_functional}
f(x)
\coloneqq
\log\int_T e^{x(t)}\pi(\dd t),
\qquad x\in\cX.
\end{equation}
Its associated Gibbs measure is
$\pi_x(\dd t)=e^{x(t)}\pi(\dd t)/\int_T e^{x(s)}\pi(\dd s)$.

\begin{lemma}[Smoothness of the log-partition functional]
\label{lem:logpartition_smooth}
The functional $f$ defined in~\eqref{eq:log_partition_functional} is
infinitely Fr\'echet differentiable on $\cX$, and for every integer $k\ge1$,
$x\in\cX$, and $h_1,\ldots,h_k\in\cX$,
\[
D^kf(x)[h_1,\ldots,h_k]
=
\kappa_{\pi_x}(h_1,\ldots,h_k),
\]
the $k$-th joint cumulant of $h_1,\ldots,h_k$, viewed as random variables on
$T$, under $\pi_x$. In particular, $Df(x)[h]=\int_T h(t)\,\pi_x(\dd t)$ and
\[
\|Df(x)\|_{\mathrm{op}}\le1,
\qquad
\|D^2f(x)\|_{\mathrm{op}}\le1,
\qquad
\|D^3f(x)\|_{\mathrm{op}}\le2,
\qquad x\in\cX.
\]
\end{lemma}

\begin{proof}
\emph{Step 1: Fr\'echet smoothness.}
Consider the pointwise exponential map
$\operatorname{Exp}:\cX\to\cX$, $\operatorname{Exp}(x)=e^x$. Repeated
pointwise differentiation suggests that its $k$-th derivative should be
multiplication by $e^x$ and the $k$ directions. Accordingly, for $k\ge0$,
define the candidate family
\[
A_k(x)[h_1,\ldots,h_k]
\coloneqq
e^x\prod_{i=1}^k h_i,
\]
with $A_0=\operatorname{Exp}$. For $k\ge1$, each $A_k(x)$ is a bounded
$k$-linear map with
$\|A_k(x)\|_{\mathrm{op}}\le e^{\|x\|_\infty}$.
Moreover, the scalar bound
$|e^v-1-v|\le \frac12e^{|v|}v^2$ gives
\[
\bigl\|A_k(x+u)-A_k(x)-A_{k+1}(x)[u,\cdot]\bigr\|_{\mathrm{op}}
\le
\frac12e^{\|x\|_\infty+\|u\|_\infty}\|u\|_\infty^2.
\]
Thus, $A_k$ is Fr\'echet differentiable with
$DA_k(x)[u]=A_{k+1}(x)[u,\cdot]$. Starting from
$A_0=\operatorname{Exp}$, induction shows that
$D^k\operatorname{Exp}=A_k$ for every $k$. Hence,
$\operatorname{Exp}$ is infinitely Fr\'echet differentiable. Since
$L(g)\coloneqq\int_Tg(t)\,\pi(\dd t)$ is continuous and linear,
$L(e^x)>0$, and the scalar logarithm is smooth on $(0,\infty)$, the
composition $f=\log\circ L\circ\operatorname{Exp}$ is infinitely Fr\'echet
differentiable.

\emph{Step 2: Identification of the derivatives.}
Fix $x,h_1,\ldots,h_k\in\cX$ and define
$\Lambda:\bbR^k\to\bbR$ by
$\Lambda(\lambda)=f(x+\sum_{i=1}^k\lambda_i h_i)-f(x)$. Then
\[
\Lambda(\lambda)
=
\log\E_{t\sim\pi_x}
\exp\Bigl\{\textstyle\sum_{i=1}^k\lambda_i h_i(t)\Bigr\}.
\]
Since the $h_i$ are bounded, the displayed expression is finite and
real-analytic on $\bbR^k$; it is the joint cumulant generating function of
$(h_1(t),\ldots,h_k(t))_{t\sim\pi_x}$. Hence, its mixed derivative at the
origin is the corresponding joint cumulant. The chain rule applied to the
left-hand side therefore gives
$D^kf(x)[h_1,\ldots,h_k]=\kappa_{\pi_x}(h_1,\ldots,h_k)$.

\emph{Step 3: Operator-norm bounds.}
For the stated bounds, let $h_1,h_2,h_3\in\cX$ with $\|h_i\|_\infty\le1$, and
write $\bar h_i\coloneqq h_i-\int_T h_i\,\dd\pi_x$ for the centered version.
Since $|\int_T h_i\,\dd\pi_x|\le\|h_i\|_\infty\le1$, we have
$\|\bar h_i\|_\infty\le2\|h_i\|_\infty\le2$, and $\bar h_i$ has
$\pi_x$-variance $\var_{\pi_x}(h_i)\le1$, since $h_i$ takes values in an
interval of length at most $2$. First, $|Df(x)[h_1]|=|\int_T h_1\,\dd\pi_x|\le1$.
Next, by Cauchy--Schwarz,
\[
|D^2f(x)[h_1,h_2]|
=
\Bigl|\int_T\bar h_1\bar h_2\,\dd\pi_x\Bigr|
\le
\sqrt{\var_{\pi_x}(h_1)\var_{\pi_x}(h_2)}
\le1.
\]
Finally, applying Cauchy--Schwarz once more,
\begin{align*}
|D^3f(x)[h_1,h_2,h_3]|
&=
\Bigl|\int_T\bar h_1\bar h_2\bar h_3\,\dd\pi_x\Bigr|
\le
\|\bar h_3\|_\infty
\int_T|\bar h_1\bar h_2|\,\dd\pi_x
\\
&\le
\|\bar h_3\|_\infty
\sqrt{\var_{\pi_x}(h_1)\var_{\pi_x}(h_2)}
\le2.
\end{align*}
Taking the supremum over $\|h_i\|_\infty\le1$ gives the three stated operator
norm bounds.
\end{proof}

The exact two-element Gaussian integration-by-parts identity needed below is
available for the log-partition functional $f$. It is proved as
Lemma~\ref{lem:exact_gibp_two_logpartition} in
Appendix~\ref{app:gibp_separable_banach}.

\begin{lemma}[Exact two-element Gaussian integration by parts for the
log-partition functional]
\label{lem:exact_gibp_two}
Let $Y_1,Y_2$ be jointly Gaussian elements of $\cX$ with means
$m_1,m_2\in\cX$ and cross-covariance kernel
$K_{12}(s,t)\coloneqq\cov(Y_1(s),Y_2(t))$. For the log-partition functional
$f$ defined in~\eqref{eq:log_partition_functional},
\[
\E\bigl[\sD f(Y_1)[Y_2]\bigr]
=
\E[\sD f(Y_1)][m_2]
+
\bigl\langle\E[\sD^2f(Y_1)],K_{12}\bigr\rangle.
\]
\end{lemma}

\begin{proof}
This is Lemma~\ref{lem:exact_gibp_two_logpartition} in
Appendix~\ref{app:gibp_separable_banach}.
\end{proof}

For fixed $x\in\cX$ and a covariance kernel $K_{12}$ on $T\times T$, taking
centered $(V,W)$ with cross-covariance $K_{12}$ and using
$D^2f(x)[h,h']=\cov_{\pi_x}(h,h')$ from Lemma~\ref{lem:logpartition_smooth},
\begin{equation}\label{eq:contraction_identity}
\begin{aligned}
\langle D^2f(x),K_{12}\rangle
&=
\E\bigl[D^2f(x)[V,W]\bigr]
\\
&=
\int_T K_{12}(t,t)\,\pi_x(\dd t)
-
\iint_{T^2}K_{12}(s,t)\,\pi_x(\dd s)\,\pi_x(\dd t),
\end{aligned}
\end{equation}
since $\pi_x$ is a fixed (non-random) measure and Fubini's theorem allows
$\E$ to be exchanged with the $\pi_x$-integrals. This identity, applied with
$x$ replaced by a random element and an outer expectation taken, is used
repeatedly below.

\subsection{Proof of Lemma~\ref{lem:deri_FG}}
\label{app:proof_deri_FG}

\begin{proof}
Recall that $f$ is the log-partition functional defined
in~\eqref{eq:log_partition_functional}. Fix $\beta\in\cB$. Under the assumed
joint Gaussianity and integrability conditions, we proceed as follows. By
Lemma~\ref{lem:logpartition_smooth},
$\|\sD f\|_{\mathrm{op}}\le1$, so $f$ is $1$-Lipschitz. Hence, for every
nonzero $u$ such that $\beta+u\in\cB$, the mean-value inequality gives
\[
\left|
\frac{f(Y_{\beta+u})-f(Y_\beta)}{u}
\right|
\le
\left\|
\frac{Y_{\beta+u}-Y_\beta}{u}
\right\|_\infty
\le
\sup_{r\in\cB}\|\dot Y_r\|_\infty.
\]
The right-hand side is integrable by assumption. Moreover, setting
$\Delta_u\coloneqq(Y_{\beta+u}-Y_\beta)/u$, differentiability of
$\beta\mapsto Y_\beta$ in $\cX$ gives
$\Delta_u\to\dot Y_\beta$ almost surely. Fr\'echet differentiability of $f$
then yields
\[
\frac{f(Y_{\beta+u})-f(Y_\beta)}{u}
=
\sD f(Y_\beta)[\Delta_u]+o(1)
\longrightarrow
\sD f(Y_\beta)[\dot Y_\beta]
\qquad Q_\theta\text{-almost surely}.
\]
The dominated convergence theorem therefore gives
\[
\partial_\beta F^{\mathrm G}(\theta,\beta)
=
\partial_\beta\E_{Q_\theta}[f(Y_\beta)]
=
\E_{Q_\theta}\bigl[\sD f(Y_\beta)[\dot Y_\beta]\bigr].
\]
By the assumed joint Gaussianity of $(Y_\beta,\dot Y_\beta)$,
Lemma~\ref{lem:exact_gibp_two}, applied with $Y_1=Y_\beta$ and
$Y_2=\dot Y_\beta$, gives
\[
\E_{Q_\theta}\bigl[\sD f(Y_\beta)[\dot Y_\beta]\bigr]
=
\E_{Q_\theta}[\sD f(Y_\beta)]\bigl[\E_{Q_\theta}\dot Y_\beta\bigr]
+
\bigl\langle\E_{Q_\theta}[\sD^2f(Y_\beta)],K_{12}\bigr\rangle,
\]
where $K_{12}(s,t)\coloneqq\cov_{Q_\theta}(Y_\beta(s),\dot Y_\beta(t))$.

Since $\sD f(x)[h]=\int_T h(t)\,\pi_x(\dd t)$ by
Lemma~\ref{lem:logpartition_smooth}, Fubini's theorem gives, for the first
term,
\[
\E_{Q_\theta}[\sD f(Y_\beta)]\bigl[\E_{Q_\theta}\dot Y_\beta\bigr]
=
\int_T
\E_{Q_\theta}[\dot Y_\beta(t)]
\,\E_{Q_\theta}[\pi_{Y_\beta}](\dd t).
\]
For the second term, identity~\eqref{eq:contraction_identity}, applied with
$x=Y_\beta$ and $K_{12}$ as above, and taking the outer expectation over
$Y_\beta$, gives
\begin{align*}
\bigl\langle\E_{Q_\theta}[\sD^2f(Y_\beta)],K_{12}\bigr\rangle
&=
\int_T
\cov_{Q_\theta}(Y_\beta(t),\dot Y_\beta(t))
\,\E_{Q_\theta}[\pi_{Y_\beta}](\dd t)
\\
&\quad-
\iint_{T^2}
\cov_{Q_\theta}(Y_\beta(s),\dot Y_\beta(t))
\,\E_{Q_\theta}[\pi_{Y_\beta}^{\otimes2}](\dd s,\dd t).
\end{align*}
Combining the two displays, and marginalizing the first integral over an
auxiliary $s\in T$,
\begin{align*}
\partial_\beta F^{\mathrm G}(\theta,\beta)
&=
\int_T
\Bigl(
\E_{Q_\theta}[\dot Y_\beta(t)]
+
\cov_{Q_\theta}(Y_\beta(t),\dot Y_\beta(t))
\Bigr)
\,\E_{Q_\theta}[\pi_{Y_\beta}](\dd t)
\\
&\quad-
\iint_{T^2}
\cov_{Q_\theta}(Y_\beta(s),\dot Y_\beta(t))
\,\E_{Q_\theta}[\pi_{Y_\beta}^{\otimes2}](\dd s,\dd t)
\\
&=
\iint_{T^2}\psi_{\theta,\beta}^{\mathrm G}(s,t)\,\E_{Q_\theta}[\pi_{Y_\beta}^{\otimes2}](\dd s,\dd t),
\end{align*}
by the definition of $\psi_{\theta,\beta}^{\mathrm G}$ in~\eqref{eq:psi_Y}.
\end{proof}

\subsection{Proof of Lemma~\ref{lem:deri_F}}
\label{app:proof_deri_F}

\begin{proof}
Recall that $f$ is defined in~\eqref{eq:log_partition_functional}. Fix
$\beta\in\cB$. Since $\|\sD f\|_{\mathrm{op}}\le1$, for every nonzero $u$
such that $\beta+u\in\cB$,
\[
\left|
\frac{f(X_{\beta+u})-f(X_\beta)}{u}
\right|
\le
\sup_{r\in\cB}\|\dot X_r\|_\infty.
\]
The right-hand side is integrable by assumption.
Setting $\Delta_u\coloneqq(X_{\beta+u}-X_\beta)/u$, we have
$\Delta_u\to\dot X_\beta$ almost surely, and Fr\'echet differentiability of
$f$ gives
\[
\frac{f(X_{\beta+u})-f(X_\beta)}{u}
=
\sD f(X_\beta)[\Delta_u]+o(1)
\longrightarrow
\sD f(X_\beta)[\dot X_\beta]
\qquad P_\theta\text{-almost surely}.
\]
The dominated convergence theorem, followed by
$\dot X_\beta=\sum_{m=1}^M\dot X_{\beta,m}$ and linearity, gives
\[
\partial_\beta F(\theta,\beta)
=
\E_{P_\theta}\bigl[\sD f(X_\beta)[\dot X_\beta]\bigr]
=
\sum_{m=1}^M
\E_{P_\theta}\bigl[\sD f(X_\beta)[\dot X_{\beta,m}]\bigr].
\]

Fix $m\in[M]$ and write
$\dot X_{\beta,m}=\E_{P_\theta}[\dot X_{\beta,m}]+(\dot X_{\beta,m}-\E_{P_\theta}[\dot X_{\beta,m}])$.
By linearity,
\begin{align*}
\E_{P_\theta}\bigl[\sD f(X_\beta)[\dot X_{\beta,m}]\bigr]
&=
\E_{P_\theta}\bigl[\sD f(X_\beta)\bigr]\bigl[\E_{P_\theta}[\dot X_{\beta,m}]\bigr]\\
&\qquad +
\E_{P_\theta}\bigl[\sD f(X_\beta)\bigl[\dot X_{\beta,m}-\E_{P_\theta}[\dot X_{\beta,m}]\bigr]\bigr].
\end{align*}
Since $\sD f(X_\beta)[h]=\int_T h(t)\,\pi_{X_\beta}(\dd t)$, the first term
equals 
\[
\int_T\E_{P_\theta}[\dot X_{\beta,m}(t)]\,\E_{P_\theta}[\pi_{X_\beta}](\dd t).
\]
For the second term, write $X_\beta=X_{\beta,m}+X_\beta^{(-m)}$, where
$X_\beta^{(-m)}\coloneqq\sum_{m'\ne m}X_{\beta,m'}$ is, by the independence
assumption, independent of the pair $(X_{\beta,m},\dot X_{\beta,m})$. Fix a
realization $x^{(-m)}$ of $X_\beta^{(-m)}$ and set
$F_{x^{(-m)}}(z)\coloneqq f(x^{(-m)}+z)$, so that
$\sD^kF_{x^{(-m)}}(z)=\sD^kf(x^{(-m)}+z)$ for every $k$. Hence,
\[
M_3(F_{x^{(-m)}})
\coloneqq
\sup_{z\in\cX}\|\sD^3F_{x^{(-m)}}(z)\|_{\mathrm{op}}
\le2
\]
by Lemma~\ref{lem:logpartition_smooth}, uniformly in $x^{(-m)}$.
The assumed finiteness of
$\E_{P_\theta}\|X_{\beta,m}\|_\infty^3$,
$\E_{P_\theta}\|\dot X_{\beta,m}\|_\infty^3$ ensures that the hypotheses of
Lemma~\ref{lem:approx_ibp_two} are satisfied.
Conditionally on $X_\beta^{(-m)}=x^{(-m)}$, applying
Lemma~\ref{lem:approx_ibp_two} with $F=F_{x^{(-m)}}$, $X_1=X_{\beta,m}$, and
$X_2=\dot X_{\beta,m}$ gives
\begin{align*}
\E_{P_\theta}\bigl[\sD f(X_\beta)\bigl[\dot X_{\beta,m}-\E_{P_\theta}[\dot X_{\beta,m}]\bigr]
\bigm| X_\beta^{(-m)}=x^{(-m)}\bigr]
&=
\bigl\langle\E_{P_\theta}[\sD^2f(X_\beta)\mid x^{(-m)}],K_m\bigr\rangle
\\
&\quad+
R_m(x^{(-m)}),
\end{align*}
where $K_m(s,t)\coloneqq\cov_{P_\theta}(X_{\beta,m}(s),\dot X_{\beta,m}(t))$
and
\begin{align*}
|R_m(x^{(-m)})|
&\le
\frac32M_3(F_{x^{(-m)}})
\Bigl(\E_{P_\theta}\|X_{\beta,m}-\E_{P_\theta}[X_{\beta,m}]\|_\infty^3\Bigr)^{2/3}
\\
&\quad\times
\Bigl(\E_{P_\theta}\|\dot X_{\beta,m}-\E_{P_\theta}[\dot X_{\beta,m}]\|_\infty^3\Bigr)^{1/3}.
\end{align*}
Since $M_3(F_{x^{(-m)}})\le2$ uniformly in $x^{(-m)}$, this bound persists
after averaging over $X_\beta^{(-m)}$. Define the resulting unconditional
remainder by
\[
R_{m,\beta}
\coloneqq
\E_{P_\theta}\bigl[R_m(X_\beta^{(-m)})\bigr].
\]
Then $|R_{m,\beta}|\le
\E_{P_\theta}[|R_m(X_\beta^{(-m)})|]$ gives precisely the asserted bound on
$R_{m,\beta}$. Moreover, after replacing $x^{(-m)}$ by
$X_\beta^{(-m)}$, we have
\[
\sD^2F_{X_\beta^{(-m)}}(X_{\beta,m})
=
\sD^2f(X_\beta).
\]
Therefore, since $K_m$ is deterministic, averaging the conditional main
term over $X_\beta^{(-m)}$ and using the tower property gives
\[
\E_{P_\theta}\!\left[
\left\langle
\E_{P_\theta}\!\left[\sD^2f(X_\beta)\mid X_\beta^{(-m)}\right],
K_m
\right\rangle
\right]
=
\left\langle\E_{P_\theta}[\sD^2f(X_\beta)],K_m\right\rangle.
\]

Summing over $m$, the pairs $\{(X_{\beta,m},\dot X_{\beta,m})\}_{m=1}^M$
being independent gives, for $m'\ne m''$,
\[
\cov_{P_\theta}\bigl(X_{\beta,m'}(s),\dot X_{\beta,m''}(t)\bigr)=0,
\]
so
\[
\sum_{m=1}^M K_m(s,t)
=
\sum_{m=1}^M\cov_{P_\theta}(X_{\beta,m}(s),\dot X_{\beta,m}(t))
=
\cov_{P_\theta}(X_\beta(s),\dot X_\beta(t)),
\]
and hence, by bilinearity of the contraction,
\[
\sum_{m=1}^M
\bigl\langle\E_{P_\theta}[\sD^2f(X_\beta)],K_m\bigr\rangle
=
\bigl\langle\E_{P_\theta}[\sD^2f(X_\beta)],\cov_{P_\theta}(X_\beta,\dot X_\beta)\bigr\rangle.
\]
By identity~\eqref{eq:contraction_identity}, applied with $x=X_\beta$ and
$K_{12}=\cov_{P_\theta}(X_\beta,\dot X_\beta)$, and taking the outer
expectation over $X_\beta$,
\begin{align*}
\E_{P_\theta}\Bigl[\bigl\langle\sD^2f(X_\beta),\cov_{P_\theta}(X_\beta,\dot X_\beta)\bigr\rangle\Bigr]
&=
\int_T\cov_{P_\theta}(X_\beta(t),\dot X_\beta(t))\,\E_{P_\theta}[\pi_{X_\beta}](\dd t)
\\
&\quad-
\iint_{T^2}\cov_{P_\theta}(X_\beta(s),\dot X_\beta(t))\,\E_{P_\theta}[\pi_{X_\beta}^{\otimes2}](\dd s,\dd t).
\end{align*}
Combining this with the mean term derived above and the definition of
$\psi_{\theta,\beta}$ in~\eqref{eq:psi_X},
\begin{align*}
&\int_T\Bigl(\E_{P_\theta}[\dot X_\beta(t)]+\cov_{P_\theta}(X_\beta(t),\dot X_\beta(t))\Bigr)\,\E_{P_\theta}[\pi_{X_\beta}](\dd t)
\\
&\qquad-
\iint_{T^2}\cov_{P_\theta}(X_\beta(s),\dot X_\beta(t))\,\E_{P_\theta}[\pi_{X_\beta}^{\otimes2}](\dd s,\dd t)
=
\iint_{T^2}\psi_{\theta,\beta}(s,t)\,\E_{P_\theta}[\pi_{X_\beta}^{\otimes2}](\dd s,\dd t),
\end{align*}
Together
with $\sum_{m=1}^M R_{m,\beta}$, this establishes the asserted identity.
\end{proof}

\subsection{Proof of Lemma~\ref{lem:linear_perturb_G}}
\label{app:proof_linear_perturb_G}

\begin{proof}
Since $Y_\beta$ is affine in $\beta$, we have
$\dot Y_\beta=\widetilde Y$ for every $\beta\in\cB$. Moreover, $Y_\beta$ and
$\widetilde Y$ are both affine functions of the same Gaussian element $\xi$,
so $(Y_\beta,\widetilde Y)$ is jointly Gaussian for every $\beta$, and
Fernique's theorem gives
$\E_{Q_\theta}\|\widetilde Y\|_\infty<\infty$. Thus the hypotheses of
Lemma~\ref{lem:deri_FG} are satisfied. Differentiating and using
$H_\theta=\Sigma_\theta B_\theta$ gives
\begin{align*}
\E_{Q_\theta}[\widetilde Y(t)]
&=
a\langle H_\theta[\eta(\theta)-\eta(t)],
\eta(\theta)-\eta(t)\rangle,
\\
\cov_{Q_\theta}(Y_\beta(s),\widetilde Y(t))
&=
\langle H_\theta\eta(s),\eta(t)\rangle
+
\beta\langle B_\theta^*\Sigma_\theta B_\theta\eta(s),\eta(t)\rangle.
\end{align*}
Substitution into~\eqref{eq:psi_Y} gives the stated formula for
$\psi_{\theta,\beta}^{\mathrm G}$. The first assertion then follows from
\eqref{eq:gaussian_free_energy_derivative}. At $\beta=0$, we have $Y_0=Y$ and
$\psi_{\theta,0}^{\mathrm G}=\phi_\theta$, which gives the final display.
\end{proof}

\subsection{Proof of Theorem~\ref{thm:fe_comparison}}
\label{app:proof_fe_comparison}

\begin{proof}
Fix $\theta\in\Theta$. If
\[
\sum_{m=1}^M
\E_\theta\|X_m-\E_\theta[X_m]\|_\infty^3
=\infty,
\]
then the asserted bound is immediate. We may therefore assume that this sum
is finite. Since $M<\infty$,
\[
\E_\theta\|X-\E_\theta[X]\|_\infty^3
\le
M^2\sum_{m=1}^M
\E_\theta\|X_m-\E_\theta[X_m]\|_\infty^3
<\infty,
\]
and hence $\E_\theta\|X\|_\infty^3<\infty$. In particular,
$|f(X)|\le\|X\|_\infty$ implies that $F(\theta)$ is finite. The Gaussian
free energy $F^{\mathrm G}(\theta)$ is finite by Fernique's theorem.

Since the free energies depend only on the marginal
laws of $X$ under $P_\theta$ and of $Y$ under $Q_\theta$, we may realize $X$
and $Y$ independently on a common product probability space and write $\E_\theta$
for expectation on this space. Set
\[
m_X\coloneqq\E_\theta[X],
\qquad
m_Y\coloneqq\E_\theta[Y],
\qquad
\Sigma_X\coloneqq\cov_\theta(X),
\qquad
\Sigma_Y\coloneqq\cov_\theta(Y),
\]
and
\[
\bar X\coloneqq X-m_X=\sum_{m=1}^M\bar X_m,
\qquad
\bar X_m\coloneqq X_m-\E_\theta[X_m],
\qquad
\bar Y\coloneqq Y-m_Y.
\]

\emph{Guerra interpolation and Newton--Leibniz.}
For $t\in[0,1]$, define
\[
\mu_t\coloneqq t m_X+(1-t)m_Y,
\qquad
Z_t\coloneqq
\mu_t+\sqrt t\,\bar X+\sqrt{1-t}\,\bar Y,
\]
and set $g(t)\coloneqq\E_\theta[f(Z_t)]$. Then $Z_1=X$ and $Z_0=Y$, so
$g(1)=F(\theta)$ and $g(0)=F^{\mathrm G}(\theta)$. We will justify and apply
the Newton--Leibniz identity
\begin{equation}
F(\theta)-F^{\mathrm G}(\theta)
=
g(1)-g(0)
=
\int_0^1g'(t)\,\dd t.
\label{eq:thm4_newton_leibniz}
\end{equation}
For $t\in(0,1)$,
\[
\dot Z_t
=
m_X-m_Y
+\frac{1}{2\sqrt t}\bar X
-\frac{1}{2\sqrt{1-t}}\bar Y.
\]
More precisely, for every $I=[a,b]\subset(0,1)$,
\[
\sup_{t\in I}\|\dot Z_t\|_\infty
\le
\|m_X-m_Y\|_\infty
+
\frac{\|\bar X\|_\infty}{2\sqrt a}
+
\frac{\|\bar Y\|_\infty}{2\sqrt{1-b}}
\eqqcolon H_I.
\]
The preceding bound $\E_\theta\|X\|_\infty^3<\infty$ and H\"older's inequality
give $\E_\theta\|\bar X\|_\infty<\infty$. Moreover, since $\bar Y$ is a Gaussian
element of the separable Banach space
$\cX$, Fernique's theorem~\cite[Theorem~4.13]{hairer:2025spde}; see also
\cite[Theorem~2.8.5]{bogachev:1998gaussian}, gives
$\E_\theta\exp\{\alpha\|\bar Y\|_\infty^2\}<\infty$ for some $\alpha>0$, and hence
$\E_\theta\|\bar Y\|_\infty<\infty$. Thus $\E_\theta[H_I]<\infty$. Dominated convergence,
together with $\|\sD f\|_{\mathrm{op}}\le1$ from
Lemma~\ref{lem:logpartition_smooth}, therefore gives
\begin{align}
g'(t)
&=
\E_\theta\left[\sD f(Z_t)[m_X-m_Y]\right]
\notag\\
&\quad+
\frac{1}{2\sqrt t}\E_\theta\left[\sD f(Z_t)[\bar X]\right]
-
\frac{1}{2\sqrt{1-t}}\E_\theta\left[\sD f(Z_t)[\bar Y]\right].
\label{eq:thm4_direct_derivative}
\end{align}

\emph{Gaussian fluctuation term.}
Fix $t\in(0,1)$. Conditionally on $X$, the pair
$(Z_t,\bar Y)$ is jointly Gaussian: its second component is centered, and
its cross-covariance kernel is
\[
\cov_\theta\bigl(Z_t(s),\bar Y(r)\mid X\bigr)
=
\sqrt{1-t}\,\Sigma_Y(s,r).
\]
Lemma~\ref{lem:exact_gibp_two}, applied conditionally with $Y_1=Z_t$ and
$Y_2=\bar Y$, therefore gives
\[
\E_\theta\left[\sD f(Z_t)[\bar Y]\mid X\right]
=
\sqrt{1-t}\,
\left\langle
\E_\theta\left[\sD^2f(Z_t)\mid X\right],
\Sigma_Y
\right\rangle.
\]
Taking the outer expectation and dividing by $2\sqrt{1-t}$ gives
\begin{equation}
\frac{1}{2\sqrt{1-t}}
\E_\theta\left[\sD f(Z_t)[\bar Y]\right]
=
\frac12
\left\langle
\E_\theta[\sD^2f(Z_t)],
\Sigma_Y
\right\rangle.
\label{eq:thm4_direct_gaussian}
\end{equation}

\emph{Non-Gaussian fluctuation term.}
For $m\in[M]$, write
$\bar X=\bar X_m+\bar X^{(-m)}$, where
$\bar X^{(-m)}\coloneqq\sum_{m'\ne m}\bar X_{m'}$, and define
\[
B_{t,m}
\coloneqq
\mu_t+\sqrt t\,\bar X^{(-m)}+\sqrt{1-t}\,\bar Y.
\]
The random element $B_{t,m}$ is independent of $\bar X_m$. For a fixed
realization $b$ of $B_{t,m}$, let
$F_{t,b}(z)\coloneqq f(b+\sqrt t\,z)$. Then
\[
\sD^kF_{t,b}(z)
=
t^{k/2}\sD^kf(b+\sqrt t\,z),
\]
and Lemma~\ref{lem:logpartition_smooth} gives
\[
M_3(F_{t,b})
\coloneqq
\sup_{z\in\cX}\|\sD^3F_{t,b}(z)\|_{\mathrm{op}}
\le
2t^{3/2},
\]
uniformly in $b$. Conditionally on $B_{t,m}=b$, apply
Lemma~\ref{lem:approx_ibp_two} with $F=F_{t,b}$ and
$X_1=X_2=\bar X_m$. Since $\bar X_m$ is centered with covariance
$\Sigma_m\coloneqq\cov_\theta(X_m)$, this gives
\begin{align}
&\sqrt t\,
\E_\theta\left[
\sD f(Z_t)[\bar X_m]
\mid B_{t,m}=b
\right]
\notag\\
&\qquad=
t\left\langle
\E_\theta\left[\sD^2f(Z_t)\mid B_{t,m}=b\right],
\Sigma_m
\right\rangle
+\rho_{t,m}(b),
\label{eq:thm4_direct_nongaussian_conditional}
\end{align}
where
\[
|\rho_{t,m}(b)|
\le
\frac32M_3(F_{t,b})\,\E_\theta\|\bar X_m\|_\infty^3
\le
3t^{3/2}\E_\theta\|\bar X_m\|_\infty^3.
\]
Define
\[
R_{t,m}
\coloneqq
\frac{1}{2t}\E_\theta[\rho_{t,m}(B_{t,m})].
\]
Averaging \eqref{eq:thm4_direct_nongaussian_conditional} over $B_{t,m}$ and
dividing by $2t$ yields
\[
\frac{1}{2\sqrt t}
\E_\theta\left[\sD f(Z_t)[\bar X_m]\right]
=
\frac12
\left\langle
\E_\theta[\sD^2f(Z_t)],
\Sigma_m
\right\rangle
+
R_{t,m},
\]
with
\[
|R_{t,m}|
\le
\frac32\sqrt t\,\E_\theta\|\bar X_m\|_\infty^3.
\]
Because $X_1,\ldots,X_M$ are independent,
$\sum_{m=1}^M\Sigma_m=\Sigma_X$. Summing over $m$ therefore gives
\begin{equation}
\frac{1}{2\sqrt t}
\E_\theta\left[\sD f(Z_t)[\bar X]\right]
=
\frac12
\left\langle
\E_\theta[\sD^2f(Z_t)],
\Sigma_X
\right\rangle
+
R_t,
\label{eq:thm4_direct_nongaussian}
\end{equation}
where $R_t\coloneqq\sum_{m=1}^MR_{t,m}$ satisfies
\begin{equation}
|R_t|
\le
\frac32\sqrt t
\sum_{m=1}^M
\E_\theta\|X_m-\E_\theta[X_m]\|_\infty^3.
\label{eq:thm4_direct_remainder}
\end{equation}

\emph{Combining and integrating.}
Substituting \eqref{eq:thm4_direct_gaussian} and
\eqref{eq:thm4_direct_nongaussian} into
\eqref{eq:thm4_direct_derivative} gives
\begin{equation}
\begin{aligned}
g'(t)
&=
\E_\theta\left[\sD f(Z_t)[m_X-m_Y]\right]
\\
&\quad+
\frac12
\left\langle
\E_\theta[\sD^2f(Z_t)],
\Sigma_X
\right\rangle
-\frac12
\left\langle
\E_\theta[\sD^2f(Z_t)],
\Sigma_Y
\right\rangle
+
R_t.
\end{aligned}
\label{eq:thm4_direct_ibp}
\end{equation}
By Lemma~\ref{lem:logpartition_smooth},
\[
\left|
\E_\theta\left[\sD f(Z_t)[m_X-m_Y]\right]
\right|
\le
\|m_X-m_Y\|_\infty.
\]
Applying \eqref{eq:contraction_identity} separately to the covariance
kernels $\Sigma_X$ and $\Sigma_Y$ gives
\begin{align*}
&\frac12
\left\langle\E_\theta[\sD^2f(Z_t)],\Sigma_X\right\rangle
-
\frac12
\left\langle\E_\theta[\sD^2f(Z_t)],\Sigma_Y\right\rangle
\\
&\qquad=
\frac12\int_T
\bigl[\Sigma_X(s,s)-\Sigma_Y(s,s)\bigr]
\,\E_\theta[\pi_{Z_t}](\dd s)
\\
&\qquad\quad-
\frac12\iint_{T^2}
\bigl[\Sigma_X(s,r)-\Sigma_Y(s,r)\bigr]
\,\E_\theta[\pi_{Z_t}^{\otimes2}](\dd s,\dd r).
\end{align*}
Both $\E_\theta[\pi_{Z_t}]$ and
$\E_\theta[\pi_{Z_t}^{\otimes2}]$ are probability measures. Therefore,
\[
\left|
\frac12
\left\langle\E_\theta[\sD^2f(Z_t)],\Sigma_X\right\rangle
-
\frac12
\left\langle\E_\theta[\sD^2f(Z_t)],\Sigma_Y\right\rangle
\right|
\le
\|\Sigma_X-\Sigma_Y\|_\infty.
\]
Together with \eqref{eq:thm4_direct_remainder}, this yields the integrable
bound
\[
|g'(t)|
\le
\|m_X-m_Y\|_\infty
+
\|\Sigma_X-\Sigma_Y\|_\infty
+
\frac32\sqrt t
\sum_{m=1}^M
\E_\theta\|X_m-\E_\theta[X_m]\|_\infty^3.
\]

It remains to justify the endpoint passage in
\eqref{eq:thm4_newton_leibniz}. For $s,t\in[0,1]$,
\[
Z_t-Z_s
=
(t-s)(m_X-m_Y)
+
(\sqrt t-\sqrt s)\bar X
+
\bigl(\sqrt{1-t}-\sqrt{1-s}\bigr)\bar Y.
\]
Consequently,
\begin{align*}
\E_\theta\|Z_t-Z_s\|_\infty
&\le
|t-s|\,\|m_X-m_Y\|_\infty
+
|\sqrt t-\sqrt s|\,\E_\theta\|\bar X\|_\infty
\\
&\quad+
\bigl|\sqrt{1-t}-\sqrt{1-s}\bigr|\,
\E_\theta\|\bar Y\|_\infty.
\end{align*}
Here
\[
\E_\theta\|\bar X\|_\infty
\le
2\E_\theta\|X\|_\infty
\le
2\bigl(\E_\theta\|X\|_\infty^3\bigr)^{1/3}
<\infty,
\]
whereas Fernique's theorem~\cite[Theorem~4.13]{hairer:2025spde}; see also
\cite[Theorem~2.8.5]{bogachev:1998gaussian}, gives
$\E_\theta\|\bar Y\|_\infty<\infty$. Hence the displayed upper bound tends
to zero as $t\to s$, so $t\mapsto Z_t$ is continuous from $[0,1]$ into
$L^1(\Omega;\cX)$, where
$\|U\|_{L^1(\Omega;\cX)}\coloneqq\E_\theta\|U\|_\infty$. Since $f$ is
$1$-Lipschitz,
\[
|g(t)-g(s)|
\le
\E_\theta\|Z_t-Z_s\|_\infty,
\]
so $g$ is continuous on $[0,1]$. For every
$I\subset(0,1)$ compact, the same dominated-continuity argument shows that
$g'$ is continuous on $I$, and hence $g\in C^1(I)$. Therefore, for every
$\varepsilon\in(0,1/2)$, Newton--Leibniz on
$[\varepsilon,1-\varepsilon]$ gives
\[
g(1-\varepsilon)-g(\varepsilon)
=
\int_\varepsilon^{1-\varepsilon}g'(t)\,\dd t.
\]
Letting $\varepsilon\downarrow0$, using the endpoint continuity of $g$ and
the integrable bound on $g'$, proves
\eqref{eq:thm4_newton_leibniz}. Consequently,
\begin{align*}
\bigl|F(\theta)-F^{\mathrm G}(\theta)\bigr|
&\le
\|m_X-m_Y\|_\infty
+
\|\Sigma_X-\Sigma_Y\|_\infty
\\
&\quad+
\frac32
\left(\int_0^1\sqrt t\,\dd t\right)
\sum_{m=1}^M
\E_\theta\|X_m-\E_\theta[X_m]\|_\infty^3
\\
&=
\|m_X-m_Y\|_\infty
+
\|\Sigma_X-\Sigma_Y\|_\infty
+
\sum_{m=1}^M
\E_\theta\|X_m-\E_\theta[X_m]\|_\infty^3.
\end{align*}
This is the asserted bound, with coefficients $1$, $1$, and $1$ for the mean,
covariance, and third-moment terms, respectively.
\end{proof}

\subsection{Proof of Theorem~\ref{thm:Phi_comparison}}
\label{app:proof_Phi_comparison}

\begin{proof}
The bound is trivial when $\Delta_1(\theta,\eps)$ or
$\Delta_2(\theta,\eps)$ is infinite. Otherwise, finiteness of the third
moments in $\Delta_1$ and $\Delta_2$, together with hypotheses~(a) and~(b),
ensures that the conditions of Lemmas~\ref{lem:deri_FG},
\ref{lem:deri_F}, and Theorem~\ref{thm:fe_comparison} are met. For every
$\beta\in\cB$, Lemma~\ref{lem:deri_FG} and~\eqref{eq:Phi_G_beta} give
\begin{equation}\label{eq:proof_PhiG_derivative_error}
\bigl|\partial_\beta F^{\mathrm G}(\theta,\beta)
-\Phi^{\mathrm G}(\theta,\beta)\bigr|
\le
\|\psi_{\theta,\beta}^{\mathrm G}-\phi_\theta\|_\infty.
\end{equation}
At $\beta=0$, Lemma~\ref{lem:deri_F} and the definition of $\Phi(\theta)$ give
\begin{equation}\label{eq:proof_Phi_derivative_error}
\bigl|\Phi(\theta)-\partial_\beta F(\theta,0)\bigr|
\le \rho_\theta,
\end{equation}
where
\[
\rho_\theta
\coloneqq
\|\phi_\theta-\psi_{\theta,0}\|_\infty
+3\sum_{m=1}^M
\Bigl(\E_\theta\|X_m-\E_\theta[X_m]\|_\infty^3\Bigr)^{2/3}
\Bigl(
\E_\theta\|\dot X_{0,m}-\E_\theta[\dot X_{0,m}]\|_\infty^3
\Bigr)^{1/3}.
\]
Applying Theorem~\ref{thm:fe_comparison} at $0$ and at $\beta$ yields
\begin{equation}\label{eq:proof_perturbed_fe_errors}
\bigl|F(\theta,0)-F^{\mathrm G}(\theta,0)\bigr|
+
\bigl|F(\theta,\beta)-F^{\mathrm G}(\theta,\beta)\bigr|
\le \Delta_2(\theta,\beta).
\end{equation}

For the upper bound, convexity and
\eqref{eq:proof_PhiG_derivative_error}--\eqref{eq:proof_perturbed_fe_errors}
give
\begin{align*}
\Phi(\theta)
&\le \partial_\beta F(\theta,0)+\rho_\theta \\
&\le \frac{F(\theta,\eps)-F(\theta,0)}{\eps}+\rho_\theta \\
&\le \frac{F^{\mathrm G}(\theta,\eps)-F^{\mathrm G}(\theta,0)}{\eps}
+\rho_\theta+\frac{\Delta_2(\theta,\eps)}{\eps} \\
&\le \partial_\beta F^{\mathrm G}(\theta,\eps)
+\rho_\theta+\frac{\Delta_2(\theta,\eps)}{\eps} \\
&\le \Phi^{\mathrm G}(\theta,\eps)
+\Delta_1(\theta,\eps)
+\frac{\Delta_2(\theta,\eps)}{\eps}.
\end{align*}
The last line uses~\eqref{eq:posterior_comparison_delta1}. The lower bound
follows analogously.
\end{proof}

\section{Proofs for Section~\ref{subsec:exp_family}}
\label{app:proofs_exp_family}

The results proved here are special cases of the regular local comparison
results in Appendix~\ref{app:proofs_regular_local}. We nevertheless give a
separate proof to keep the exponential-family argument self-contained. We
first record a uniform exponential-tilting estimate. In
Subsection~\ref{app:proof_exp_free_energy}, derivatives of the cumulant
generating functions then give the mean, covariance, and third-moment bounds
needed for Theorem~\ref{thm:fe_comparison}. The proof of
Theorem~\ref{thm:exp_posterior_univ} in
Subsection~\ref{app:proof_exp_posterior} reuses these bounds and establishes
the additional directional estimates required by
Theorem~\ref{thm:Phi_comparison}.

\begin{lemma}[Uniform control under exponential tilting]
\label{lem:exp_tilt_control}
Let $Z$ be an $\bbR^q$-valued random variable with cumulant generating
function $L(t)=\log\E e^{\langle t,Z\rangle}$. Let $K\subset\bbR^q$ be
compact, and suppose that $L$ is finite on an open neighborhood of $K$. For
$v\in K$, define
\[
\frac{\dd P_v}{\dd P}
=
\exp\{\langle v,Z\rangle-L(v)\}.
\]
Then the cumulant generating function under $P_v$ is
\begin{equation}
L_v(s)=L(v+s)-L(v).
\label{eq:exp_tilt_cgf}
\end{equation}
Moreover, for some $a>0$,
\begin{equation}
\sup_{v\in K}\E_{P_v}e^{a\|Z\|_2}<\infty.
\label{eq:exp_tilt_uniform_integrability}
\end{equation}
Consequently, all moments and cumulants of any fixed order are bounded
uniformly on $K$.
\end{lemma}

\begin{proof}
Identity~\eqref{eq:exp_tilt_cgf} follows directly from the definition of
$P_v$. Since $L$ is finite on an open neighborhood of the compact set $K$,
we may choose $a>0$ sufficiently small that $L(v+a\sigma)<\infty$ for every
$v\in K$ and $\sigma\in\{-1,1\}^q$.

For $z\in\bbR^q$, choose $\sigma_z\in\{-1,1\}^q$ coordinatewise so that
$\langle\sigma_z,z\rangle=\|z\|_1$. Since $\|z\|_2\le\|z\|_1$ and every
term in the following sum is nonnegative,
\[
e^{a\|z\|_2}
\le
e^{a\langle\sigma_z,z\rangle}
\le
\sum_{\sigma\in\{-1,1\}^q}e^{a\langle\sigma,z\rangle}
.
\]
Taking expectations under $P_v$ and applying~\eqref{eq:exp_tilt_cgf} with
$s=a\sigma$ yield
\[
\E_{P_v}e^{a\|Z\|_2}
\le
\sum_{\sigma\in\{-1,1\}^q}
\exp\{L(v+a\sigma)-L(v)\}.
\]
For each of the finitely many sign vectors $\sigma$, continuity of $L$ and
compactness of $K$ imply that $v\mapsto L(v+a\sigma)-L(v)$ is bounded on
$K$. Consequently,
\[
\sup_{v\in K}\E_{P_v}e^{a\|Z\|_2}
\le
\sum_{\sigma\in\{-1,1\}^q}
\sup_{v\in K}\exp\{L(v+a\sigma)-L(v)\}
<\infty,
\]
which proves~\eqref{eq:exp_tilt_uniform_integrability}.

Finally, for every integer $k\ge1$, the exponential series gives
$r^k\le k!a^{-k}e^{ar}$ for $r\ge0$. Moreover, for any
$i_1,\ldots,i_k\in\{1,\ldots,q\}$,
$|Z_{i_1}\cdots Z_{i_k}|\le\|Z\|_2^k$. Hence,
\[
\sup_{v\in K}\bigl|\E_{P_v}[Z_{i_1}\cdots Z_{i_k}]\bigr|
\le
\sup_{v\in K}\E_{P_v}\|Z\|_2^k
\le
\frac{k!}{a^k}\sup_{v\in K}\E_{P_v}e^{a\|Z\|_2}
<\infty.
\]
Thus all mixed moments of any fixed order are uniformly bounded on $K$.
The corresponding cumulant bounds follow from the moment--cumulant formula,
since each cumulant of fixed order is a finite polynomial in moments of no
greater order.
\end{proof}

We shall use two families of exponential tilts. For $\|u\|<\rho$, define the
model-side tilt $P_u^\Lambda$ of $\lambda$ and the data-generating tilt $P_u$
of $\gamma$ by
\begin{equation}
\begin{aligned}
\frac{\dd P_u^\Lambda}{\dd\lambda}
&=
\exp\{\langle W_0,u\rangle-\Lambda(u)\},
&
\frac{\dd P_u}{\dd\gamma}
&=
\exp\{\langle W_0^\star,u\rangle-\Gamma(u)\}.
\end{aligned}
\label{eq:exp_tilted_laws}
\end{equation}
Under $P_u^\Lambda$, the cumulant generating function of $W_0$ is
$s\mapsto\Lambda(u+s)-\Lambda(u)$. Since $\Gamma(u)=\Xi(0,u)$, the joint
cumulant generating function of $(W_0,W_0^\star)$ under $P_u$ is
\begin{equation}
(v,s)\longmapsto \Xi(v,u+s)-\Xi(0,u).
\label{eq:exp_joint_tilt_cgf}
\end{equation}
Fix $0<\rho_0<\rho$, with $\rho$ as in
Assumption~\ref{assump:expfam_regular}. Apply
Lemma~\ref{lem:exp_tilt_control} first to $W_0$ under the model-side tilts
$P_u^\Lambda$, and then to $(W_0,W_0^\star)$ under the data-generating tilts
$P_u$. It follows that, for some $C<\infty$,
\begin{align}
\sup_{\|u\|\le\rho_0}
\left\{
\E_{P_u^\Lambda}\|W_0\|^3
+\E_{P_u}\|W_0\|^3
+\E_{P_u}\|W_0^\star\|^3
\right\}
\le C.
\label{eq:exp_uniform_third_moments}
\end{align}

\subsection{Proof of Theorem~\ref{thm:exp_free_energy_univ}}
\label{app:proof_exp_free_energy}

We apply Theorem~\ref{thm:fe_comparison} to the exponential-family process
$X$ in~\eqref{eq:X_product} and its Gaussian counterpart $Y$
in~\eqref{eq:Y_product}. The proof reduces to bounding the
mean discrepancy, the covariance discrepancy, and the local centered third
moments. We consider these three quantities in turn.

\emph{Mean discrepancy.}
Fix $\theta\in\Theta_n$. For each $m\le M_n$, equations~\eqref{eq:expf_X},
\eqref{eq:expf_P}, and
\eqref{eq:exp_tilted_laws} give
\begin{equation}
\begin{aligned}
\E_\theta[X_m(t)]
&=
\left\langle
\nabla_v\Xi\bigl(0,\eta_m(\theta)\bigr),\eta_m(t)
\right\rangle
-\Lambda\bigl(\eta_m(t)\bigr),
\\
\E_{Q_\theta}[Y_m(t)]
&=
\left\langle A_0\eta_m(\theta),\eta_m(t)\right\rangle
-\frac12\left\langle K_0\eta_m(t),\eta_m(t)\right\rangle.
\end{aligned}
\label{eq:exp_block_means}
\end{equation}
Taylor's theorem with integral remainder, together with the operator
definitions in~\eqref{eq:expf_gaussian_operators}, gives
\begin{align*}
&\Lambda\bigl(\eta_m(t)\bigr)
-\left\langle\nabla\Lambda(0),\eta_m(t)\right\rangle
-\frac12\left\langle K_0\eta_m(t),\eta_m(t)\right\rangle
\\
&\qquad=
\int_0^1\frac{(1-s)^2}{2}
\nabla^3\Lambda\bigl(s\eta_m(t)\bigr)
[\eta_m(t),\eta_m(t),\eta_m(t)]\,\dd s,
\end{align*}
and
\begin{align*}
&\left\langle
\nabla_v\Xi\bigl(0,\eta_m(\theta)\bigr)
-\nabla_v\Xi(0,0)-A_0\eta_m(\theta),
\eta_m(t)
\right\rangle
\\
&\qquad=
\int_0^1(1-s)
\nabla_{vuu}^3\Xi\bigl(0,s\eta_m(\theta)\bigr)
[\eta_m(t),\eta_m(\theta),\eta_m(\theta)]\,\dd s.
\end{align*}
The baseline mean condition in Assumption~\ref{assump:expfam_regular} is
$\nabla_v\Xi(0,0)=\nabla\Lambda(0)$. Using this identity to cancel the
first-order terms in the difference of the two block means
in~\eqref{eq:exp_block_means} yields
\begin{equation}
\begin{aligned}
\E_\theta[X_m(t)]-\E_{Q_\theta}[Y_m(t)]
&=
\int_0^1(1-s)
\nabla_{vuu}^3\Xi\bigl(0,s\eta_m(\theta)\bigr)
[\eta_m(t),\eta_m(\theta),\eta_m(\theta)]\,\dd s
\\
&\quad-
\int_0^1\frac{(1-s)^2}{2}
\nabla^3\Lambda\bigl(s\eta_m(t)\bigr)
[\eta_m(t),\eta_m(t),\eta_m(t)]\,\dd s.
\end{aligned}
\label{eq:exp_local_mean_remainder}
\end{equation}

The third derivatives in~\eqref{eq:exp_local_mean_remainder} are cumulant
tensors under the tilted laws in~\eqref{eq:exp_tilted_laws}. Writing $W_{0,i}$
and $W_{0,i}^\star$ for the coordinates of $W_0$ and $W_0^\star$, respectively,
their entries are
\begin{equation}
\begin{aligned}
\nabla^3\Lambda(u)_{ijk}
&=
\E_{P_u^\Lambda}
\bigl(W_{0,i}-\E_{P_u^\Lambda}W_{0,i}\bigr)
\bigl(W_{0,j}-\E_{P_u^\Lambda}W_{0,j}\bigr)
\bigl(W_{0,k}-\E_{P_u^\Lambda}W_{0,k}\bigr),
\\
\nabla_{vuu}^3\Xi(0,u)_{ijk}
&=
\E_{P_u}
\bigl(W_{0,i}-\E_{P_u}W_{0,i}\bigr)
\bigl(W_{0,j}^\star-\E_{P_u}W_{0,j}^\star\bigr)
\bigl(W_{0,k}^\star-\E_{P_u}W_{0,k}^\star\bigr).
\end{aligned}
\label{eq:exp_mean_cumulant_tensors}
\end{equation}
For fixed $u$, set
$\overline W_0^\Lambda=W_0-\E_{P_u^\Lambda}W_0$,
$\overline W_0=W_0-\E_{P_u}W_0$, and
$\overline W_0^\star=W_0^\star-\E_{P_u}W_0^\star$.
Define the random vector $Z^\Lambda(u)\in\bbR^{d^3}$ by
\[
Z_{ijk}^\Lambda(u)
\coloneqq
\overline W_{0,i}^\Lambda
\overline W_{0,j}^\Lambda
\overline W_{0,k}^\Lambda,
\qquad 1\le i,j,k\le d.
\]
The first identity in~\eqref{eq:exp_mean_cumulant_tensors} states that the
coordinate array of $\nabla^3\Lambda(u)$ is $\E_{P_u^\Lambda}Z^\Lambda(u)$.
Since the Euclidean norm is convex, Jensen's inequality gives
$\|\E_{P_u^\Lambda}Z^\Lambda(u)\|_2
\le \E_{P_u^\Lambda}\|Z^\Lambda(u)\|_2$; equivalently,
\begin{align*}
\left(
\sum_{1\le i,j,k\le d}
\left|\nabla^3\Lambda(u)_{ijk}\right|^2
\right)^{1/2}
&\le
\E_{P_u^\Lambda}
\left(
\sum_{1\le i,j,k\le d}
\left|\overline W_{0,i}^\Lambda
\overline W_{0,j}^\Lambda
\overline W_{0,k}^\Lambda\right|^2
\right)^{1/2}
\\
&=
\E_{P_u^\Lambda}
\left(\sum_{i=1}^d|\overline W_{0,i}^\Lambda|^2\right)^{3/2}
=
\E_{P_u^\Lambda}\|\overline W_0^\Lambda\|^3.
\end{align*}
Similarly, the second identity in~\eqref{eq:exp_mean_cumulant_tensors} and
H\"older's inequality yield
\begin{align*}
\left(
\sum_{1\le i,j,k\le d}
\left|\nabla_{vuu}^3\Xi(0,u)_{ijk}\right|^2
\right)^{1/2}
&\le
\E_{P_u}\bigl[\|\overline W_0\|
\|\overline W_0^\star\|^2\bigr]
\\
&\le
\left(\E_{P_u}\|\overline W_0\|^3\right)^{1/3}
\left(\E_{P_u}\|\overline W_0^\star\|^3\right)^{2/3}.
\end{align*}
Using $\E\|W-\E W\|^3\le 8\E\|W\|^3$ and Young's inequality in these
two estimates, and then squaring, gives
\begin{equation}
\begin{aligned}
\sum_{1\le i,j,k\le d}
\left|\nabla^3\Lambda(u)_{ijk}\right|^2
&\lesssim
\left(\E_{P_u^\Lambda}\|W_0\|^3\right)^2,
\\
\sum_{1\le i,j,k\le d}
\left|\nabla_{vuu}^3\Xi(0,u)_{ijk}\right|^2
&\lesssim
\left\{
\E_{P_u}\|W_0\|^3+
\E_{P_u}\|W_0^\star\|^3
\right\}^2.
\end{aligned}
\label{eq:exp_mean_cumulant_bounds}
\end{equation}

Combining~\eqref{eq:exp_local_mean_remainder} and
\eqref{eq:exp_mean_cumulant_bounds} with Cauchy--Schwarz and
\eqref{eq:feature_radius} gives, uniformly in $m\le M_n$ and $t\in T_n$,
\begin{equation}
\left|\E_\theta[X_m(t)]-\E_{Q_\theta}[Y_m(t)]\right|
\lesssim
r_n^3
\sup_{\|u\|\le r_n}
\left\{
\E_{P_u^\Lambda}\|W_0\|^3
+\E_{P_u}\|W_0\|^3
+\E_{P_u}\|W_0^\star\|^3
\right\}.
\label{eq:exp_local_mean_bound}
\end{equation}
The product decompositions~\eqref{eq:X_product} and~\eqref{eq:Y_product} and
linearity of expectation give
\[
\E_\theta[X(t)]-\E_{Q_\theta}[Y(t)]
=
\sum_{m=1}^{M_n}
\left\{\E_\theta[X_m(t)]-\E_{Q_\theta}[Y_m(t)]\right\}.
\]
Taking the supremum over $t\in T_n$, applying the triangle inequality, and
using~\eqref{eq:exp_local_mean_bound} therefore yield
\begin{equation}
\bigl\|\E_\theta[X]-\E_{Q_\theta}[Y]\bigr\|_\infty
\lesssim
M_n r_n^3
\sup_{\|u\|\le r_n}
\left\{
\E_{P_u^\Lambda}\|W_0\|^3
+\E_{P_u}\|W_0\|^3
+\E_{P_u}\|W_0^\star\|^3
\right\}.
\label{eq:exp_mean_bound}
\end{equation}

\emph{Covariance discrepancy.}
The deterministic log-partition term in~\eqref{eq:expf_X} has zero
covariance. Thus, by~\eqref{eq:exp_tilted_laws} and
\eqref{eq:Y_local_covariance},
\begin{equation}
\begin{aligned}
\cov_\theta\bigl(X_m(s),X_m(t)\bigr)
&=
\nabla_{vv}^2\Xi\bigl(0,\eta_m(\theta)\bigr)
[\eta_m(s),\eta_m(t)],
\\
\cov_{Q_\theta}\bigl(Y_m(s),Y_m(t)\bigr)
&=
\nabla_{vv}^2\Xi(0,0)[\eta_m(s),\eta_m(t)].
\end{aligned}
\label{eq:exp_local_covariances}
\end{equation}
The operator-valued fundamental theorem of calculus gives
\begin{equation}
\nabla_{vv}^2\Xi(0,u)-\nabla_{vv}^2\Xi(0,0)
=
\int_0^1\nabla_{vvu}^3\Xi(0,\tau u)[u],\dd\tau
\label{eq:exp_cov_cumulant_expansion}
\end{equation}
along the line $\tau\mapsto\tau u$. Subtracting the two identities
in~\eqref{eq:exp_local_covariances} and applying
\eqref{eq:exp_cov_cumulant_expansion} with $u=\eta_m(\theta)$, gives
\begin{equation}
\begin{aligned}
&\cov_\theta\bigl(X_m(s),X_m(t)\bigr)
-\cov_{Q_\theta}\bigl(Y_m(s),Y_m(t)\bigr)
\\
&\qquad=
\int_0^1
\nabla_{vvu}^3\Xi\bigl(0,\tau\eta_m(\theta)\bigr)
[\eta_m(s),\eta_m(t),\eta_m(\theta)]\,\dd\tau.
\end{aligned}
\label{eq:exp_local_cov_remainder}
\end{equation}

The derivative in~\eqref{eq:exp_local_cov_remainder} is a tilted cumulant
tensor with entries
\begin{equation}
\nabla_{vvu}^3\Xi(0,u)_{ijk}
=
\E_{P_u}
\bigl(W_{0,i}-\E_{P_u}W_{0,i}\bigr)
\bigl(W_{0,j}-\E_{P_u}W_{0,j}\bigr)
\bigl(W_{0,k}^\star-\E_{P_u}W_{0,k}^\star\bigr).
\label{eq:exp_cov_cumulant_tensor}
\end{equation}
For fixed $u$, write
$\overline W_0=W_0-\E_{P_u}W_0$ and
$\overline W_0^\star=W_0^\star-\E_{P_u}W_0^\star$, and define
$Z^\Xi(u)\in\bbR^{d^3}$ by
\[
Z_{ijk}^\Xi(u)
\coloneqq
\overline W_{0,i}\overline W_{0,j}\overline W_{0,k}^\star,
\qquad 1\le i,j,k\le d.
\]
By~\eqref{eq:exp_cov_cumulant_tensor}, the coordinate array of
$\nabla_{vvu}^3\Xi(0,u)$ is $\E_{P_u}Z^\Xi(u)$. Jensen's and H\"older's
inequalities therefore give
\begin{align*}
\left(
\sum_{1\le i,j,k\le d}
\left|\nabla_{vvu}^3\Xi(0,u)_{ijk}\right|^2
\right)^{1/2}
&\le
\E_{P_u}\bigl[\|\overline W_0\|^2\|\overline W_0^\star\|\bigr]
\\
&\le
\left(\E_{P_u}\|\overline W_0\|^3\right)^{2/3}
\left(\E_{P_u}\|\overline W_0^\star\|^3\right)^{1/3}.
\end{align*}
Using $\E\|W-\E W\|^3\le8\E\|W\|^3$, Young's inequality, and then
squaring gives
\begin{equation}
\sum_{1\le i,j,k\le d}
\left|\nabla_{vvu}^3\Xi(0,u)_{ijk}\right|^2
\lesssim
\left\{
\E_{P_u}\|W_0\|^3+
\E_{P_u}\|W_0^\star\|^3
\right\}^2.
\label{eq:exp_cov_cumulant_bound}
\end{equation}
Combining~\eqref{eq:exp_local_cov_remainder} and
\eqref{eq:exp_cov_cumulant_bound} with Cauchy--Schwarz and
\eqref{eq:feature_radius} gives
\begin{equation}
\begin{aligned}
\left|
\cov_\theta\bigl(X_m(s),X_m(t)\bigr)
-\cov_{Q_\theta}\bigl(Y_m(s),Y_m(t)\bigr)
\right|
\\
\lesssim
r_n^3
\sup_{\|u\|\le r_n}
\left\{
\E_{P_u}\|W_0\|^3+
\E_{P_u}\|W_0^\star\|^3
\right\}.
\label{eq:exp_local_cov_bound}
\end{aligned}
\end{equation}

The product structure~\eqref{eq:product_data}--\eqref{eq:X_product}
makes the $X_m$ independent under $P_\theta$, while
\eqref{eq:Y_product} makes the $Y_m$ independent under $Q_\theta$. Thus the
two global covariance kernels are the sums of their local covariance kernels.
Taking the supremum over $s,t\in T_n$, applying the triangle inequality, and
using~\eqref{eq:exp_local_cov_bound} yield
\begin{equation}
\bigl\|\cov_\theta(X)-\cov_{Q_\theta}(Y)\bigr\|_\infty
\lesssim
M_n r_n^3
\sup_{\|u\|\le r_n}
\left\{
\E_{P_u}\|W_0\|^3+
\E_{P_u}\|W_0^\star\|^3
\right\}.
\label{eq:exp_cov_bound}
\end{equation}

\emph{Centered third moments.}
By~\eqref{eq:product_data} and~\eqref{eq:expf_P}, the marginal corresponding
to block $m$ is $P_{\theta,m}=P_{\eta_m(\theta)}$. Centering the local process
therefore gives
\[
X_m(t)-\E_\theta X_m(t)
=
\left\langle
W_0-\E_{P_{\eta_m(\theta)}}W_0,\eta_m(t)
\right\rangle,
\]
and hence, by~\eqref{eq:feature_radius},
\[
\|X_m-\E_\theta X_m\|_\infty
\le
r_n\left\|W_0-\E_{P_{\eta_m(\theta)}}W_0\right\|.
\]
The elementary inequality $\|x-y\|^3\le4(\|x\|^3+\|y\|^3)$, followed by
Jensen's inequality, gives
\[
\E_{P_u}\left\|W_0-\E_{P_u}W_0\right\|^3
\le
4\left\{\E_{P_u}\|W_0\|^3+\|\E_{P_u}W_0\|^3\right\}
\le 8\E_{P_u}\|W_0\|^3.
\]
Consequently, for every $m\le M_n$,
\begin{equation}
\E_\theta\bigl\|X_m-\E_\theta X_m\bigr\|_\infty^3
\le
8r_n^3\E_{P_{\eta_m(\theta)}}\|W_0\|^3.
\label{eq:exp_local_third_moment_bound}
\end{equation}
Summing~\eqref{eq:exp_local_third_moment_bound} over $m$ and using
$\|\eta_m(\theta)\|\le r_n$ gives
\begin{equation}
\sum_{m=1}^{M_n}
\E_\theta
\bigl\|X_m-\E_\theta X_m\bigr\|_\infty^3
\lesssim
M_n r_n^3
\sup_{\|u\|\le r_n}\E_{P_u}\|W_0\|^3.
\label{eq:exp_third_moment_bound}
\end{equation}

\emph{Conclusion.}
The assumptions $M_n=\omega(n)$ and
$r_n\lesssim\sqrt{n/M_n}$ imply $r_n\to0$. Fix $0<\rho_0<\rho$ as in
\eqref{eq:exp_uniform_third_moments}. For all sufficiently large $n$, the
ball $\{u:\|u\|\le r_n\}$ is contained in $\{u:\|u\|\le\rho_0\}$.
Consequently, the uniform moment bound
\eqref{eq:exp_uniform_third_moments}, substituted into
\eqref{eq:exp_mean_bound}, \eqref{eq:exp_cov_bound}, and
\eqref{eq:exp_third_moment_bound}, gives
\[
\begin{aligned}
&\bigl\|\E_\theta X-\E_{Q_\theta}Y\bigr\|_\infty
+\bigl\|\cov_\theta(X)-\cov_{Q_\theta}(Y)\bigr\|_\infty
\\
&\qquad+
\sum_{m=1}^{M_n}
\E_\theta\bigl\|X_m-\E_\theta X_m\bigr\|_\infty^3
\lesssim M_n r_n^3,
\end{aligned}
\]
uniformly over $\theta\in\Theta_n$. Theorem~\ref{thm:fe_comparison} therefore
implies $|F(\theta)-F^{\mathrm G}(\theta)|\lesssim M_nr_n^3$. Finally, by
the normalization in~\eqref{eq:normalized_free_energies} and the assumption
$M_nr_n^2\lesssim n$,
\[
\sup_{\theta\in\Theta_n}
\bigl|F_n(\theta)-F_n^{\mathrm G}(\theta)\bigr|
\lesssim
\frac{M_n r_n^3}{n}
\lesssim r_n,
\]
which proves the theorem.

\subsection{Proof of Theorem~\ref{thm:exp_posterior_univ}}
\label{app:proof_exp_posterior}

\begin{proof}
We first identify the local non-Gaussian and Gaussian perturbations and verify
the structural conditions of Theorem~\ref{thm:Phi_comparison}. We next derive
the blockwise covariance and moment estimates used to control $\Delta_1$ and
$\Delta_2$. Substitution of these estimates into
Theorem~\ref{thm:Phi_comparison} then completes the proof.

Set $B_0\coloneqq\Sigma_0^+H_0$. The span condition gives
$\Sigma_0B_0=H_0$, and $B_0$ is a fixed bounded operator on $\bbR^d$.
Recalling~\eqref{eq:exp_X_tilde} and~\eqref{eq:Y_tilde}, the corresponding
local perturbations are
\begin{equation}
\begin{aligned}
\widetilde X_m(t)
&=
\left\langle
W_0(\omega_m)-\E_{P_{\eta_m(\theta)}}[W_0],
B_0\eta_m(t)
\right\rangle
\\
&\qquad +a\left\langle
H_0[\eta_m(\theta)-\eta_m(t)],
\eta_m(\theta)-\eta_m(t)
\right\rangle,
\\
\widetilde Y_m(t)
&=
\left\langle\xi_m,B_0\eta_m(t)\right\rangle
+a\left\langle
H_0[\eta_m(\theta)-\eta_m(t)],
\eta_m(\theta)-\eta_m(t)
\right\rangle.
\end{aligned}
\label{eq:exp_local_perturbations}
\end{equation}
Thus the deterministic quadratic terms coincide, while the random linear
terms are centered under $P_\theta$ and $Q_\theta$, respectively.

We verify the hypotheses of Theorem~\ref{thm:Phi_comparison}.
For fixed $\theta$, the pairs
$(X_{\beta,m},\dot X_{\beta,m})=(X_{\beta,m},\widetilde X_m)$ are independent
over $m$, and $(Y_\beta,\dot Y_\beta)$ is jointly Gaussian because both
are affine functions of the same Gaussian element, giving~(a).
Since $\dot X_\beta=\widetilde X$ and $\dot Y_\beta=\widetilde Y$ are
independent of $\beta$, the uniform moment
bound~\eqref{eq:exp_uniform_third_moments} and Fernique's theorem
give~(b).
Finally, $X_\beta$ and $Y_\beta$ are affine in $\beta$, so the perturbed
free energies are convex, giving~(c). It remains to prove,
uniformly for $|\beta|\le\eps_0$, with $\eps_0>0$ sufficiently small, that
\begin{equation}
\frac1n\Delta_1(\theta,\beta)
\le C(|\beta|+r_n),
\qquad
\frac1n\Delta_2(\theta,\beta)
\le Cr_n.
\label{eq:exp_posterior_Delta_targets}
\end{equation}

\emph{Blockwise directional covariance bounds.}
By~\eqref{eq:exp_tilted_laws} and~\eqref{eq:exp_local_perturbations}, the
non-Gaussian block covariances are
\begin{equation}
\begin{aligned}
\cov_\theta\bigl(X_m(s),\widetilde X_m(t)\bigr)
&=
\nabla_{vv}^2\Xi\bigl(0,\eta_m(\theta)\bigr)
[\eta_m(s),B_0\eta_m(t)],
\\
\cov_\theta\bigl(\widetilde X_m(s),\widetilde X_m(t)\bigr)
&=
\nabla_{vv}^2\Xi\bigl(0,\eta_m(\theta)\bigr)
[B_0\eta_m(s),B_0\eta_m(t)].
\end{aligned}
\label{eq:exp_nongaussian_directional_covariances}
\end{equation}
For the Gaussian blocks, the equivalent representation
\eqref{eq:gaussian_location_xi_representation} from
Definition~\ref{def:gaussian_loc} gives the corresponding covariances
directly. The random parts of $Y_m(s)$ and
$\widetilde Y_m(t)$ are, respectively,
$\langle\xi_m,\eta_m(s)\rangle$ and
$\langle\xi_m,B_0\eta_m(t)\rangle$, where
$\xi_m\sim\normal(0,\Sigma_0)$. Since the remaining terms are
deterministic,
\begin{equation}
\begin{aligned}
\cov_{Q_\theta}\bigl(Y_m(s),\widetilde Y_m(t)\bigr)
&=
\E_{Q_\theta}\left[
\langle\xi_m,\eta_m(s)\rangle
\langle\xi_m,B_0\eta_m(t)\rangle
\right]
=\langle\Sigma_0\eta_m(s),B_0\eta_m(t)\rangle,
\\
\cov_{Q_\theta}\bigl(\widetilde Y_m(s),\widetilde Y_m(t)\bigr)
&=
\E_{Q_\theta}\left[
\langle\xi_m,B_0\eta_m(s)\rangle
\langle\xi_m,B_0\eta_m(t)\rangle
\right]
=\langle\Sigma_0B_0\eta_m(s),B_0\eta_m(t)\rangle.
\end{aligned}
\label{eq:exp_gaussian_directional_covariances}
\end{equation}
Here $\Sigma_0=\nabla_{vv}^2\Xi(0,0)$ by
\eqref{eq:expf_gaussian_operators}. Subtracting
\eqref{eq:exp_gaussian_directional_covariances} from
\eqref{eq:exp_nongaussian_directional_covariances} yields
\begin{equation}
\begin{aligned}
&\left|
\cov_\theta\bigl(X_m(s),\widetilde X_m(t)\bigr)
-\cov_{Q_\theta}\bigl(Y_m(s),\widetilde Y_m(t)\bigr)
\right|
\\
&\qquad=
\left|
\left\{\nabla_{vv}^2\Xi\bigl(0,\eta_m(\theta)\bigr)-\Sigma_0\right\}
[\eta_m(s),B_0\eta_m(t)]
\right|
\lesssim r_n^3,
\end{aligned}
\label{eq:exp_directional_cov_X}
\end{equation}
and
\begin{equation}
\begin{aligned}
&\left|
\cov_\theta\bigl(\widetilde X_m(s),\widetilde X_m(t)\bigr)
-\cov_{Q_\theta}\bigl(\widetilde Y_m(s),\widetilde Y_m(t)\bigr)
\right|
\\
&\qquad=
\left|
\left\{\nabla_{vv}^2\Xi\bigl(0,\eta_m(\theta)\bigr)-\Sigma_0\right\}
[B_0\eta_m(s),B_0\eta_m(t)]
\right|
\lesssim r_n^3.
\end{aligned}
\label{eq:exp_directional_cov_tilde}
\end{equation}
To justify the final bounds in \eqref{eq:exp_directional_cov_X} and
\eqref{eq:exp_directional_cov_tilde}, apply the cumulant expansion
\eqref{eq:exp_cov_cumulant_expansion} with $u=\eta_m(\theta)$.  For arbitrary
$h_1,h_2\in\cH_0$, this gives
\[
\begin{aligned}
&\left\{
\nabla_{vv}^2\Xi(0,u)-\nabla_{vv}^2\Xi(0,0)
\right\}[h_1,h_2]
\\
&\qquad=
\int_0^1
\nabla_{vvu}^3\Xi(0,\tau u)[h_1,h_2,u],\dd\tau.
\end{aligned}
\]
Cauchy--Schwarz for the trilinear form therefore bounds its absolute value by
\[
\|h_1\|\,\|h_2\|\,\|u\|
\sup_{0\le\tau\le1}
\left(
\sum_{1\le i,j,k\le d}
\left|\nabla_{vvu}^3\Xi(0,\tau u)_{ijk}\right|^2
\right)^{1/2}.
\]
For \eqref{eq:exp_directional_cov_X}, take
$(h_1,h_2)=(\eta_m(s),B_0\eta_m(t))$; for
\eqref{eq:exp_directional_cov_tilde}, take
$(h_1,h_2)=(B_0\eta_m(s),B_0\eta_m(t))$.  Since
$\|u\|,\|\eta_m(s)\|,\|\eta_m(t)\|\le r_n$ by
\eqref{eq:feature_radius} and $\|B_0\|_{\mathrm{op}}<\infty$, the tensor
bound \eqref{eq:exp_cov_cumulant_bound} shows that both quantities are bounded
by
\[
r_n^3
\sup_{\|u\|\le r_n}
\left\{
\E_{P_u}[\|W_0\|^3]+\E_{P_u}[\|W_0^\star\|^3]
\right\}
\lesssim r_n^3,
\]
where the final inequality follows from
\eqref{eq:exp_uniform_third_moments}.

The same uniform moment control applies to the random part of the perturbation.
Indeed, \eqref{eq:exp_local_perturbations} gives
\[
\widetilde X_m(t)-\E_\theta[\widetilde X_m(t)]
=
\left\langle
W_0(\omega_m)-\E_{P_{\eta_m(\theta)}}[W_0],
B_0\eta_m(t)
\right\rangle.
\]
The argument leading to \eqref{eq:exp_local_third_moment_bound}, together with
$\|B_0\|_{\mathrm{op}}<\infty$ and
\eqref{eq:exp_uniform_third_moments}, therefore yields
\begin{equation}
\E_\theta\left[
\bigl\|\widetilde X_m-\E_\theta[\widetilde X_m]\bigr\|_\infty^3
\right]
\lesssim r_n^3.
\label{eq:exp_perturbation_third_moment}
\end{equation}

\emph{Control of $\Delta_1$.}
We bound the three contributions in \eqref{eq:posterior_comparison_delta1}
separately. Lemma~\ref{lem:linear_perturb_G} gives
\[
\psi_{\theta,\beta}^{\mathrm G}(s,t)-\phi_\theta(s,t)
=
\beta\left\langle
B^*\Sigma B[\eta(t)-\eta(s)],\eta(t)
\right\rangle.
\]
The product structure, the feature bound \eqref{eq:feature_radius}, and the
fixed operator bounds on $B_0$ and $\Sigma_0$ give
\begin{equation}
\bigl\|\psi_{\theta,\beta}^{\mathrm G}-\phi_\theta\bigr\|_\infty
\lesssim |\beta|M_nr_n^2.
\label{eq:exp_Delta1_G}
\end{equation}

At $\beta=0$, Lemma~\ref{lem:linear_perturb_G} also gives
$\psi_{\theta,0}^{\mathrm G}=\phi_\theta$. Moreover,
\eqref{eq:exp_local_perturbations} shows that
$\E_\theta[\widetilde X_m(t)]=\E_{Q_\theta}[\widetilde Y_m(t)]$.
Comparing the non-Gaussian test function in
\eqref{eq:exp_psi_theta0} with the Gaussian test function at $\beta=0$ from
Lemma~\ref{lem:linear_perturb_G}, and then using independence over $m$, gives
\begin{align*}
\psi_{\theta,0}(s,t)-\phi_\theta(s,t)
&=
\sum_{m=1}^{M_n}
\Bigl\{
\cov_\theta\bigl(X_m(t),\widetilde X_m(t)\bigr)
-\cov_{Q_\theta}\bigl(Y_m(t),\widetilde Y_m(t)\bigr)
\Bigr\}
\\
&\quad-
\sum_{m=1}^{M_n}
\Bigl\{
\cov_\theta\bigl(X_m(s),\widetilde X_m(t)\bigr)
-\cov_{Q_\theta}\bigl(Y_m(s),\widetilde Y_m(t)\bigr)
\Bigr\}.
\end{align*}
Equation~\eqref{eq:exp_directional_cov_X} now implies
\begin{equation}
\bigl\|\psi_{\theta,0}-\phi_\theta\bigr\|_\infty
\lesssim M_nr_n^3.
\label{eq:exp_Delta1_X}
\end{equation}

Finally, \eqref{eq:exp_local_third_moment_bound} and
\eqref{eq:exp_perturbation_third_moment} imply
\begin{equation}
\begin{aligned}
&\sum_{m=1}^{M_n}
\left(
\E_\theta\left[\|X_m-\E_\theta[X_m]\|_\infty^3\right]
\right)^{2/3}
\left(
\E_\theta\left[
\|\dot X_{0,m}-\E_\theta[\dot X_{0,m}]\|_\infty^3
\right]
\right)^{1/3}
\\
&\qquad\lesssim M_nr_n^3,
\end{aligned}
\label{eq:exp_Delta1_third_moment}
\end{equation}
because $\dot X_{0,m}=\widetilde X_m$. Substitution of these three estimates
\eqref{eq:exp_Delta1_G}, \eqref{eq:exp_Delta1_X}, and
\eqref{eq:exp_Delta1_third_moment} into
\eqref{eq:posterior_comparison_delta1} gives
\begin{equation}
\Delta_1(\theta,\beta)
\lesssim |\beta|M_nr_n^2+M_nr_n^3.
\label{eq:exp_Delta1_bound}
\end{equation}

\emph{Control of $\Delta_2$.}
The unperturbed mean, covariance, and centered third-moment contributions to
\eqref{eq:posterior_comparison_delta2} are bounded by $M_nr_n^3$ through
\eqref{eq:exp_mean_bound}, \eqref{eq:exp_cov_bound}, and
\eqref{eq:exp_third_moment_bound}. It remains to establish the same order for
their perturbed counterparts.

Since the random linear terms in \eqref{eq:exp_local_perturbations} are
centered and the deterministic quadratic terms agree,
\[
\E_\theta[\widetilde X]=\E_{Q_\theta}[\widetilde Y].
\]
Consequently,
\begin{equation}
\E_\theta[X_\beta]-\E_{Q_\theta}[Y_\beta]
=
\E_\theta[X]-\E_{Q_\theta}[Y],
\label{eq:exp_perturbed_mean_difference}
\end{equation}
so the perturbed mean discrepancy is controlled by
\eqref{eq:exp_mean_bound}.

For the covariance, bilinearity and the product structure give
\begin{equation}
\begin{aligned}
&\cov_\theta(X_\beta)-\cov_{Q_\theta}(Y_\beta)
\\
&\quad=
\cov_\theta(X)-\cov_{Q_\theta}(Y)
+\beta\sum_{m=1}^{M_n}
\Bigl\{
\cov_\theta(X_m,\widetilde X_m)
-\cov_{Q_\theta}(Y_m,\widetilde Y_m)
\Bigr\}
\\
&\qquad+
\beta\sum_{m=1}^{M_n}
\Bigl\{
\cov_\theta(\widetilde X_m,X_m)
-\cov_{Q_\theta}(\widetilde Y_m,Y_m)
\Bigr\}
\\
&\qquad+
\beta^2\sum_{m=1}^{M_n}
\Bigl\{
\cov_\theta(\widetilde X_m)
-\cov_{Q_\theta}(\widetilde Y_m)
\Bigr\}.
\end{aligned}
\label{eq:exp_perturbed_covariance_expansion}
\end{equation}
Each covariance in \eqref{eq:exp_perturbed_covariance_expansion} denotes its
kernel on $T_n\times T_n$. By symmetry, the two linear sums are controlled by
\eqref{eq:exp_directional_cov_X}, while the quadratic sum is controlled by
\eqref{eq:exp_directional_cov_tilde}. Together with
\eqref{eq:exp_cov_bound}, this yields, uniformly for $|\beta|\le\eps_0$,
\begin{equation}
\bigl\|
\cov_\theta(X_\beta)-\cov_{Q_\theta}(Y_\beta)
\bigr\|_\infty
\lesssim M_nr_n^3.
\label{eq:exp_perturbed_covariance_bound}
\end{equation}

Finally,
\begin{equation}
X_{\beta,m}-\E_\theta[X_{\beta,m}]
=
(X_m-\E_\theta[X_m])
+\beta(\widetilde X_m-\E_\theta[\widetilde X_m]).
\label{eq:exp_perturbed_centered_block}
\end{equation}
The inequality $\|x+y\|^3\le4(\|x\|^3+\|y\|^3)$, together with
\eqref{eq:exp_perturbed_centered_block},
\eqref{eq:exp_local_third_moment_bound}, and
\eqref{eq:exp_perturbation_third_moment}, gives
\begin{equation}
\sum_{m=1}^{M_n}
\E_\theta\left[
\|X_{\beta,m}-\E_\theta[X_{\beta,m}]\|_\infty^3
\right]
\lesssim M_nr_n^3.
\label{eq:exp_perturbed_third_moment_bound}
\end{equation}
uniformly for $|\beta|\le\eps_0$. Substituting
\eqref{eq:exp_perturbed_mean_difference}--
\eqref{eq:exp_perturbed_third_moment_bound} into
\eqref{eq:posterior_comparison_delta2} gives
\begin{equation}
\Delta_2(\theta,\beta)\lesssim M_nr_n^3.
\label{eq:exp_Delta2_bound}
\end{equation}

\emph{Conclusion.}
Since $M_nr_n^2\lesssim n$, equations
\eqref{eq:exp_Delta1_bound} and \eqref{eq:exp_Delta2_bound} imply the targets
in \eqref{eq:exp_posterior_Delta_targets}.  Apply
Theorem~\ref{thm:Phi_comparison} with $\beta=\pm\eps$ and divide by $n$ to
obtain
\begin{align*}
\Phi_n^{\mathrm G}(\theta,-\eps)
-C\left(\eps+r_n+\frac{r_n}{\eps}\right)
&\le \Phi_n(\theta)
\\
&\le
\Phi_n^{\mathrm G}(\theta,\eps)
+C\left(\eps+r_n+\frac{r_n}{\eps}\right).
\end{align*}
For $0<\eps<1$, the term $r_n$ is absorbed by $r_n/\eps$, yielding the stated
bounds.
\end{proof}

\section{Proofs for Section~\ref{subsec:regular_local_log_likelihoods}}
\label{app:proofs_regular_local}

Throughout this section, we omit the superscript $(n)$ from the
local random elements $U_m^{(n)},V_m^{(n)},R_m^{(n)},S_m^{(n)},\Delta_m^{(n)}$,
and use $\E_0$ and $\E_\theta$ to denote expectation under
$P_0=\bigotimes_{m=1}^{M_n}P_{0,m}^{(n)}$ and
$P_\theta=\bigotimes_{m=1}^{M_n}P_{\theta,m}^{(n)}$, respectively. For
$m\le M_n$, $t\in T_n$, and $\theta\in\Theta_n$, define
\[
U_m(t)\coloneqq\langle U_m,\eta_m(t)\rangle_{\cH_m},
\qquad
V_m(t)\coloneqq\langle V_m\eta_m(t),\eta_m(t)\rangle_{\cH_m},
\]
\[
S_m(\theta)\coloneqq\langle S_m,\eta_m(\theta)\rangle_{\cH_m},
\]
so that, by~\eqref{eq:regular_X_local} and~\eqref{eq:regular_likelihood_ratio},
\[
X_m(t)=U_m(t)-\tfrac12V_m(t)+R_m(t),
\qquad
\cL_m(\theta)
\coloneqq
\log\frac{\dd P_{\theta,m}^{(n)}}{\dd P_{0,m}^{(n)}}
=
S_m(\theta)+\Delta_m(\theta).
\]
Because $P_0$ and $P_\theta$ factor over $m$, and each of
$U_m(t),V_m(t),R_m(t)$ depends only on the $m$th coordinate $\omega_m$, every
integrable function $h$ of that coordinate satisfies the local
change-of-measure identity
$\E_\theta[h]=\E_0[h\,e^{\cL_m(\theta)}]$.

\begin{lemma}[Local change of measure]\label{lem:regular_change_of_measure}
Let $Q\ll Q_0$ be probability measures such that
$0<\dd Q/\dd Q_0<\infty$ holds $Q_0$-almost surely, and set
$\cL\coloneqq\log(\dd Q/\dd Q_0)$. Let $h$ be a random variable integrable
under both $Q_0$ and $Q$. Then
\begin{align}
\bigl|\E_Q[h]-\E_{Q_0}[h]\bigr|
&\lesssim
\E_{Q_0}[|h\cL|]+\E_Q[|h\cL|],
\label{eq:regular_cm_first}
\\
\bigl|\E_Q[h]-\E_{Q_0}[h]-\E_{Q_0}[h\cL]\bigr|
&\lesssim
\E_{Q_0}[|h\cL^2|]+\E_Q[|h\cL^2|].
\label{eq:regular_cm_second}
\end{align}
\end{lemma}

\begin{proof}
For each inequality, it suffices to consider the nontrivial case in which its
right-hand side is finite; otherwise the bound is immediate. In particular,
for~\eqref{eq:regular_cm_second}, finiteness of the right-hand side and
$|\cL|\le1+|\cL|^2$ imply
$\E_{Q_0}|h\cL|<\infty$, so the linear term
$\E_{Q_0}[h\cL]$ is well defined.

By the integral Taylor expansions of the exponential function,
\[
e^x=1+x\int_0^1e^{sx}\,\dd s,
\qquad
e^x=1+x+x^2\int_0^1(1-s)e^{sx}\,\dd s,
\]
applied with $x=\cL$ inside $\E_Q[h]=\E_{Q_0}[h\,e^{\cL}]$,
\[
\E_Q[h]-\E_{Q_0}[h]
=
\E_{Q_0}\Bigl[h\cL\int_0^1e^{s\cL}\,\dd s\Bigr],
\]
\[
\E_Q[h]-\E_{Q_0}[h]-\E_{Q_0}[h\cL]
=
\E_{Q_0}\Bigl[h\cL^2\int_0^1(1-s)e^{s\cL}\,\dd s\Bigr].
\]
For each fixed $s\in[0,1]$, Young's inequality $a^sb^{1-s}\le sa+(1-s)b$, for
$a,b\ge0$, applied with $a=e^{\cL}$ and $b=1$, gives
\[
e^{s\cL}\le s\,e^{\cL}+(1-s).
\]
Multiplying by $|h\cL|$ and integrating over $s\in[0,1]$,
\[
|h\cL|\int_0^1e^{s\cL}\,\dd s
\le
\tfrac12|h\cL|e^{\cL}+\tfrac12|h\cL|.
\]
In the nontrivial case of~\eqref{eq:regular_cm_first}, the right-hand side is
$Q_0$-integrable, so Fubini's theorem permits exchanging the $s$-integral and
the expectation in the first Taylor remainder.
Taking $\E_{Q_0}$ on both sides and using
$\E_{Q_0}[|h\cL|e^{\cL}]=\E_Q[|h\cL|]$, which holds since $e^{\cL}=\dd Q/\dd Q_0$,
\[
\bigl|\E_Q[h]-\E_{Q_0}[h]\bigr|
\le
\E_{Q_0}\Bigl[|h\cL|\int_0^1e^{s\cL}\,\dd s\Bigr]
\le
\tfrac12\E_Q[|h\cL|]+\tfrac12\E_{Q_0}[|h\cL|],
\]
which proves~\eqref{eq:regular_cm_first}. The same inequality, multiplied
instead by $(1-s)|h\cL^2|$ and integrated over $s\in[0,1]$, gives
\[
(1-s)|h\cL^2|e^{s\cL}
\le
s(1-s)|h\cL^2|e^{\cL}+(1-s)^2|h\cL^2|,
\]
In the nontrivial case of~\eqref{eq:regular_cm_second}, this bound likewise
provides the absolute integrability required for the corresponding exchange.
Hence, using $\int_0^1s(1-s)\,\dd s=\tfrac16$ and
$\int_0^1(1-s)^2\,\dd s=\tfrac13$,
\[
\bigl|\E_Q[h]-\E_{Q_0}[h]-\E_{Q_0}[h\cL]\bigr|
\le
\E_{Q_0}\Bigl[|h\cL^2|\int_0^1(1-s)e^{s\cL}\,\dd s\Bigr]
\le
\tfrac16\E_Q[|h\cL^2|]+\tfrac13\E_{Q_0}[|h\cL^2|],
\]
which proves~\eqref{eq:regular_cm_second}.
\end{proof}

\begin{lemma}[Local moment scale]\label{lem:regular_local_moment_scale}
Under Assumption~\ref{assump:regular_local_bounds}, uniformly over
$m\le M_n$, $\theta\in\Theta_n$, and $t\in T_n$,
\begin{align}
\E_0|U_m(t)|^2
\lesssim
r_n^2,
\qquad
\E_0|U_m(t)|^3+\E_\theta|U_m(t)|^3
\lesssim
r_n^2\delta_n(\theta),
\label{eq:lemma_scale_U}
\\
\E_0|S_m(\theta)|^2
\lesssim
r_n^2,
\qquad
\E_0|S_m(\theta)|^3+\E_\theta|S_m(\theta)|^3
\lesssim
r_n^2\delta_n(\theta),
\label{eq:lemma_scale_S}
\\
\sqrt{\E_0|V_m(t)|^2}
\lesssim
r_n\delta_n(\theta),
\qquad
\E_0|V_m(t)|^3+\E_\theta|V_m(t)|^3
\lesssim
r_n^2\delta_n(\theta),
\label{eq:lemma_scale_V}
\\
|\E_0[R_m(t)]|
\lesssim
r_n^2\delta_n(\theta),
\qquad
\sqrt{\E_0|R_m(t)|^2}
\lesssim
r_n\delta_n(\theta),
\label{eq:lemma_scale_R_mean}
\\
\E_0|R_m(t)|^3
+
\E_\theta\left[\sup_{t\in T_n}|R_m(t)|^3\right]
\lesssim
r_n^2\delta_n(\theta),
\label{eq:lemma_scale_R}
\\
\sqrt{\E_0|\Delta_m(\theta)|^2}
\lesssim
r_n\delta_n(\theta),
\qquad
\E_0|\Delta_m(\theta)|^3+\E_\theta|\Delta_m(\theta)|^3
\lesssim
r_n^2\delta_n(\theta).
\label{eq:lemma_scale_Delta}
\end{align}
\end{lemma}

\begin{proof}
By Cauchy--Schwarz, the operator-norm bound, and~\eqref{eq:feature_radius},
\[
|U_m(t)|\le r_n\|U_m\|_{\cH_m},
\qquad
|V_m(t)|\le r_n^2\|V_m\|_{\mathrm{op}},
\qquad
|S_m(\theta)|\le r_n\|S_m\|_{\cH_m}.
\]
The second-moment bounds for $U_m(t)$ and $S_m(\theta)$ follow immediately
from $\mathsf L_{m,n}\lesssim1$ in~\eqref{eq:regular_leading_moment}. Their
third-moment bounds, as well as those for $V_m(t)$ and $\Delta_m(\theta)$,
follow from the corresponding third-moment terms in
\eqref{eq:regular_moment_U}--\eqref{eq:regular_moment_Delta} and the displayed
pointwise bounds. The $P_0$ third moment of $R_m(t)$ and its $P_\theta$
sup-norm third moment both follow directly from
\eqref{eq:regular_moment_R}. This proves all the third-moment claims in the
lemma.

Finally, the mean bound for $R_m(t)$ and the square-root bounds for
$V_m(t)$, $R_m(t)$, and $\Delta_m(\theta)$ follow directly from the
corresponding mean and square-root terms in
\eqref{eq:regular_moment_V}--\eqref{eq:regular_moment_Delta}. Squaring the
latter bounds also gives
\[
\E_0|V_m(t)|^2
+
\E_0|R_m(t)|^2
+
\E_0|\Delta_m(\theta)|^2
\lesssim
r_n^2\delta_n(\theta)^2.
\]
\end{proof}

\subsection{Proof of Theorem~\ref{thm:regular_free_energy_univ}}
\label{app:proof_regular_free_energy}

We apply Theorem~\ref{thm:fe_comparison} with
$X=\sum_{m=1}^{M_n}X_m$ and its Gaussian counterpart
$Y=\sum_{m=1}^{M_n}Y_m$.
Assumption~\ref{assump:regular_local_bounds} and the feature-radius
bound~\eqref{eq:feature_radius} imply
$\E_\theta\|X_m\|_\infty^3<\infty$ for every $m\le M_n$, so
$F(\theta)$ is finite; $F^{\mathrm G}(\theta)$ is finite by Fernique's
theorem.
The next three lemmas provide blockwise bounds for
the mean, covariance, and third-moment terms in the comparison theorem. Their
combination yields Theorem~\ref{thm:regular_free_energy_univ}.

\begin{lemma}[Mean comparison]\label{lem:regular_mean_comparison}
Work under the assumptions of Theorem~\ref{thm:regular_free_energy_univ}.
For every $m\le M_n$ and $\theta\in\Theta_n$,
\begin{equation}
\sup_{t\in T_n}\bigl|\E_\theta[X_m(t)]-\E_{Q_\theta}[Y_m(t)]\bigr|
\lesssim
r_n^2\{\delta_n(\theta)+\delta_n(\theta)^2\}.
\label{eq:regular_mean_comparison_bound}
\end{equation}
\end{lemma}

\begin{proof}
By~\eqref{eq:Y_local_moments}, $\E_{Q_\theta}[Y_m(t)]=\langle A_m\eta_m(\theta),\eta_m(t)\rangle
-\tfrac12\langle K_m\eta_m(t),\eta_m(t)\rangle$,
so that
\begin{align}
\E_\theta[X_m(t)]-\E_{Q_\theta}[Y_m(t)]
&=
\bigl\{\E_\theta[U_m(t)]-\langle A_m\eta_m(\theta),\eta_m(t)\rangle\bigr\}
\notag
\\
&\quad
-\tfrac12
\bigl\{\E_\theta[V_m(t)]-\langle K_m\eta_m(t),\eta_m(t)\rangle\bigr\}
+
\E_\theta[R_m(t)].
\label{eq:mean_comparison_decomposition}
\end{align}
Apply Lemma~\ref{lem:regular_change_of_measure} to the local measures
$P_{0,m}^{(n)}$ and $P_{\theta,m}^{(n)}$, with $h=U_m(t)$ and
$\cL=\cL_m(\theta)$. Since $\E_0[U_m(t)]=0$,
\[
\E_\theta[U_m(t)]
=
\E_0[U_m(t)\cL_m(\theta)]
+
O\bigl(\E_0|U_m(t)\cL_m(\theta)^2|+\E_\theta|U_m(t)\cL_m(\theta)^2|\bigr).
\]
By~\eqref{eq:regular_likelihood_ratio}, $\cL_m(\theta)=S_m(\theta)+\Delta_m(\theta)$,
and, by~\eqref{eq:regular_operator_forms},
$\E_0[U_m(t)S_m(\theta)]=\langle A_m\eta_m(\theta),\eta_m(t)\rangle$, so
\begin{align}
\E_\theta[U_m(t)]-\langle A_m\eta_m(\theta),\eta_m(t)\rangle
&=
\E_0[U_m(t)\Delta_m(\theta)]
\notag
\\
&\quad
+
O\bigl(\E_0|U_m(t)\cL_m(\theta)^2|+\E_\theta|U_m(t)\cL_m(\theta)^2|\bigr).
\label{eq:mean_comparison_U}
\end{align}
Applying the same expansion with $h=V_m(t)$ and $h=R_m(t)$, and using
$\E_0[V_m(t)]=\langle K_m\eta_m(t),\eta_m(t)\rangle$
from~\eqref{eq:regular_operator_K}, gives
\begin{align}
\E_\theta[V_m(t)]-\langle K_m\eta_m(t),\eta_m(t)\rangle
&=
\E_0[V_m(t)\cL_m(\theta)]
\notag
\\
&\quad
+
O\bigl(\E_0|V_m(t)\cL_m(\theta)^2|+\E_\theta|V_m(t)\cL_m(\theta)^2|\bigr),
\label{eq:mean_comparison_V}
\\
\E_\theta[R_m(t)]
&=
\E_0[R_m(t)]
+
\E_0[R_m(t)\cL_m(\theta)]
\notag
\\
&\quad
+
O\bigl(\E_0|R_m(t)\cL_m(\theta)^2|+\E_\theta|R_m(t)\cL_m(\theta)^2|\bigr).
\label{eq:mean_comparison_R}
\end{align}
We bound the terms in
\eqref{eq:mean_comparison_U}--\eqref{eq:mean_comparison_R} uniformly over
$t\in T_n$, using $\cL_m(\theta)=S_m(\theta)+\Delta_m(\theta)$ and
Lemma~\ref{lem:regular_local_moment_scale}.

\emph{Second-order terms.} By Cauchy--Schwarz and
\eqref{eq:lemma_scale_U} and~\eqref{eq:lemma_scale_Delta},
\[
\bigl|\E_0[U_m(t)\Delta_m(\theta)]\bigr|
\le
\sqrt{\E_0|U_m(t)|^2}\,\sqrt{\E_0|\Delta_m(\theta)|^2}
\lesssim
r_n\times r_n\delta_n(\theta)
=
r_n^2\delta_n(\theta).
\]
Similarly, splitting $\cL_m(\theta)=S_m(\theta)+\Delta_m(\theta)$ and using
\eqref{eq:lemma_scale_S}, \eqref{eq:lemma_scale_Delta},
\eqref{eq:lemma_scale_V}, and~\eqref{eq:lemma_scale_R_mean},
\begin{align*}
\bigl|\E_0[V_m(t)\cL_m(\theta)]\bigr|
&\lesssim
\sqrt{\E_0|V_m(t)|^2}\Bigl(\sqrt{\E_0|S_m(\theta)|^2}+\sqrt{\E_0|\Delta_m(\theta)|^2}\Bigr)
\\
&\lesssim
r_n\delta_n(\theta)\bigl(r_n+r_n\delta_n(\theta)\bigr)
\lesssim 
r_n^2\{\delta_n(\theta)+\delta_n(\theta)^2\}
,
\end{align*}
\begin{align*}
\bigl|\E_0[R_m(t)\cL_m(\theta)]\bigr|
&\lesssim
\sqrt{\E_0|R_m(t)|^2}\Bigl(\sqrt{\E_0|S_m(\theta)|^2}+\sqrt{\E_0|\Delta_m(\theta)|^2}\Bigr)
\\
&\lesssim
r_n\delta_n(\theta)\bigl(r_n+r_n\delta_n(\theta)\bigr)
\lesssim 
r_n^2\{\delta_n(\theta)+\delta_n(\theta)^2\}
.
\end{align*}

\emph{Third-order terms.} Let
$Z(t)\in\{U_m(t),V_m(t),R_m(t)\}$. For
$\bullet\in\{0,\theta\}$, Young's inequality gives
\[
\E_\bullet|Z(t)\cL_m(\theta)^2|
\le
\tfrac13\E_\bullet|Z(t)|^3+\tfrac23\E_\bullet|\cL_m(\theta)|^3.
\]
Equations~\eqref{eq:lemma_scale_U}, \eqref{eq:lemma_scale_V}, and
\eqref{eq:lemma_scale_R} bound the first term by
$r_n^2\delta_n(\theta)$. Moreover,
$|\cL_m(\theta)|^3\lesssim|S_m(\theta)|^3+|\Delta_m(\theta)|^3$, so
\eqref{eq:lemma_scale_S} and~\eqref{eq:lemma_scale_Delta} give the same bound
for the second term. Hence,
\[
\E_\bullet|Z(t)\cL_m(\theta)^2|
\lesssim
r_n^2\delta_n(\theta),
\]
uniformly over all three choices of $Z$.

Finally, \eqref{eq:lemma_scale_R_mean} gives
$|\E_0[R_m(t)]|\lesssim r_n^2\delta_n(\theta)$. Substituting these estimates
into~\eqref{eq:mean_comparison_decomposition} yields
\[
\sup_{t\in T_n}\bigl|\E_\theta[X_m(t)]-\E_{Q_\theta}[Y_m(t)]\bigr|
\lesssim
r_n^2\{\delta_n(\theta)+\delta_n(\theta)^2\}
.
\]
\end{proof}

\begin{lemma}[Covariance comparison]\label{lem:regular_cov_comparison}
Work under the assumptions of Theorem~\ref{thm:regular_free_energy_univ}.
For every $m\le M_n$ and $\theta\in\Theta_n$,
\begin{equation}
\sup_{s,t\in T_n}
\bigl|\cov_\theta(X_m(s),X_m(t))-\cov_{Q_\theta}(Y_m(s),Y_m(t))\bigr|
\lesssim
r_n^2\{\delta_n(\theta)+\delta_n(\theta)^2\}
.
\label{eq:regular_cov_comparison_bound}
\end{equation}
\end{lemma}

\begin{proof}
Write
$\tilde R_m(t)\coloneqq-\tfrac12V_m(t)+R_m(t)$. By
\eqref{eq:Y_local_covariance} and~\eqref{eq:regular_operator_forms},
\[
\cov_{Q_\theta}(Y_m(s),Y_m(t))
=
\langle\Sigma_m\eta_m(s),\eta_m(t)\rangle
=
\E_0[U_m(s)U_m(t)]
=
\cov_0(U_m(s),U_m(t)),
\]
where the last equality uses $\E_0[U_m]=0$. Since
$X_m=U_m+\tilde R_m$, it follows that
\begin{equation}
\begin{aligned}
&\cov_\theta(X_m(s),X_m(t))-\cov_{Q_\theta}(Y_m(s),Y_m(t))
\\
&\quad=
\bigl[\cov_\theta(U_m(s),U_m(t))-\cov_0(U_m(s),U_m(t))\bigr]
+
\cov_\theta(\tilde R_m(s),\tilde R_m(t))
\\
&\qquad\qquad
+
\cov_\theta(U_m(s),\tilde R_m(t))
+
\cov_\theta(U_m(t),\tilde R_m(s)).
\label{eq:cov_comparison_decomposition}
\end{aligned}
\end{equation}

\emph{First bracket.} For the first bracket
in~\eqref{eq:cov_comparison_decomposition}, write out the two covariances
directly. By definition,
\[
\cov_\theta(U_m(s),U_m(t))
=
\E_\theta[U_m(s)U_m(t)]-\E_\theta[U_m(s)]\E_\theta[U_m(t)],
\]
while, since $\E_0[U_m]=0$,
\[
\cov_0(U_m(s),U_m(t))
=
\E_0[U_m(s)U_m(t)].
\]
Subtracting the two,
\begin{align*}
&\cov_\theta(U_m(s),U_m(t))-\cov_0(U_m(s),U_m(t))
\\
&\quad=
\bigl\{\E_\theta[U_m(s)U_m(t)]-\E_0[U_m(s)U_m(t)]\bigr\}
-
\E_\theta[U_m(s)]\E_\theta[U_m(t)].
\end{align*}
Applying~\eqref{eq:regular_cm_first} to $h=U_m(s)U_m(t)$ bounds the first
term,
\begin{align*}
&\bigl|\E_\theta[U_m(s)U_m(t)]-\E_0[U_m(s)U_m(t)]\bigr|
\\
&\quad\lesssim
\E_0|U_m(s)U_m(t)\cL_m(\theta)|+\E_\theta|U_m(s)U_m(t)\cL_m(\theta)|,
\end{align*}
so that
\begin{align*}
&\bigl|\cov_\theta(U_m(s),U_m(t))-\cov_0(U_m(s),U_m(t))\bigr|
\\
&\quad\lesssim
\E_0|U_m(s)U_m(t)\cL_m(\theta)|
+\E_\theta|U_m(s)U_m(t)\cL_m(\theta)|
+\bigl|\E_\theta[U_m(s)]\E_\theta[U_m(t)]\bigr|.
\end{align*}
By H\"older's inequality with exponents $(3,3,3)$, applied to the three
distinct factors $U_m(s)$, $U_m(t)$, $\cL_m(\theta)$,
\begin{equation}
\begin{aligned}
\E_\bullet|U_m(s)U_m(t)\cL_m(\theta)|
&\le
\bigl(\E_\bullet|U_m(s)|^3\bigr)^{1/3}
\bigl(\E_\bullet|U_m(t)|^3\bigr)^{1/3}
\\
&\qquad\times
\bigl(\E_\bullet|\cL_m(\theta)|^3\bigr)^{1/3},
\qquad \bullet\in\{0,\theta\}.
\end{aligned}
\label{eq:cov_first_trilinear}
\end{equation}
By~\eqref{eq:lemma_scale_U} and, as in the mean comparison,
$\E_\bullet|\cL_m(\theta)|^3\lesssim r_n^2\delta_n(\theta)$, each factor on
the right is $\lesssim(r_n^2\delta_n(\theta))^{1/3}$, so the product is
$\lesssim r_n^2\delta_n(\theta)$.
It remains to bound $\E_\theta[U_m(s)]\E_\theta[U_m(t)]$. By~\eqref{eq:regular_cm_first} with
$h=U_m(t)$ and $\E_0[U_m]=0$,
\begin{equation}
|\E_\theta[U_m(t)]|
=
|\E_\theta[U_m(t)]-\E_0[U_m(t)]|
\lesssim
\E_0|U_m(t)\cL_m(\theta)|+\E_\theta|U_m(t)\cL_m(\theta)|.
\label{eq:cov_U_mean_cm}
\end{equation}
Multiplying by $\E_\theta|U_m(s)|$ and applying Young's inequality
$ab\le\tfrac23a^{3/2}+\tfrac13b^3$ (conjugate exponents $3/2,3$) to each of
the two resulting products,
\begin{align*}
\bigl|\E_\theta[U_m(t)]\E_\theta[U_m(s)]\bigr|
&\lesssim
\bigl(\E_0|U_m(t)\cL_m(\theta)|\bigr)^{3/2}
+
\bigl(\E_\theta|U_m(t)\cL_m(\theta)|\bigr)^{3/2}
\\
&\quad
+
\bigl(\E_\theta|U_m(s)|\bigr)^3.
\end{align*}
We convert each term $(\E_\bullet|U_m(t)\cL_m(\theta)|)^{3/2}$, for
$\bullet\in\{0,\theta\}$, into a sum of third moments in two steps. First, by
Jensen's inequality applied to the convex function $x\mapsto x^{3/2}$ on
$x\ge0$,
\[
\bigl(\E_\bullet|U_m(t)\cL_m(\theta)|\bigr)^{3/2}
\le
\E_\bullet\bigl[|U_m(t)\cL_m(\theta)|^{3/2}\bigr]
=
\E_\bullet\bigl[|U_m(t)|^{3/2}|\cL_m(\theta)|^{3/2}\bigr].
\]
Second, applying Young's inequality $ab\le\tfrac12a^2+\tfrac12b^2$ pointwise,
with $a=|U_m(t)|^{3/2}$ and $b=|\cL_m(\theta)|^{3/2}$, and then taking
$\E_\bullet$,
\[
\E_\bullet\bigl[|U_m(t)|^{3/2}|\cL_m(\theta)|^{3/2}\bigr]
\le
\tfrac12\E_\bullet|U_m(t)|^3+\tfrac12\E_\bullet|\cL_m(\theta)|^3.
\]
Combining the two steps,
\begin{equation}
\bigl(\E_\bullet|U_m(t)\cL_m(\theta)|\bigr)^{3/2}
\lesssim
\E_\bullet|U_m(t)|^3+\E_\bullet|\cL_m(\theta)|^3,
\label{eq:cov_mixed_three_halves}
\end{equation}
and, by Jensen's inequality, $(\E_\theta|U_m(s)|)^3\le\E_\theta|U_m(s)|^3$.
Every term on the right is
now a third moment of $U_m(t)$, $U_m(s)$, or $\cL_m(\theta)=S_m(\theta)+\Delta_m(\theta)$,
under $\E_0$ or $\E_\theta$, so~\eqref{eq:lemma_scale_U},
\eqref{eq:lemma_scale_S}, and~\eqref{eq:lemma_scale_Delta} give
\[
\bigl|\E_\theta[U_m(s)]\E_\theta[U_m(t)]\bigr|
\lesssim
r_n^2\delta_n(\theta).
\]
Combining this estimate with~\eqref{eq:cov_first_trilinear},
\begin{equation}
\bigl|\cov_\theta(U_m(s),U_m(t))-\cov_0(U_m(s),U_m(t))\bigr|
\lesssim
r_n^2\delta_n(\theta).
\label{eq:cov_bracket1}
\end{equation}
Throughout the remaining three brackets, recall
$\tilde R_m(t)\coloneqq-\tfrac12V_m(t)+R_m(t)$. By the triangle inequality
$|{-\tfrac12}V_m(t)+R_m(t)|^k\lesssim|V_m(t)|^k+|R_m(t)|^k$ and
Lemma~\ref{lem:regular_local_moment_scale},
\begin{equation}
\sqrt{\E_0|\tilde R_m(t)|^2}
\lesssim r_n\delta_n(\theta),
\qquad
\E_0|\tilde R_m(t)|^3+\E_\theta|\tilde R_m(t)|^3
\lesssim r_n^2\delta_n(\theta).
\label{eq:cov_remainder_moments}
\end{equation}

\emph{Second bracket.} For the second bracket
in~\eqref{eq:cov_comparison_decomposition}, we first control the second
moment of $\tilde R_m$ under $P_\theta$. Applying
\eqref{eq:regular_cm_first} with $h=\tilde R_m(t)^2$,
\[
\E_\theta|\tilde R_m(t)|^2
\le
\E_0|\tilde R_m(t)|^2+\E_0|\tilde R_m(t)^2\cL_m(\theta)|+\E_\theta|\tilde R_m(t)^2\cL_m(\theta)|.
\]
By Young's inequality $a^2b\le\tfrac23a^3+\tfrac13b^3$, as in the
third-order terms in the mean comparison, applied with $a=|\tilde R_m(t)|$
and $b=|\cL_m(\theta)|$,
\[
\E_\bullet|\tilde R_m(t)^2\cL_m(\theta)|
\le
\tfrac23\E_\bullet|\tilde R_m(t)|^3+\tfrac13\E_\bullet|\cL_m(\theta)|^3,
\qquad
\bullet\in\{0,\theta\},
\]
so that
\[
\E_\theta|\tilde R_m(t)|^2
\lesssim
r_n^2\{\delta_n(\theta)+\delta_n(\theta)^2\}
,
\]
and likewise with $t$ replaced by $s$. Hence, by Cauchy--Schwarz,
\begin{equation}
\bigl|\cov_\theta(\tilde R_m(s),\tilde R_m(t))\bigr|
\le
\sqrt{\var_\theta(\tilde R_m(s))\var_\theta(\tilde R_m(t))}
\lesssim
r_n^2\{\delta_n(\theta)+\delta_n(\theta)^2\}
.
\label{eq:cov_bracket2}
\end{equation}

\emph{Third and fourth brackets.} For the remaining two brackets
in~\eqref{eq:cov_comparison_decomposition}, write
\[
\cov_\theta(U_m(s),\tilde R_m(t))
=
\E_\theta[U_m(s)\tilde R_m(t)]-\E_\theta[U_m(s)]\E_\theta[\tilde R_m(t)].
\]
Applying~\eqref{eq:regular_cm_first} with
$h=U_m(s)\tilde R_m(t)$ gives
\begin{align*}
\bigl|\E_\theta[U_m(s)\tilde R_m(t)]\bigr|
&\le
\bigl|\E_0[U_m(s)\tilde R_m(t)]\bigr|
\\
&\quad
+
\E_0|U_m(s)\tilde R_m(t)\cL_m(\theta)|
+
\E_\theta|U_m(s)\tilde R_m(t)\cL_m(\theta)|
\\
&\lesssim
r_n^2\delta_n(\theta).
\end{align*}
Indeed, Cauchy--Schwarz,~\eqref{eq:lemma_scale_U}, and
\eqref{eq:cov_remainder_moments} give
\[
\bigl|\E_0[U_m(s)\tilde R_m(t)]\bigr|
\le
\sqrt{\E_0|U_m(s)|^2}\sqrt{\E_0|\tilde R_m(t)|^2}
\lesssim r_n^2\delta_n(\theta),
\]
while $|abc|\le\tfrac13(|a|^3+|b|^3+|c|^3)$ bounds each of the two
change-of-measure terms by the corresponding third moments of $U_m(s)$,
$\tilde R_m(t)$, and $\cL_m(\theta)$.
For the product of means, since $\E_0[U_m]=0$,
applying~\eqref{eq:regular_cm_first} with $h=U_m(s)$ gives
\[
|\E_\theta[U_m(s)]|
\lesssim
\E_0|U_m(s)\cL_m(\theta)|+\E_\theta|U_m(s)\cL_m(\theta)|,
\]
as recorded in~\eqref{eq:cov_U_mean_cm}. Multiplying by
$\E_\theta|\tilde R_m(t)|$ and applying Young's inequality
$ab\le\tfrac23a^{3/2}+\tfrac13b^3$ (conjugate exponents $3/2,3$),
\[
\bigl|\E_\theta[U_m(s)]\E_\theta[\tilde R_m(t)]\bigr|
\lesssim
\bigl(\E_0|U_m(s)\cL_m(\theta)|\bigr)^{3/2}+\bigl(\E_\theta|U_m(s)\cL_m(\theta)|\bigr)^{3/2}
+\bigl(\E_\theta|\tilde R_m(t)|\bigr)^3.
\]
The first two terms are $\lesssim r_n^2\delta_n(\theta)$ by
\eqref{eq:cov_mixed_three_halves}, \eqref{eq:lemma_scale_U},
\eqref{eq:lemma_scale_S}, and~\eqref{eq:lemma_scale_Delta}. Moreover,
Jensen's inequality and~\eqref{eq:cov_remainder_moments} give
$(\E_\theta|\tilde R_m(t)|)^3\le
\E_\theta|\tilde R_m(t)|^3\lesssim r_n^2\delta_n(\theta)$.
Combining these estimates, and exchanging $s$ and $t$ for the symmetric
bracket,
\begin{equation}
\bigl|\cov_\theta(U_m(s),\tilde R_m(t))\bigr|,\
\bigl|\cov_\theta(U_m(t),\tilde R_m(s))\bigr|
\ \lesssim\
r_n^2\delta_n(\theta).
\label{eq:cov_bracket34}
\end{equation}

Substituting~\eqref{eq:cov_bracket1}, \eqref{eq:cov_bracket2},
and~\eqref{eq:cov_bracket34} into~\eqref{eq:cov_comparison_decomposition} gives
\[
\sup_{s,t\in T_n}
\bigl|\cov_\theta(X_m(s),X_m(t))-\cov_{Q_\theta}(Y_m(s),Y_m(t))\bigr|
\lesssim
r_n^2\{\delta_n(\theta)+\delta_n(\theta)^2\}
.
\]
\end{proof}

\begin{lemma}[Third-moment bound]\label{lem:regular_third_moment}
Work under the assumptions of Theorem~\ref{thm:regular_free_energy_univ}.
For every $m\le M_n$ and $\theta\in\Theta_n$,
\begin{equation}
\E_\theta\|X_m-\E_\theta[X_m]\|_\infty^3
\lesssim
r_n^2\delta_n(\theta).
\label{eq:regular_third_moment_bound}
\end{equation}
\end{lemma}

\begin{proof}
By Cauchy--Schwarz and~\eqref{eq:feature_radius}, pathwise,
\[
\sup_{t\in T_n}|U_m(t)-\E_\theta U_m(t)|\le r_n\|U_m-\E_\theta U_m\|_{\cH_m},
\]
\[
\sup_{t\in T_n}|V_m(t)-\E_\theta V_m(t)|\le r_n^2\|V_m-\E_\theta V_m\|_{\mathrm{op}}.
\]
Cubing, taking $\E_\theta$, and using $\|a-b\|^k\lesssim\|a\|^k+\|b\|^k$
together with Jensen's inequality for the centering term,
\begin{align*}
\E_\theta\left[\sup_{t\in T_n}
|U_m(t)-\E_\theta U_m(t)|^3\right]
&\lesssim r_n^3\E_\theta\|U_m\|_{\cH_m}^3,
\\
\E_\theta\left[\sup_{t\in T_n}
|V_m(t)-\E_\theta V_m(t)|^3\right]
&\lesssim r_n^6\E_\theta\|V_m\|_{\mathrm{op}}^3.
\end{align*}
By the third-moment terms of $\cE^U_{m,n}(\theta)$
in~\eqref{eq:regular_moment_U} and of $\cE^V_{m,n}(\theta)$
in~\eqref{eq:regular_moment_V}, $\E_\theta\|U_m\|_{\cH_m}^3\lesssim\delta_n(\theta)/r_n$
and $\E_\theta\|V_m\|_{\mathrm{op}}^3\lesssim\delta_n(\theta)/r_n^4$, so
\[
\E_\theta\left[\sup_{t\in T_n}
|U_m(t)-\E_\theta U_m(t)|^3\right]
\lesssim
r_n^2\delta_n(\theta),
\]
\[
\E_\theta\left[\sup_{t\in T_n}
|V_m(t)-\E_\theta V_m(t)|^3\right]
\lesssim
r_n^2\delta_n(\theta).
\]
For the remainder term, the triangle inequality gives
\[
\|R_m-\E_\theta R_m\|_\infty^3
\lesssim
\sup_{t\in T_n}|R_m(t)|^3
+
\Bigl(\sup_{t\in T_n}|\E_\theta R_m(t)|\Bigr)^3.
\]
Since
$\sup_t|\E_\theta R_m(t)|\le\E_\theta[\sup_t|R_m(t)|]$, Jensen's inequality
and~\eqref{eq:regular_moment_R} yield
\[
\E_\theta\|R_m-\E_\theta R_m\|_\infty^3
\lesssim
\E_\theta\Bigl[\sup_{t\in T_n}|R_m(t)|^3\Bigr]
\lesssim
r_n^2\delta_n(\theta).
\]
Combining the three
pieces via $\|a+b+c\|_\infty^3\lesssim\|a\|_\infty^3+\|b\|_\infty^3+\|c\|_\infty^3$,
\[
\E_\theta\|X_m-\E_\theta[X_m]\|_\infty^3
\lesssim
r_n^2\delta_n(\theta).
\]
\end{proof}

We now combine Lemmas~\ref{lem:regular_mean_comparison},
\ref{lem:regular_cov_comparison}, and~\ref{lem:regular_third_moment}.
The processes $X_1,\ldots,X_{M_n}$ are independent under $P_\theta$, and
$Y_1,\ldots,Y_{M_n}$ are independent under $Q_\theta$. Thus, the means and
covariances of the sums are obtained by summing their blockwise counterparts.
Consequently, the three lemmas and $M_nr_n^2\lesssim n$ give
\begin{align*}
&\bigl\|\E_\theta[X]-\E_{Q_\theta}[Y]\bigr\|_\infty
+
\bigl\|\cov_\theta(X)-\cov_{Q_\theta}(Y)\bigr\|_\infty
+
\sum_{m=1}^{M_n}\E_\theta\|X_m-\E_\theta[X_m]\|_\infty^3
\\
&\quad\lesssim
M_nr_n^2\{\delta_n(\theta)+\delta_n(\theta)^2\}
\lesssim
n\{\delta_n(\theta)+\delta_n(\theta)^2\}.
\end{align*}
By Theorem~\ref{thm:fe_comparison},
\[
|F(\theta)-F^{\mathrm G}(\theta)|
\lesssim
n\{\delta_n(\theta)+\delta_n(\theta)^2\}.
\]
Dividing by $n$ gives
\[
|F_n(\theta)-F_n^{\mathrm G}(\theta)|
\lesssim
\delta_n(\theta)+\delta_n(\theta)^2,
\]
as claimed.

\subsection{Proof of Theorem~\ref{thm:regular_posterior_comparison}}
\label{app:proof_regular_posterior}

The Gaussian perturbation is already specified in
Theorem~\ref{thm:regular_posterior_comparison}. We construct its non-Gaussian
counterpart and apply Theorem~\ref{thm:Phi_comparison}. Once the required
independence and convexity conditions are verified, it remains to control two
errors: $\Delta_1$ measures how accurately the perturbations generate the target
posterior functional, whereas $\Delta_2$ compares the free energies of the
unperturbed and perturbed processes.
We use the product operators in
\eqref{eq:regular_posterior_product_operators} and suppress the superscript
$(n)$ throughout this subsection.

\emph{Construction of the non-Gaussian perturbation.}
For $m\le M_n$, define
\begin{align*}
q_m(t)
&\coloneqq
a\langle H_m[\eta_m(\theta)-\eta_m(t)],
\eta_m(\theta)-\eta_m(t)\rangle_{\cH_m},
\\
\zeta_m(t)
&\coloneqq
\langle U_m-\E_\theta U_m,B_m\eta_m(t)\rangle_{\cH_m},
\\
\widetilde X_m(t)
&\coloneqq \zeta_m(t)+q_m(t).
\end{align*}
The non-Gaussian perturbation is
\begin{align}
X_{\beta,m}(t)
&\coloneqq X_m(t)+\beta\widetilde X_m(t),
\qquad
X_\beta(t)\coloneqq\sum_{m=1}^{M_n}X_{\beta,m}(t).
\label{eq:regular_X_beta_pert}
\end{align}
For the Gaussian perturbation~\eqref{eq:Y_beta}--\eqref{eq:Y_tilde}, write
\[
\zeta_m^{\mathrm G}(t)
\coloneqq\langle\xi_m,B_m\eta_m(t)\rangle_{\cH_m},
\qquad
\widetilde Y_m(t)=\zeta_m^{\mathrm G}(t)+q_m(t).
\]
Thus, the two perturbations have the same deterministic part, while $\zeta_m$
and $\zeta_m^{\mathrm G}$ are their respective centered linear parts. The
assumption $\mathfrak h_n=O(1)$ gives
$\|H_m\|_{\mathrm{op}}+\|B_m\|_{\mathrm{op}}\lesssim1$ uniformly in $m$.

We verify the hypotheses of
Theorem~\ref{thm:Phi_comparison}. For each fixed $\theta$, the pairs
$(X_{\beta,m},\dot X_{\beta,m})=(X_{\beta,m},\widetilde X_m)$ are
independent under $P_\theta$, and $(Y_\beta,\dot Y_\beta)$ is jointly
Gaussian because both are affine functions of the same Gaussian
element, giving~(a). Since $\dot X_\beta=\widetilde X$ and
$\dot Y_\beta=\widetilde Y$ are independent of $\beta$, the
deterministic bound on $q_m$ and the moment bounds on $\zeta_m$ and
$\zeta_m^{\mathrm G}$ (the latter via Fernique's theorem) give~(b).
Finally, $X_\beta$ and $Y_\beta$ are affine in $\beta$, so the
perturbed free energies are convex, giving~(c). It remains to prove, uniformly for
$|\beta|\le\eps$ with $\eps>0$ sufficiently small,
\begin{equation}
\frac1n\Delta_1(\theta,\beta)
\lesssim |\beta|+\delta_n(\theta),
\qquad
\frac1n\Delta_2(\theta,\beta)
\lesssim 
\widetilde\delta_n(\theta)
.
\label{eq:regular_posterior_Delta_targets}
\end{equation}

\emph{Auxiliary estimates for $\Delta_1$ and $\Delta_2$.}
We first record two blockwise estimates used to prove
\eqref{eq:regular_posterior_Delta_targets}. Uniformly over
$m\le M_n$ and $s,t\in T_n$,
\begin{align}
&\left|
\cov_\theta\bigl(X_m(s),\zeta_m(t)\bigr)
-
\cov_{Q_\theta}\bigl(Y_m(s),\zeta_m^{\mathrm G}(t)\bigr)
\right|
\lesssim r_n^2\delta_n(\theta),
\label{eq:regular_directional_cov_X}
\\
&\left|
\cov_\theta\bigl(\zeta_m(s),\zeta_m(t)\bigr)
-
\cov_{Q_\theta}\bigl(\zeta_m^{\mathrm G}(s),\zeta_m^{\mathrm G}(t)\bigr)
\right|
\lesssim r_n^2\delta_n(\theta).
\label{eq:regular_directional_cov_L}
\end{align}
To verify these bounds, set
$U_m^B(t)\coloneqq\langle U_m,B_m\eta_m(t)\rangle_{\cH_m}$. Since
$\|B_m\|_{\mathrm{op}}\lesssim1$, the pointwise second- and third-moment
bounds for $U_m^B(t)$ are the same as those for $U_m(t)$ in
Lemma~\ref{lem:regular_local_moment_scale}. Moreover, $\E_0U_m^B(t)=0$.
Since
$\zeta_m(t)=U_m^B(t)-\E_\theta U_m^B(t)$, invariance of covariance under
deterministic shifts gives
\begin{align}
\cov_\theta(X_m(s),\zeta_m(t))
&=
\cov_\theta(X_m(s),U_m^B(t))
\notag\\
&=
\E_\theta[X_m(s)U_m^B(t)]
-\E_\theta[X_m(s)]\E_\theta[U_m^B(t)],
\label{eq:regular_nongaussian_cov_XB}
\\
\cov_\theta(\zeta_m(s),\zeta_m(t))
&=
\cov_\theta(U_m^B(s),U_m^B(t))
\notag\\
&=
\E_\theta[U_m^B(s)U_m^B(t)]
-\E_\theta[U_m^B(s)]\E_\theta[U_m^B(t)].
\label{eq:regular_nongaussian_cov_BB}
\end{align}
Because $H_m=\Sigma_mB_m$, taking adjoints gives
$H_m^*=B_m^*\Sigma_m$; moreover, $H_m^*=H_m$ by self-adjointness.
Consequently,
\begin{align}
\cov_{Q_\theta}(Y_m(s),\zeta_m^{\mathrm G}(t))
&=
\E_{Q_\theta}\left[
\bigl(Y_m(s)-\E_{Q_\theta}Y_m(s)\bigr)\zeta_m^{\mathrm G}(t)
\right]
\notag\\
&=
\E_{Q_\theta}\left[
\langle\xi_m,\eta_m(s)\rangle_{\cH_m}
\langle\xi_m,B_m\eta_m(t)\rangle_{\cH_m}
\right]
\notag\\
&=
\langle\Sigma_m\eta_m(s),B_m\eta_m(t)\rangle_{\cH_m}
=
\langle B_m^*\Sigma_m\eta_m(s),\eta_m(t)\rangle_{\cH_m}
\notag\\
&=
\langle H_m^*\eta_m(s),\eta_m(t)\rangle_{\cH_m}
=
\langle H_m\eta_m(s),\eta_m(t)\rangle_{\cH_m}
\notag\\
&=
\E_0[U_m(s)U_m^B(t)],
\label{eq:regular_gaussian_cov_XB}
\\
\cov_{Q_\theta}(\zeta_m^{\mathrm G}(s),\zeta_m^{\mathrm G}(t))
&=
\langle\Sigma_mB_m\eta_m(s),B_m\eta_m(t)\rangle_{\cH_m}
\notag\\
&=
\E_0[U_m^B(s)U_m^B(t)].
\label{eq:regular_gaussian_cov_BB}
\end{align}

We now use the change-of-measure argument underlying
Lemma~\ref{lem:regular_cov_comparison}. For clarity, we record it in a
reusable form that applies to both~\eqref{eq:regular_directional_cov_X}
and~\eqref{eq:regular_directional_cov_L}. Let $A$ and $C$ be any two
variables of the form $U_m(v)$ or $U_m^B(v)$, $v\in T_n$. In particular,
$\E_0A=\E_0C=0$. Expanding the covariance and
applying~\eqref{eq:regular_cm_first} with $h=AC$ gives
\begin{align*}
&\left|\cov_\theta(A,C)-\E_0[AC]\right|
\\
&\quad\le
\E_0|AC\cL_m(\theta)|
+\E_\theta|AC\cL_m(\theta)|
+|\E_\theta A\,\E_\theta C|.
\end{align*}
For $\bullet\in\{0,\theta\}$, H\"older's inequality yields
\[
\E_\bullet|AC\cL_m(\theta)|
\le
(\E_\bullet|A|^3)^{1/3}
(\E_\bullet|C|^3)^{1/3}
(\E_\bullet|\cL_m(\theta)|^3)^{1/3}
\lesssim r_n^2\delta_n(\theta),
\]
where we used Lemma~\ref{lem:regular_local_moment_scale} and
$|\cL_m(\theta)|^3\lesssim|S_m(\theta)|^3+|\Delta_m(\theta)|^3$.
To control the product of means, apply~\eqref{eq:regular_cm_first} to $A$.
Since $\E_0A=0$, Young's inequality with conjugate exponents $3/2$ and $3$
gives
\begin{align*}
|\E_\theta A\,\E_\theta C|
&\lesssim
\sum_{\bullet\in\{0,\theta\}}
\left\{(\E_\bullet|A\cL_m(\theta)|)^{3/2}
+(\E_\theta|C|)^3\right\}
\\
&\lesssim
\sum_{\bullet\in\{0,\theta\}}
\bigl(\E_\bullet|A|^3+\E_\bullet|\cL_m(\theta)|^3
+\E_\theta|C|^3\bigr).
\end{align*}
Here the second inequality follows from Jensen's inequality and the pointwise
bound $|A\cL_m(\theta)|^{3/2}\le
\tfrac12|A|^3+\tfrac12|\cL_m(\theta)|^3$. Hence,
\[
|\E_\theta A\,\E_\theta C|
\lesssim
\sum_{\bullet\in\{0,\theta\}}
\bigl(\E_\bullet|A|^3+\E_\bullet|C|^3
+\E_\bullet|\cL_m(\theta)|^3\bigr)
\lesssim r_n^2\delta_n(\theta).
\]
We have therefore proved
\begin{equation}
\left|\cov_\theta(A,C)-\E_0[AC]\right|
\lesssim r_n^2\delta_n(\theta)
\label{eq:regular_directional_cov_generic}
\end{equation}
for all such $A$ and $C$.

Taking $A=U_m^B(s)$ and $C=U_m^B(t)$ in
\eqref{eq:regular_directional_cov_generic} gives
\[
\left|
\cov_\theta(U_m^B(s),U_m^B(t))
-\E_0[U_m^B(s)U_m^B(t)]
\right|
\lesssim r_n^2\delta_n(\theta).
\]
Equations~\eqref{eq:regular_nongaussian_cov_BB}
and~\eqref{eq:regular_gaussian_cov_BB} show that
\begin{align*}
\cov_\theta(\zeta_m(s),\zeta_m(t))
&=\cov_\theta(U_m^B(s),U_m^B(t)),
\\
\cov_{Q_\theta}(\zeta_m^{\mathrm G}(s),\zeta_m^{\mathrm G}(t))
&=\E_0[U_m^B(s)U_m^B(t)].
\end{align*}
Thus this application of~\eqref{eq:regular_directional_cov_generic} is
exactly~\eqref{eq:regular_directional_cov_L}.
For~\eqref{eq:regular_directional_cov_X}, write
$X_m=U_m+\widetilde R_m$, where
$\widetilde R_m=-\tfrac12V_m+R_m$. Since centering the second argument does
not change a covariance,
\begin{align*}
&\cov_\theta(X_m(s),\zeta_m(t))
-\cov_{Q_\theta}(Y_m(s),\zeta_m^{\mathrm G}(t))
\\
&\quad=
\cov_\theta(U_m(s),U_m^B(t))
-\E_0[U_m(s)U_m^B(t)]
+\cov_\theta(\widetilde R_m(s),U_m^B(t)).
\end{align*}
The first difference is bounded by
\eqref{eq:regular_directional_cov_generic}. For the last covariance,
\eqref{eq:regular_cm_first} gives
\begin{align*}
|\cov_\theta(\widetilde R_m(s),U_m^B(t))|
&\le
|\E_0[\widetilde R_m(s)U_m^B(t)]|
\\
&\quad+
\sum_{\bullet\in\{0,\theta\}}
\E_\bullet|\widetilde R_m(s)U_m^B(t)\cL_m(\theta)|
\\
&\quad+
|\E_\theta\widetilde R_m(s)\,\E_\theta U_m^B(t)|.
\end{align*}
By Cauchy--Schwarz and Lemma~\ref{lem:regular_local_moment_scale},
\[
|\E_0[\widetilde R_m(s)U_m^B(t)]|
\le
\sqrt{\E_0|\widetilde R_m(s)|^2}
\sqrt{\E_0|U_m^B(t)|^2}
\lesssim r_n^2\delta_n(\theta).
\]
The two triple products are bounded by H\"older's inequality and the
corresponding third moments. Finally, because $\E_0U_m^B(t)=0$,
\eqref{eq:regular_cm_first} and the Jensen--Young calculation used to prove
\eqref{eq:regular_directional_cov_generic} give
\begin{align*}
|\E_\theta\widetilde R_m(s)\,\E_\theta U_m^B(t)|
&\lesssim
\sum_{\bullet\in\{0,\theta\}}
\left\{(\E_\bullet|U_m^B(t)\cL_m(\theta)|)^{3/2}
+(\E_\theta|\widetilde R_m(s)|)^3\right\}
\\
&\lesssim
\sum_{\bullet\in\{0,\theta\}}
\bigl(\E_\bullet|U_m^B(t)|^3
+\E_\bullet|\cL_m(\theta)|^3
+\E_\theta|\widetilde R_m(s)|^3\bigr)
\\
&\lesssim r_n^2\delta_n(\theta).
\end{align*}
Thus every term is $\lesssim r_n^2\delta_n(\theta)$, which proves
\eqref{eq:regular_directional_cov_X}.

We also record the third-moment bound for the perturbation. Since $q_m$ is
deterministic and $\E_\theta\zeta_m=0$,
\[
\widetilde X_m-\E_\theta\widetilde X_m=\zeta_m,
\qquad
\|\zeta_m\|_\infty
\lesssim r_n\|U_m-\E_\theta U_m\|_{\cH_m}.
\]
Consequently, by~\eqref{eq:regular_moment_U} and Jensen's inequality,
\begin{equation}
\E_\theta
\|\widetilde X_m-\E_\theta\widetilde X_m\|_\infty^3
\lesssim r_n^2\delta_n(\theta).
\label{eq:regular_perturbation_third_moment}
\end{equation}

\emph{Perturbation approximation error $\Delta_1$.}
Lemma~\ref{lem:linear_perturb_G} gives
\[
\psi_{\theta,\beta}^{\mathrm G}(s,t)-\phi_\theta(s,t)
=
\beta\langle B^*\Sigma B[\eta(t)-\eta(s)],\eta(t)\rangle.
\]
By the product structure in~\eqref{eq:product_feature_map} and
\eqref{eq:regular_posterior_product_operators},
\begin{align*}
&\left|
\langle B^*\Sigma B[\eta(t)-\eta(s)],\eta(t)\rangle_{\cH}
\right|
\\
&\quad\le
\sum_{m=1}^{M_n}
\|B_m\|_{\mathrm{op}}^2
\|\Sigma_m\|_{\mathrm{op}}
\|\eta_m(t)-\eta_m(s)\|_{\cH_m}
\|\eta_m(t)\|_{\cH_m}.
\end{align*}
The definition of $\Sigma_m$ and $\mathsf L_{m,n}\lesssim1$ give
$\|\Sigma_m\|_{\mathrm{op}}\lesssim1$, while $\mathfrak h_n=O(1)$ gives
$\|B_m\|_{\mathrm{op}}\lesssim1$ uniformly in $m$. Finally,
\eqref{eq:feature_radius} implies
$\|\eta_m(t)-\eta_m(s)\|_{\cH_m}\le2r_n$ and
$\|\eta_m(t)\|_{\cH_m}\le r_n$. Consequently,
\begin{equation}
\|\psi_{\theta,\beta}^{\mathrm G}-\phi_\theta\|_\infty
\lesssim |\beta|M_nr_n^2.
\label{eq:regular_Delta1_G}
\end{equation}
For the non-Gaussian perturbation, $\E_\theta\widetilde X_m(t)=q_m(t)$, and
\eqref{eq:psi_X} yields
\begin{align}
\psi_{\theta,0}(s,t)-\phi_\theta(s,t)
&=
\sum_{m=1}^{M_n}
\Bigl[
\cov_\theta(X_m(t),\zeta_m(t))
-\langle H_m\eta_m(t),\eta_m(t)\rangle
\Bigr]
\\
&\quad-
\sum_{m=1}^{M_n}
\Bigl[
\cov_\theta(X_m(s),\zeta_m(t))
-\langle H_m\eta_m(s),\eta_m(t)\rangle
\Bigr].
\label{eq:regular_psi0_decomposition}
\end{align}
Equation~\eqref{eq:regular_gaussian_cov_XB} identifies each inner product in
\eqref{eq:regular_psi0_decomposition} with the corresponding Gaussian
covariance. Applying~\eqref{eq:regular_directional_cov_X} to each bracket
and summing over $m$ therefore gives
\begin{equation}
\|\psi_{\theta,0}-\phi_\theta\|_\infty
\lesssim M_nr_n^2\delta_n(\theta).
\label{eq:regular_Delta1_X}
\end{equation}
Finally, Lemma~\ref{lem:regular_third_moment},
\eqref{eq:regular_perturbation_third_moment}, and H\"older's inequality give
\begin{align}
&\sum_{m=1}^{M_n}
\Bigl(\E_\theta\|X_m-\E_\theta X_m\|_\infty^3\Bigr)^{2/3}
\Bigl(\E_\theta\|\dot X_{0,m}-\E_\theta\dot X_{0,m}\|_\infty^3\Bigr)^{1/3}
\lesssim
M_nr_n^2\delta_n(\theta).
\label{eq:regular_Delta1_mixed_moment}
\end{align}
Combining~\eqref{eq:regular_Delta1_G},
\eqref{eq:regular_Delta1_X}, and
\eqref{eq:regular_Delta1_mixed_moment}, and using $M_nr_n^2\lesssim n$,
proves the first bound in~\eqref{eq:regular_posterior_Delta_targets}.

\emph{Free energy comparison error $\Delta_2$.}
Summing~\eqref{eq:regular_mean_comparison_bound},
\eqref{eq:regular_cov_comparison_bound}, and
\eqref{eq:regular_third_moment_bound} over $m$, and using
$M_nr_n^2\lesssim n$, bounds the unperturbed mean, covariance, and
third-moment terms by $n\widetilde\delta_n(\theta)$. For the perturbed mean, the
centered linear terms satisfy
$\E_\theta\zeta_m=\E_{Q_\theta}\zeta_m^{\mathrm G}=0$, while the deterministic
terms $q_m$ are
identical. Hence,
\begin{equation}
\E_\theta[X_\beta]-\E_{Q_\theta}[Y_\beta]
=
\E_\theta[X]-\E_{Q_\theta}[Y].
\label{eq:regular_perturbed_mean_identity}
\end{equation}
Equations~\eqref{eq:regular_perturbed_mean_identity} and
\eqref{eq:regular_mean_comparison_bound} therefore give
\begin{equation}
\left\|\E_\theta[X_\beta]-\E_{Q_\theta}[Y_\beta]\right\|_\infty
\lesssim n
\widetilde\delta_n(\theta)
.
\label{eq:regular_perturbed_mean_bound}
\end{equation}
For the covariance, the deterministic terms again make no contribution, and
bilinearity gives
\begin{align}
&\cov_\theta(X_\beta)-\cov_{Q_\theta}(Y_\beta)
\\
&\quad=
\cov_\theta(X)-\cov_{Q_\theta}(Y)
+\beta\mathcal C_1+\beta^2\mathcal C_2,
\label{eq:regular_perturbed_cov_decomposition}
\end{align}
where
\begin{align*}
\mathcal C_1(s,t)
&\coloneqq
\sum_{m=1}^{M_n}
\Bigl\{
\cov_\theta(X_m(s),\zeta_m(t))
+\cov_\theta(\zeta_m(s),X_m(t))
\\
&\hspace{35mm}
-\cov_{Q_\theta}(Y_m(s),\zeta_m^{\mathrm G}(t))
-\cov_{Q_\theta}(\zeta_m^{\mathrm G}(s),Y_m(t))
\Bigr\},
\\
\mathcal C_2(s,t)
&\coloneqq
\sum_{m=1}^{M_n}
\Bigl\{
\cov_\theta(\zeta_m(s),\zeta_m(t))
-\cov_{Q_\theta}(\zeta_m^{\mathrm G}(s),\zeta_m^{\mathrm G}(t))
\Bigr\}.
\end{align*}
Equations~\eqref{eq:regular_directional_cov_X} and
\eqref{eq:regular_directional_cov_L} imply
\begin{equation}
\|\mathcal C_1\|_\infty+\|\mathcal C_2\|_\infty
\lesssim M_nr_n^2\delta_n(\theta).
\label{eq:regular_perturbed_cov_cross_bound}
\end{equation}
Combining~\eqref{eq:regular_perturbed_cov_decomposition},
\eqref{eq:regular_cov_comparison_bound}, and
\eqref{eq:regular_perturbed_cov_cross_bound}, and using
$M_nr_n^2\lesssim n$, gives, for $|\beta|$ sufficiently small,
\begin{equation}
\left\|\cov_\theta(X_\beta)-\cov_{Q_\theta}(Y_\beta)\right\|_\infty
\lesssim n
\widetilde\delta_n(\theta)
.
\label{eq:regular_perturbed_cov_bound}
\end{equation}
Finally,
\[
X_{\beta,m}-\E_\theta X_{\beta,m}
=
(X_m-\E_\theta X_m)+\beta\zeta_m.
\]
Lemma~\ref{lem:regular_third_moment} and
\eqref{eq:regular_perturbation_third_moment} therefore give
\begin{equation}
\sum_{m=1}^{M_n}
\E_\theta\|X_{\beta,m}-\E_\theta X_{\beta,m}\|_\infty^3
\lesssim M_nr_n^2\delta_n(\theta)
\lesssim n\delta_n(\theta).
\label{eq:regular_perturbed_third_moment_bound}
\end{equation}
Equations~\eqref{eq:regular_perturbed_mean_bound},
\eqref{eq:regular_perturbed_cov_bound}, and
\eqref{eq:regular_perturbed_third_moment_bound} prove the second bound in
\eqref{eq:regular_posterior_Delta_targets}.

\emph{Conclusion.}
Applying Theorem~\ref{thm:Phi_comparison} with $\beta=\pm\eps$ and dividing
by $n$ now gives
\[
\Phi_n^{\mathrm G}(\theta,-\eps)
-C\left(\eps+\delta_n(\theta)
+\frac{\widetilde\delta_n(\theta)}{\eps}\right)
\le
\Phi_n(\theta)
\le
\Phi_n^{\mathrm G}(\theta,\eps)
+C\left(\eps+\delta_n(\theta)
+\frac{\widetilde\delta_n(\theta)}{\eps}\right).
\]
Taking $\eps>0$ sufficiently small, we may assume that $\eps<1$. Since
$\delta_n(\theta)\le\widetilde\delta_n(\theta)$, the term
$\delta_n(\theta)$ is absorbed into
$\widetilde\delta_n(\theta)/\eps$, which yields the asserted bound.

\subsection{Proof of Corollary~\ref{cor:regular_posterior_limit}}
\label{app:proof_regular_posterior_limit}

\emph{Averaging the comparison bound.}
Set
\[
\overline\Phi_n
\coloneqq
\int_{\Theta_n}\Phi_n(\theta)\,\nu_n(\dd\theta),
\qquad
\overline{\widetilde\delta}_n
\coloneqq
\int_{\Theta_n}\widetilde\delta_n(\theta)\,\nu_n(\dd\theta).
\]
Averaging the sandwich bound in
Theorem~\ref{thm:regular_posterior_comparison} gives, for every fixed
sufficiently small $\eps>0$,
\begin{align}
\int_{\Theta_n}\Phi_n^{\mathrm G}(\theta,-\eps)\,\nu_n(\dd\theta)
-C\left(\eps+\frac{\overline{\widetilde\delta}_n}{\eps}\right)
&\le \overline\Phi_n,
\label{eq:proof_regular_limit_lower}
\\
\overline\Phi_n
&\le
\int_{\Theta_n}\Phi_n^{\mathrm G}(\theta,\eps)\,\nu_n(\dd\theta)
+C\left(\eps+\frac{\overline{\widetilde\delta}_n}{\eps}\right).
\label{eq:proof_regular_limit_upper}
\end{align}

\emph{Gaussian observables and free energy derivatives.}
By Lemma~\ref{lem:linear_perturb_G} and
\eqref{eq:proof_PhiG_derivative_error},
\eqref{eq:regular_Delta1_G} gives, uniformly in $\theta$ and for $\beta$ in
a sufficiently small neighborhood of zero,
\begin{equation}
\left|
\partial_\beta F_n^{\mathrm G}(\theta,\beta)
-\Phi_n^{\mathrm G}(\theta,\beta)
\right|
\le C|\beta|.
\label{eq:regular_limit_pointwise_derivative_error}
\end{equation}
The product structure in \eqref{eq:product_feature_map},
\eqref{eq:regular_posterior_product_operators} and
\eqref{eq:feature_radius}, bounds
$n^{-1}\|\phi_\theta\|_\infty$ and
$n^{-1}\|\psi_{\theta,\beta}^{\mathrm G}\|_\infty$ uniformly in $n$ and
$\theta$ on every fixed compact $\beta$-interval contained in this
neighborhood. Hence, by
\eqref{eq:gaussian_free_energy_derivative},
\[
\left|\partial_\beta F_n^{\mathrm G}(\theta,\beta)\right|
\le
n^{-1}\|\psi_{\theta,\beta}^{\mathrm G}\|_\infty
\le C.
\]
Because $\nu_n$ is a probability measure, this uniform bound provides an
integrable dominating function for the $\beta$-derivatives. Differentiation
may therefore be interchanged with the $\nu_n$-integral. Integrating
\eqref{eq:regular_limit_pointwise_derivative_error} yields
\begin{equation}
\left|
(\mathcal F_n^{\mathrm G})'(\beta)
-
\int_{\Theta_n}\Phi_n^{\mathrm G}(\theta,\beta)\,
\nu_n(\dd\theta)
\right|
\le C|\beta|,
\label{eq:proof_regular_limit_derivative}
\end{equation}
with equality at $\beta=0$.

\emph{Passage to the convex limit.}
For every $n$, $\mathcal F_n^{\mathrm G}$ is convex because it is the
$\nu_n$-average of the convex functions
$F_n^{\mathrm G}(\theta,\cdot)$. Its finite pointwise limit
$\mathcal F^{\mathrm G}$ is therefore convex. We next derive a consequence
of this pointwise convergence that will be used at $\beta=\pm\eps_k$ and, in
the differentiable case, at $\beta=0$.
Fix any $\beta$ in the open neighborhood on which
$\mathcal F_n^{\mathrm G}\to\mathcal F^{\mathrm G}$ pointwise, and suppose
that $\mathcal F^{\mathrm G}$ is differentiable at this $\beta$. For all
sufficiently small $h>0$, both $\beta-h$ and $\beta+h$ remain in that
neighborhood, and convexity gives
\begin{equation}
\frac{\mathcal F_n^{\mathrm G}(\beta)-
\mathcal F_n^{\mathrm G}(\beta-h)}{h}
\le (\mathcal F_n^{\mathrm G})'(\beta)
\le
\frac{\mathcal F_n^{\mathrm G}(\beta+h)-
\mathcal F_n^{\mathrm G}(\beta)}{h}.
\label{eq:regular_limit_secant_bounds}
\end{equation}
Letting first $n\to\infty$ in
\eqref{eq:regular_limit_secant_bounds} and then $h\downarrow0$ shows that
\begin{equation}
(\mathcal F_n^{\mathrm G})'(\beta)
\longrightarrow
(\mathcal F^{\mathrm G})'(\beta).
\label{eq:regular_limit_derivative_convergence}
\end{equation}

A finite convex function is differentiable except at most countably many
points. We may consequently choose $\eps_k\downarrow0$ such that
$\mathcal F^{\mathrm G}$ is differentiable at both $-\eps_k$ and $\eps_k$.
We first apply \eqref{eq:regular_limit_derivative_convergence} at these two
points. For fixed $k$, equations~\eqref{eq:proof_regular_limit_derivative}
and \eqref{eq:regular_limit_derivative_convergence}, with
$\beta=-\eps_k$ and $\beta=\eps_k$, imply
\begin{align*}
(\mathcal F^{\mathrm G})'(-\eps_k)-C\eps_k
&\le
\liminf_{n\to\infty}
\int_{\Theta_n}\Phi_n^{\mathrm G}(\theta,-\eps_k)\,
\nu_n(\dd\theta),
\\
\limsup_{n\to\infty}
\int_{\Theta_n}\Phi_n^{\mathrm G}(\theta,\eps_k)\,
\nu_n(\dd\theta)
&\le
(\mathcal F^{\mathrm G})'(\eps_k)+C\eps_k.
\end{align*}
Since $\overline{\widetilde\delta}_n\to0$, applying these two bounds in
\eqref{eq:proof_regular_limit_lower} and
\eqref{eq:proof_regular_limit_upper}, respectively, gives
\[
(\mathcal F^{\mathrm G})'(-\eps_k)-C\eps_k
\le
\liminf_{n\to\infty}\overline\Phi_n
\le
\limsup_{n\to\infty}\overline\Phi_n
\le
(\mathcal F^{\mathrm G})'(\eps_k)+C\eps_k.
\]
For a convex function, derivatives at differentiability points converging to
zero from the left and right converge to the corresponding one-sided
derivatives at zero. Letting $k\to\infty$ therefore proves
\[
\partial_-\mathcal F^{\mathrm G}(0)
\le
\liminf_{n\to\infty}\overline\Phi_n
\le
\limsup_{n\to\infty}\overline\Phi_n
\le
\partial_+\mathcal F^{\mathrm G}(0).
\]

\emph{Differentiability at the origin.}
If $\mathcal F^{\mathrm G}$ is differentiable at zero, these one-sided
derivatives agree, and hence $\overline\Phi_n\to
(\mathcal F^{\mathrm G})'(0)$. We now use
\eqref{eq:regular_limit_derivative_convergence} a second time, with
$\beta=0$. Together with \eqref{eq:proof_regular_limit_derivative} at
$\beta=0$, it gives
\[
\int_{\Theta_n}\Phi_n^{\mathrm G}(\theta)\,\nu_n(\dd\theta)
=
(\mathcal F_n^{\mathrm G})'(0)
\longrightarrow
(\mathcal F^{\mathrm G})'(0),
\]
which proves the final assertion.

\section{Proofs for Section~\ref{sec:app_sparse_bernoulli}}
\label{app:proofs_section5}

We prove the results stated in Section~\ref{sec:app_sparse_bernoulli}. The
proof of Theorem~\ref{thm:sparse_bernoulli_univ} proceeds by expanding the
local Bernoulli likelihood, identifying the moment-matched Gaussian model, and
verifying the local moment bounds required by
Theorem~\ref{thm:regular_free_energy_univ}. For
Theorem~\ref{thm:sparse_bernoulli_posterior_univ}, we verify the operator
conditions for the prescribed perturbation directions and apply
Theorem~\ref{thm:regular_posterior_comparison}. Finally,
Corollary~\ref{cor:sparse_posterior_limits} follows from the corresponding
limiting Gaussian free energies.

\subsection{Proof of Theorem~\ref{thm:sparse_bernoulli_univ}}
\label{app:sparse_bernoulli_details}

The Gaussian comparison process has already been specified explicitly in
\eqref{eq:bernoulli_gaussian_process}. To apply
Theorem~\ref{thm:regular_free_energy_univ}, it remains to show that this process
is precisely the Gaussian counterpart assigned to the Bernoulli likelihood by
the matching rules in~\eqref{eq:regular_operator_forms}--
\eqref{eq:regular_operator_K}, and to verify the required local moment bounds.
Throughout the proof, we suppress the superscript $(n)$ on the local
quantities. We first expand the Bernoulli likelihood in the regular local form and identify
the coefficients $U_\alpha,V_\alpha,R_\alpha,S_\alpha$, and $\Delta_\alpha$.
We then compute the null-law moments defining $A_\alpha,\Sigma_\alpha$, and
$K_\alpha$ and check that the resulting product Gaussian construction recovers
\eqref{eq:bernoulli_gaussian_process}. Finally, we use the two-point structure
of $G_\alpha$ to verify
Assumption~\ref{assump:regular_local_bounds}. For this purpose, write
\[
q_n\coloneqq\frac{\de_n}{s_n},
\qquad
\kappa_n\coloneqq q_n(1-q_n),
\qquad
c_{\alpha,n}\coloneqq
b_nG_\alpha-b_n^{-1}(1-G_\alpha).
\]
Thus $b_n=((1-q_n)/q_n)^{1/2}$ and, under $P_{0,\alpha}$,
$G_\alpha\sim\operatorname{Bernoulli}(q_n)$.

\emph{Local expansion and Gaussian matching.}
Let $z_\alpha(t)=\langle J_\alpha,\eta_\alpha(t)\rangle$. Taylor expansion of
the two logarithms in~\eqref{eq:bernoulli_local_likelihood} gives
\begin{equation}
X_\alpha(t,G_\alpha)
=
\langle U_\alpha(G_\alpha),\eta_\alpha(t)\rangle
-\frac12
\langle V_\alpha(G_\alpha)\eta_\alpha(t),\eta_\alpha(t)\rangle
+R_\alpha(t,G_\alpha),
\label{eq:sparse_local_expansion}
\end{equation}
where $U_\alpha,V_\alpha$ are given in~\eqref{eq:bernoulli_UV}. In particular,
\begin{equation}
U_\alpha=\sqrt\lambda\,c_{\alpha,n}J_\alpha,
\qquad
V_\alpha=\lambda c_{\alpha,n}^2 J_\alpha J_\alpha^\top.
\label{eq:sparse_UV_coefficients}
\end{equation}
For some $v_\alpha(t)$ between $0$ and $z_\alpha(t)$, the remainder is
\begin{equation}
R_\alpha(t,G_\alpha)
=
\frac{\lambda^{3/2}}3
\left\{
\frac{b_n^3G_\alpha}
{(1+b_n\sqrt\lambda\,v_\alpha(t))^3}
-
\frac{b_n^{-3}(1-G_\alpha)}
{(1-\sqrt\lambda\,v_\alpha(t)/b_n)^3}
\right\}
z_\alpha(t)^3.
\label{eq:sparse_bernoulli_remainder}
\end{equation}

We next place the data-generating law in the local form
\eqref{eq:regular_likelihood_ratio}. By
\eqref{eq:bernoulli_hypergraph_data}, its local likelihood ratio relative to
$P_{0,\alpha}$ is
\begin{align*}
\log\frac{\dd P_{\theta,\alpha}}{\dd P_{0,\alpha}}(G_\alpha)
&=
G_\alpha\log\bigl(1+b_n\sqrt{\lambda_\star}\,
z_\alpha(\theta)\bigr)
+
(1-G_\alpha)\log\left(1-
\frac{\sqrt{\lambda_\star}\,z_\alpha(\theta)}{b_n}\right).
\end{align*}
This is the same scalar expression as the local statistical likelihood, with
$t$ replaced by the true parameter $\theta$ and $\lambda$ replaced by
$\lambda_\star$. Applying the expansion
\eqref{eq:sparse_local_expansion} at $\eta_\alpha(\theta)$ therefore gives
\begin{equation}
\log\frac{\dd P_{\theta,\alpha}}{\dd P_{0,\alpha}}(G_\alpha)
=
\langle S_\alpha(G_\alpha),\eta_\alpha(\theta)\rangle
+\Delta_\alpha(\theta,G_\alpha),
\label{eq:sparse_data_expansion}
\end{equation}
The linear coefficient and the remaining quadratic and third-order terms are,
respectively,
\begin{align}
S_\alpha
&=\sqrt{\lambda_\star}\,c_{\alpha,n}J_\alpha,
\label{eq:sparse_S_def}
\\
\Delta_\alpha(\theta)
&=-\frac12
\langle V_\alpha^\star\eta_\alpha(\theta),
\eta_\alpha(\theta)\rangle
+R_\alpha^\star(\theta).
\label{eq:sparse_Delta_def}
\end{align}
Here $V_\alpha^\star$ and $R_\alpha^\star$ are the quantities in
\eqref{eq:sparse_UV_coefficients} and
\eqref{eq:sparse_bernoulli_remainder}, respectively, with $\lambda$ replaced
by $\lambda_\star$.

Under the reference law $P_{0,\alpha}$,
$G_\alpha\sim\operatorname{Bernoulli}(q_n)$. Since
$b_n^2=(1-q_n)/q_n$, direct calculation gives
\begin{align}
\E_0[c_{\alpha,n}]
&=q_nb_n-(1-q_n)b_n^{-1}
=\sqrt{q_n(1-q_n)}-\sqrt{q_n(1-q_n)}=0,
\label{eq:sparse_c_null_mean}
\\
\E_0[c_{\alpha,n}^2]
&=q_nb_n^2+(1-q_n)b_n^{-2}
=(1-q_n)+q_n=1.
\label{eq:sparse_c_null_second_moment}
\end{align}
Combining~\eqref{eq:sparse_c_null_mean}--
\eqref{eq:sparse_c_null_second_moment} with
\eqref{eq:sparse_UV_coefficients} and~\eqref{eq:sparse_S_def} yields
\begin{equation}
\begin{aligned}
\E_0[U_\alpha]&=0,
&
\E_0[U_\alpha U_\alpha^\top]
&=\lambda J_\alpha J_\alpha^\top,
\\
\E_0[V_\alpha]
&=\lambda J_\alpha J_\alpha^\top,
&
\E_0[S_\alpha U_\alpha^\top]
&=\sqrt{\lambda\lambda_\star}\,
J_\alpha J_\alpha^\top.
\end{aligned}
\label{eq:sparse_local_coefficient_moments}
\end{equation}

The moment identities in~\eqref{eq:sparse_local_coefficient_moments}
determine the three local operators in the regular framework. Indeed, for
arbitrary $h,g\in\bbR^{d^p}$, the definitions in
\eqref{eq:regular_operator_forms}--\eqref{eq:regular_operator_K} give
\begin{equation*}
\begin{aligned}
\langle A_\alpha h,g\rangle
&=\E_0[\langle S_\alpha,h\rangle\langle U_\alpha,g\rangle]
=\sqrt{\lambda\lambda_\star}
\langle J_\alpha,h\rangle\langle J_\alpha,g\rangle,
\\
\langle\Sigma_\alpha h,g\rangle
&=\E_0[\langle U_\alpha,h\rangle\langle U_\alpha,g\rangle]
=\lambda
\langle J_\alpha,h\rangle\langle J_\alpha,g\rangle,
\\
\langle K_\alpha h,g\rangle
&=\E_0[\langle V_\alpha h,g\rangle]
=\lambda
\langle J_\alpha,h\rangle\langle J_\alpha,g\rangle.
\end{aligned}
\end{equation*}
Since these identities hold for every $h$ and $g$, they identify the
operators as
\[
A_\alpha=\sqrt{\lambda\lambda_\star}J_\alpha J_\alpha^\top,
\qquad
\Sigma_\alpha=K_\alpha=\lambda J_\alpha J_\alpha^\top,
\]
in agreement with~\eqref{eq:bernoulli_gaussian_operators}. Inserting these
operators into the product Gaussian construction
\eqref{eq:gaussian_product_operators}--\eqref{eq:Y_local_moments} produces
exactly the explicit comparison process
\eqref{eq:bernoulli_gaussian_process}. This completes the Gaussian matching.

\emph{Reduction to a scalar moment bound.}
It remains to verify the local moment conditions in
Assumption~\ref{assump:regular_local_bounds}. The coefficients in
\eqref{eq:sparse_UV_coefficients}--\eqref{eq:sparse_Delta_def} are either
powers of $c_{\alpha,n}$ or, in the case of the Taylor remainders, are bounded
by such powers. It therefore suffices to establish a single moment estimate
for this scalar Bernoulli variable. Introduce the auxiliary quantity
\begin{equation}
x_n\coloneqq b_n+b_n^{-1}=\kappa_n^{-1/2}.
\label{eq:sparse_xn_def}
\end{equation}
For every integer $k\ge2$, the two possible values of $c_{\alpha,n}$ give
\begin{equation}
\E_0[|c_{\alpha,n}|^k]
=q_nb_n^k+(1-q_n)b_n^{-k}
=\sqrt{\kappa_n}
\bigl(b_n^{k-1}+b_n^{-(k-1)}\bigr)
\le x_n^{k-2}.
\label{eq:sparse_binary_moment_null}
\end{equation}
Under $P_{\theta,\alpha}$, the success probability is
\[
q_{\alpha,\theta}
=q_n+\sqrt{\kappa_n}\sqrt{\lambda_\star}\,
z_\alpha(\theta).
\]
It follows from the same two-point calculation that
\begin{align}
\E_\theta[|c_{\alpha,n}|^k]
&\le
\E_0[|c_{\alpha,n}|^k]
+\sqrt{\kappa_n}\sqrt{\lambda_\star}
|z_\alpha(\theta)|(b_n^k+b_n^{-k})
\lesssim
x_n^{k-2}\bigl(1+r_nx_n\bigr).
\label{eq:sparse_binary_moment_theta}
\end{align}
Indeed, $|z_\alpha(\theta)|\lesssim r_n$ because the couplings are uniformly
bounded. Moreover, since $p$ and $d$ are fixed and the coordinates of
$t\in T_n$ lie in $[-1,1]^d$, the tensor feature map satisfies
\begin{equation}
r_n\lesssim s_n^{-1/2},
\qquad
x_n
=\sqrt{\frac{s_n-\de_n}{\de_n}}
+\sqrt{\frac{\de_n}{s_n-\de_n}}
\le
\sqrt{s_n}\,\delta_n,
\label{eq:sparse_feature_binary_scale}
\end{equation}
where
$\delta_n=\de_n^{-1/2}+(s_n-\de_n)^{-1/2}$. Thus
$r_nx_n\lesssim\delta_n$. Multiplying
\eqref{eq:sparse_binary_moment_null} and
\eqref{eq:sparse_binary_moment_theta} by $r_n^{k-2}$ therefore eliminates the
auxiliary quantity $x_n$ and gives, uniformly in $\alpha$ and $\theta$,
\begin{equation}
 r_n^{k-2}\E_\bullet[|c_{\alpha,n}|^k]
 \lesssim \delta_n^{k-2},
\qquad \bullet\in\{0,\theta\},\quad k\ge2.
\label{eq:sparse_binary_moment}
\end{equation}
The assumptions $\de_n\to\infty$ and $s_n-\de_n\to\infty$ imply
$\delta_n\to0$. Hence $1+r_nx_n\lesssim1$ and, for every fixed $k\ge3$,
$\delta_n^{k-2}\lesssim\delta_n$. Formula
\eqref{eq:sparse_binary_moment} is the form used in all subsequent moment
estimates.

\emph{Verification of the local moment bounds.}
Taking the trace of the covariance identity for $U_\alpha$ in
\eqref{eq:sparse_local_coefficient_moments} gives
$\E_0\|U_\alpha\|^2=\lambda\|J_\alpha\|^2$. Moreover,
\eqref{eq:sparse_S_def} and \eqref{eq:sparse_c_null_second_moment} give
$\E_0\|S_\alpha\|^2=\lambda_\star\|J_\alpha\|^2$. Hence the uniform bound on
the couplings yields
\[
\E_0\|U_\alpha\|^2+\E_0\|S_\alpha\|^2
=(\lambda+\lambda_\star)\|J_\alpha\|^2\lesssim1.
\]
The explicit formulas for $U_\alpha$ and $S_\alpha$ in
\eqref{eq:sparse_UV_coefficients} and \eqref{eq:sparse_S_def}, together with
\eqref{eq:sparse_binary_moment} at $k=3$, similarly yield
\begin{equation}
r_n\bigl(
\E_0\|U_\alpha\|^3+
\E_\theta\|U_\alpha\|^3+
\E_0\|S_\alpha\|^3+
\E_\theta\|S_\alpha\|^3
\bigr)
\lesssim\delta_n.
\label{eq:sparse_US_moments}
\end{equation}
By \eqref{eq:sparse_UV_coefficients},
$\|V_\alpha\|_{\mathrm{op}}\lesssim|c_{\alpha,n}|^2$; hence
\eqref{eq:sparse_binary_moment} with $k=4$ and $k=6$ gives
\begin{equation}
r_n\sqrt{\E_0\|V_\alpha\|_{\mathrm{op}}^2}
+r_n^4\sum_{\bullet\in\{0,\theta\}}
\E_\bullet\|V_\alpha\|_{\mathrm{op}}^3
\lesssim
\delta_n+\delta_n^4
\lesssim\delta_n.
\label{eq:sparse_V_moments}
\end{equation}

The uniform well-definedness of the Bernoulli channel ensures that the
denominators in \eqref{eq:sparse_bernoulli_remainder} are bounded away from
zero uniformly over $\alpha$, $t$, and $n$. Since
$|z_\alpha(t)|\lesssim r_n$, that formula implies
\begin{equation}
\sup_{t\in T_n}|R_\alpha(t,G_\alpha)|
\lesssim r_n^3|c_{\alpha,n}|^3.
\label{eq:sparse_R_pointwise}
\end{equation}
Combining \eqref{eq:sparse_R_pointwise} with
\eqref{eq:sparse_binary_moment} at $k=3,6,9$ gives
\begin{align}
&\sup_{t\in T_n}
\left\{
r_n^{-2}|\E_0R_\alpha(t)|
+r_n^{-1}\sqrt{\E_0|R_\alpha(t)|^2}
+r_n^{-2}\E_0|R_\alpha(t)|^3
\right\}
\nonumber\\
&\qquad
+r_n^{-2}\E_\theta
\left[\sup_{t\in T_n}|R_\alpha(t)|^3\right]
\lesssim
\delta_n+\delta_n^2+\delta_n^7
\lesssim\delta_n.
\label{eq:sparse_R_moments}
\end{align}

The starred coefficients, obtained by replacing $\lambda$ with
$\lambda_\star$, satisfy the same pointwise bounds. Consequently,
\eqref{eq:sparse_Delta_def}, the formula for $V_\alpha$ in
\eqref{eq:sparse_UV_coefficients}, and the starred analogue of
\eqref{eq:sparse_R_pointwise} imply
\begin{equation}
|\Delta_\alpha(\theta)|
\lesssim
r_n^2|c_{\alpha,n}|^2+r_n^3|c_{\alpha,n}|^3.
\label{eq:sparse_Delta_pointwise}
\end{equation}
Combining \eqref{eq:sparse_Delta_pointwise} with
\eqref{eq:sparse_binary_moment} at $k=4,6,9$ yields
\begin{equation}
r_n^{-1}\sqrt{\E_0|\Delta_\alpha(\theta)|^2}
+r_n^{-2}\sum_{\bullet\in\{0,\theta\}}
\E_\bullet|\Delta_\alpha(\theta)|^3
\lesssim\delta_n.
\label{eq:sparse_Delta_moments}
\end{equation}
The second-moment identity \eqref{eq:sparse_local_coefficient_moments} and the
bounds \eqref{eq:sparse_US_moments}, \eqref{eq:sparse_V_moments},
\eqref{eq:sparse_R_moments}, and \eqref{eq:sparse_Delta_moments} verify all
conditions in Assumption~\ref{assump:regular_local_bounds}. Moreover,
$r_n\lesssim s_n^{-1/2}=\sqrt{n/M_n}$. Therefore
Theorem~\ref{thm:regular_free_energy_univ} applies and first gives an error of
order $\widetilde\delta_n=\delta_n(1+\delta_n)$. Since $\delta_n\to0$, we have
$\widetilde\delta_n\lesssim\delta_n$ for all sufficiently large $n$, which
gives the asserted free energy comparison.

\subsection{Proof of Theorem~\ref{thm:sparse_bernoulli_posterior_univ}}
\label{app:proof_sparse_bernoulli_posterior_univ}

\begin{proof}
We apply Theorem~\ref{thm:regular_posterior_comparison}. Its assumptions from
Theorem~\ref{thm:regular_free_energy_univ} were verified for the sparse
Bernoulli model in the proof of Theorem~\ref{thm:sparse_bernoulli_univ}; it
therefore remains to construct the Gaussian perturbation, verify the associated
operator conditions, and identify the posterior functionals.

\emph{Gaussian perturbation.}
To instantiate the perturbation in
Theorem~\ref{thm:regular_posterior_comparison}, take the local directions
$H_\alpha$ in \eqref{eq:sparse_H_direction}, set
$B_\alpha=\lambda^{-1}\Id$, and form the product operators $H$ and $B$ as in
\eqref{eq:regular_posterior_product_operators}. The representation
\eqref{eq:gaussian_location_xi_representation}, together with the product
decomposition \eqref{eq:gaussian_product_operators}, gives
\[
Y(t)
=
\sum_{\alpha\in[n]^p}
\left\{
\langle\xi_\alpha,\eta_\alpha(t)\rangle
+\langle A_\alpha\eta_\alpha(\theta),\eta_\alpha(t)\rangle
-\frac12\langle K_\alpha\eta_\alpha(t),\eta_\alpha(t)\rangle
\right\},
\]
where $\xi_\alpha\sim\normal(0,\Sigma_\alpha)$ independently over $\alpha$.
Using the same Gaussian variables, \eqref{eq:Y_beta}--\eqref{eq:Y_tilde} set
$Y_\beta(t)=Y(t)+\beta\widetilde Y(t)$, where
\[
\widetilde Y(t)
=
\sum_{\alpha\in[n]^p}
\left\{
\langle\xi_\alpha,B_\alpha\eta_\alpha(t)\rangle
+a\langle H_\alpha[\eta_\alpha(\theta)-\eta_\alpha(t)],
\eta_\alpha(\theta)-\eta_\alpha(t)\rangle
\right\}.
\]
Since the $\xi_\alpha$ are centered, expanding the deterministic terms and
computing the covariance of the random terms yields the perturbed operators
in \eqref{eq:sparse_perturbed_gaussian_operators} and hence the mean and
covariance of $Y_\beta$ in \eqref{eq:sparse_gaussian_perturbed_moments}. 

\emph{Operator conditions.}
Equations \eqref{eq:bernoulli_gaussian_operators} and
\eqref{eq:sparse_H_direction} give
\[
\Sigma_\alpha=\lambda J_\alpha J_\alpha^\top,
\qquad
H_\alpha=J_\alpha J_\alpha^\top=\Sigma_\alpha B_\alpha.
\]
Consequently, the product operators in
\eqref{eq:regular_posterior_product_operators} satisfy $H=\Sigma B$, and hence
the componentwise range condition \eqref{eq:H_Sigma}. Moreover,
\[
\|H_\alpha\|_{\mathrm{op}}=\|J_\alpha\|^2,
\qquad
\|B_\alpha\|_{\mathrm{op}}=\lambda^{-1}.
\]
The assumption $\sup_{\alpha\in[n]^p}\|J_\alpha\|=O(1)$ in
Theorem~\ref{thm:sparse_bernoulli_univ}, together with fixed $\lambda>0$,
therefore gives $\mathfrak h_n=O(1)$.

\emph{Posterior functionals.}
The test function required by
Theorem~\ref{thm:regular_posterior_comparison} is the function
$\phi_\theta$ in \eqref{eq:phi_theta_lin}, with $H$ given by
\eqref{eq:sparse_H_direction}. Substitution into the normalized functionals in
\eqref{eq:sparse_directional_observables}--
\eqref{eq:sparse_directional_observables_G} gives
\begin{equation*}
\begin{aligned}
a=-1:&\quad
\Phi_n(\theta)=-\cR_n(\theta),
&
\Phi_n^{\mathrm G}(\theta,\beta)
&=-\cR_n^{\mathrm G}(\theta,\beta),
\\
a=0:&\quad
\Phi_n(\theta)=\cV_n(\theta),
&
\Phi_n^{\mathrm G}(\theta,\beta)
&=\cV_n^{\mathrm G}(\theta,\beta).
\end{aligned}
\end{equation*}

The moment estimates from the proof of
Theorem~\ref{thm:sparse_bernoulli_univ} and the bound
$\mathfrak h_n=O(1)$ now verify all assumptions of
Theorem~\ref{thm:regular_posterior_comparison}. Applying that theorem with
$a=-1$ and $a=0$ first gives the sandwich error
$C\{\eps+\widetilde\delta_n/\eps\}$. Since $\delta_n\to0$, we have
$\widetilde\delta_n\lesssim\delta_n$ for all sufficiently large $n$, yielding
the asserted mean squared error and posterior variance bounds, respectively.
\end{proof}

\subsection{Proof of Corollary~\ref{cor:sparse_posterior_limits}}
\label{app:proof_sparse_posterior_limits}

\begin{proof}
The proof of Theorem~\ref{thm:sparse_bernoulli_univ} places the sparse
Bernoulli model in the regular local setting of
Theorem~\ref{thm:regular_free_energy_univ}. The proof of
Theorem~\ref{thm:sparse_bernoulli_posterior_univ} then verifies the additional
perturbation conditions of Theorem~\ref{thm:regular_posterior_comparison} and
identifies the corresponding normalized functionals for $a=0$ and $a=-1$.
Moreover, $\widetilde\delta_n=\delta_n(1+\delta_n)\to0$. Since this error does
not depend on $\theta$, its $\nu_n$-average is also
$\widetilde\delta_n$, as required by
Corollary~\ref{cor:regular_posterior_limit}. It therefore remains only to
identify the pointwise limits of the two averaged perturbed Gaussian free
energies. Since $\pi_n$ and $\nu_n$ are product measures, the results
of~\cite{guionnet2025estimating} give, for each $a\in\{0,-1\}$ and $\beta$ in a
neighborhood of zero,
$\mathcal G_{n,a}^\nu(\beta)\to\mathcal G_a^\nu(\beta)$.

For $a=0$, the additive term in
\eqref{eq:sparse_perturbed_gaussian_free_energy} vanishes, so the averaged
perturbed Gaussian free energy is $\mathcal G_{n,0}^{\nu}(\beta)$. Moreover,
the corresponding normalized functional is
$\Phi_n(\theta)=\cV_n(\theta)$. Since $\delta_n\to0$ and the results
of~\cite{guionnet2025estimating} give
$\mathcal G_{n,0}^{\nu}\to\mathcal G_0^\nu$ pointwise near the origin,
Corollary~\ref{cor:regular_posterior_limit} yields the posterior variance
bounds.

For $a=-1$, set
\[
q_n^\nu
\coloneqq
\int_{\Theta_n}
\frac1{n^2}\left(\sum_{i=1}^n\theta_i^2\right)^2
\nu_n(\dd\theta).
\]
Averaging \eqref{eq:sparse_perturbed_gaussian_free_energy} over $\nu_n$
yields
\[
\int_{\Theta_n}F_n^{\mathrm G}(\theta,\beta)\,\nu_n(\dd\theta)
=
\mathcal G_{n,-1}^{\nu}(\beta)-\beta q_n^\nu.
\]
Because $\nu_n=\nu_0^{\otimes n}$,
\[
q_n^\nu
=
\frac1n\int_{[-1,1]}x^4\,\nu_0(\dd x)
+\left(1-\frac1n\right)
\left(\int_{[-1,1]}x^2\,\nu_0(\dd x)\right)^2
\longrightarrow q^\nu.
\]
Thus the averaged perturbed free energies converge pointwise to
\[
\mathcal F^\nu(\beta)
\coloneqq \mathcal G_{-1}^\nu(\beta)-\beta q^\nu
=-\mathcal M^\nu(\beta).
\]
Thus $\mathcal F^\nu$ is convex and $\mathcal M^\nu$ is concave.
In this case, the normalized functional is
$\Phi_n(\theta)=-\cR_n(\theta)$. Corollary~\ref{cor:regular_posterior_limit}
therefore bounds the subsequential limits of
$-\int_{\Theta_n}\cR_n(\theta)\,\nu_n(\dd\theta)$ in terms of the one-sided
derivatives of $\mathcal F^\nu$ at zero. Multiplying these inequalities by
$-1$ reverses their order. Since $\mathcal F^\nu=-\mathcal M^\nu$, we have
$-\partial_+\mathcal F^\nu(0)=\partial_+\mathcal M^\nu(0)$ and
$-\partial_-\mathcal F^\nu(0)=\partial_-\mathcal M^\nu(0)$. Consequently,
\[
\partial_+\mathcal M^\nu(0)
\le
\liminf_{n\to\infty}
\int_{\Theta_n}\cR_n(\theta)\,\nu_n(\dd\theta)
\le
\limsup_{n\to\infty}
\int_{\Theta_n}\cR_n(\theta)\,\nu_n(\dd\theta)
\le
\partial_-\mathcal M^\nu(0),
\]
as asserted.
\end{proof}

\section{Gaussian Measures and Cameron--Martin Spaces}
\label{app:gaussian_measures_banach}

For completeness, this appendix reviews, without modification, the abstract
theory of Gaussian measures on locally convex spaces developed
in~\cite{bogachev:1998gaussian}, including the Cameron--Martin space, the
Cameron--Martin formula, and the abstract Gaussian integration-by-parts
theorem. Appendix~\ref{app:gibp_separable_banach} then specializes this
theory to $\cX=C(T)$ and proves the exact Gaussian integration-by-parts
identity used in Appendix~\ref{app:proofs_section3}. Finally,
Appendix~\ref{app:gaussian_loglik_derivation} applies the Cameron--Martin
formula to the Gaussian location log-likelihood process of
Definition~\ref{def:gaussian_loc} in the infinite-dimensional Hilbert space
setting.

Following~\cite[Appendix~A.1]{bogachev:1998gaussian}, consider a locally convex
space $\cX$.\footnote{In this appendix, $\cX$ denotes a general locally
convex space, following~\cite{bogachev:1998gaussian}; in
Appendix~\ref{app:gibp_separable_banach}, we specialize to $\cX=C(T)$, the
separable Banach space used in the main text.} Two dual spaces of $\cX$ are relevant below:
the \emph{topological dual}
\begin{equation}\label{eq:topological_dual}
\cX^*
\coloneqq
\{f:\cX\to\bbR:\ f\text{ is linear and continuous}\},
\end{equation}
and the \emph{algebraic dual} $\cX'\coloneqq\{f:\cX\to\bbR:\ f\text{ is linear}\}$,
so that $\cX^*\subseteq\cX'$. We also use the algebraic dual of $\cX^*$,
denoted $(\cX^*)'$, whose elements are linear maps $\cX^*\to\bbR$.

For a locally convex space $\cX$, let $\cE(\cX)$ denote the $\sigma$-field
generated by $\cX^*$; see~\cite[Appendix~A.3]{bogachev:1998gaussian}. A
probability measure $\gamma$ on $(\cX,\cE(\cX))$ is
\emph{Gaussian} if, for every $f\in\cX^*$, the pushforward $\gamma\circ f^{-1}$
is a one-dimensional Gaussian measure on $\bbR$; see
\cite[Section~2.2, Definition~2.2.1(ii)]{bogachev:1998gaussian}. Throughout
this appendix, let $\gamma$ denote a Gaussian measure on a locally convex
space $\cX$.


Let $\mu$ be a measure on $\cE(\cX)$ such that
$\cX^*\subset L^2(\mu)$. Following
\cite[Definition~2.2.7]{bogachev:1998gaussian}, the element
$a_\mu\in(\cX^*)'$ defined by
\[
a_\mu(f)=\int_\cX f(x)\,\mu(\dd x),
\]
is called the \emph{mean} of $\mu$. The \emph{covariance operator} of $\mu$
is the map $R_\mu:\cX^*\to(\cX^*)'$ defined by
\[
R_\mu(f)(g)
\coloneqq
\int_\cX
\bigl[f(x)-a_\mu(f)\bigr]
\bigl[g(x)-a_\mu(g)\bigr]
\,\mu(\dd x).
\]
Next, let $\cX_\gamma^*$ be the closure in $L^2(\gamma)$ of the centered
functionals $\{f-a_\gamma(f):f\in\cX^*\}$.
For later use, extend the covariance operator to a map
$R_\gamma:\cX_\gamma^*\to(\cX^*)'$ by defining, for $f\in\cX_\gamma^*$,
\[
R_\gamma(f)(g)
\coloneqq
\int_\cX f(x)\bigl(g(x)-a_\gamma(g)\bigr)\,\gamma(\dd x),
\qquad g\in\cX^*.
\]
Following the notation introduced after
\cite[Definition~2.2.7]{bogachev:1998gaussian}, for $h\in\cX$, set
\[
|h|_{H(\gamma)}
\coloneqq
\sup\bigl\{f(h):f\in\cX^*,\ R_\gamma(f)(f)\le1\bigr\}.
\]
The \emph{Cameron--Martin space} of $\gamma$ is
\begin{equation}\label{eq:cameron_martin_space}
H(\gamma)
\coloneqq
\{h\in\cX:\ |h|_{H(\gamma)}<\infty\}.
\end{equation}

\begin{lemma}[{\cite[Lemma~2.4.1]{bogachev:1998gaussian}}]
\label{lem:banach_cameron_martin_char}
A vector $h\in\cX$ belongs to $H(\gamma)$ precisely when there exists
$g\in\cX_\gamma^*$ such that $h=R_\gamma(g)$.\footnote{As
in~\cite{bogachev:1998gaussian}, this equality is understood under the
canonical identification of $\cX$ with its image in $(\cX^*)'$.}
In this case, $|h|_{H(\gamma)}=\|g\|_{L^2(\gamma)}$.
\end{lemma}

If $h=R_\gamma(g)$, then $g$ is said to be associated with, or generated by,
$h$. In this case, we use the notation $\widehat h\coloneqq g$. The relation
determining $\widehat h$ is
\begin{equation}\label{eq:hat_h_characterization}
f(h)
=
\int_\cX \bigl[f(x)-a_\gamma(f)\bigr]\widehat h(x)\,\gamma(\dd x),
\qquad \forall f\in\cX^*.
\end{equation}

\subsection{The Cameron--Martin Formula}
\label{app:cameron_martin_formula}

For $h\in\cX$, let $\gamma_h\coloneqq\gamma(\,\cdot-h)$ denote the translate
of $\gamma$ by $h$. The following result gives the Radon--Nikodym density when
$h$ is generated by an element of $\cX_\gamma^*$.

\begin{theorem}[{\cite[Corollary~2.4.3]{bogachev:1998gaussian}}]
\label{thm:cameron_martin_formula}
Let $\gamma$ be a Gaussian measure on a locally convex space $\cX$. Suppose
that $h\in\cX$ satisfies $h=R_\gamma(g)$ for some $g\in\cX_\gamma^*$. Then
$\gamma_h$ and $\gamma$ are equivalent, and
\begin{equation}\label{eq:cameron_martin_formula}
\frac{\dd\gamma_h}{\dd\gamma}(x)
=
\exp\left\{g(x)-\frac12|h|_{H(\gamma)}^2\right\}.
\end{equation}
\end{theorem}

Formula~\eqref{eq:cameron_martin_formula} is the Cameron--Martin formula. The
next theorem identifies all shifts for which equivalence holds and gives the
complementary singularity statement.

\begin{theorem}[{\cite[Theorem~2.4.5]{bogachev:1998gaussian}}]
\label{thm:cameron_martin_dichotomy}
Let $\gamma$ be a Gaussian measure on a locally convex space $\cX$, and let
$h\in\cX$.
\begin{enumerate}
\item If $|h|_{H(\gamma)}=\infty$, then $\gamma_h$ and $\gamma$ are mutually
singular.
\item If $|h|_{H(\gamma)}<\infty$, then $\gamma_h$ and $\gamma$ are
equivalent.
\end{enumerate}
In particular,
\begin{equation}\label{eq:cameron_martin_equivalent_shifts}
H(\gamma)
=
\{h\in\cX:\ \gamma_h\sim\gamma\}
=
\{h\in\cX:\ |h|_{H(\gamma)}<\infty\}
=
\cX\cap R_\gamma(\cX_\gamma^*).
\end{equation}
\end{theorem}

Thus, $H(\gamma)$ characterizes precisely those directions $h\in\cX$ along
which translation leaves $\gamma$ quasi-invariant, in the sense that
$\gamma_h$ and $\gamma$ have the same null sets.

\subsection{Gaussian Integration by Parts}
\label{app:gaussian_integration_by_parts}

Let $F:\cX\to\bbR$. For $x,h\in\cX$, the directional derivative of $F$ at
$x$ along $h$ is defined, whenever the limit exists, by
\[
\partial_h F(x)
\coloneqq
\lim_{t\to0}\frac{F(x+th)-F(x)}{t}.
\]
The following theorem is the abstract Gaussian integration-by-parts formula
underlying our comparison results.

\begin{theorem}[{\cite[Theorem~5.1.8]{bogachev:1998gaussian}}]
\label{thm:banach_gibp}
Let $\gamma$ be a Radon Gaussian measure on a locally convex space $\cX$, and
let $h\in H(\gamma)$ with associated element $\widehat h\in\cX_\gamma^*$.
Suppose $F:\cX\to\bbR$ is $\gamma$-measurable and, for $\gamma$-a.e.\ $x$, the
map $t\mapsto F(x+th)$ is locally absolutely continuous on $\bbR$. If
$\partial_hF$ and $F\widehat h$ are both $\gamma$-integrable, then
\[
\int_\cX \partial_hF(x)\,\gamma(\dd x)
=
\int_\cX F(x)\widehat h(x)\,\gamma(\dd x).
\]
\end{theorem}

The associated element $\widehat h$ in
Theorem~\ref{thm:banach_gibp} is obtained from $h$
through~\eqref{eq:hat_h_characterization}. We also recall the Radon
terminology used in the statement. A finite Borel measure $\mu$ on a
topological space is \emph{Radon} if, for every Borel set $B$ and every
$\eps>0$, there exists a compact set $K_\eps\subset B$ such that
$\mu(B\setminus K_\eps)<\eps$. A Radon measure $\gamma$ on a locally convex
space $\cX$ is called a \emph{Radon Gaussian measure} if its restriction to
$\cE(\cX)$ is a Gaussian measure; see
\cite[Section~3.1, Definition~3.1.1]{bogachev:1998gaussian}.

\section{\texorpdfstring{Gaussian Integration by Parts\\on the Separable
Banach Space $C(T)$}{Gaussian Integration by Parts on the Separable Banach
Space C(T)}}
\label{app:gibp_separable_banach}

We now specialize the abstract theory of
Sections~\ref{app:cameron_martin_formula}--\ref{app:gaussian_integration_by_parts}
to $\cX=C(T)$, equipped with the supremum norm, as in
Section~\ref{sec:general_frame}, and use it to prove Gaussian integration
by parts on this separable Banach space. By the
Riesz--Markov--Kakutani representation theorem, $\cX^*$ is identified with
the space of finite signed regular Borel measures on $T$, with duality
pairing $\langle x,\nu\rangle=\int_T x(t)\,\nu(\dd t)$.

Let $Y=(Y(t))_{t\in T}$ be a Gaussian process with continuous sample paths,
viewed as a $\cX$-valued Gaussian random element, and let
$\gamma=\mathsf{Law}(Y)$ denote its law on $\cX$. Write $m(t)=\E[Y(t)]$ for
the mean function and $K(s,t)=\cov(Y(s),Y(t))$ for the covariance kernel.

\begin{lemma}[Setup for $\cX=C(T)$]
\label{lem:banach_setup_combined}
$\cX=C(T)$ is a locally convex space, $\gamma$ is a Radon Gaussian measure on
$\cX$ with $\cX^*\subseteq L^2(\gamma)$, and, for every $t\in T$,
$h_t\coloneqq K(\cdot,t)\in H(\gamma)$, with associated element
$\widehat{h}_t(x)=x(t)-m(t)$.
\end{lemma}

\begin{proof}
Since $\cX=C(T)$ is a normed space, it is locally convex, and since it is
also separable, every Borel probability measure on it, in particular
$\gamma$, is Radon. As $Y$ is $\cX$-valued Gaussian, $\gamma$ is a Gaussian
measure on $\cX$, and for $f\in\cX^*$, $f(Y)$ is one-dimensional Gaussian,
hence $f\in L^2(\gamma)$. This proves the first assertion.

For the second assertion, fix $t\in T$ and let $f\in\cX^*$ be represented by
the finite signed Borel measure $\nu$ on $T$, so that
$f(x)=\int_T x(s)\,\nu(\dd s)$. Then
\[
f(h_t)=\int_T K(s,t)\,\nu(\dd s)=\cov\Bigl(\int_T Y(s)\,\nu(\dd s),\,Y(t)\Bigr)
=\cov(f(Y),Y(t)),
\]
using Fubini's theorem and the definition of $K$.

We now verify $h_t\in H(\gamma)$. Recall that
\[
|h_t|_{H(\gamma)}
\coloneqq
\sup\bigl\{f(h_t):f\in\cX^*,\ R_\gamma(f)(f)\le1\bigr\},
\]
and that $h_t\in H(\gamma)$ precisely when this supremum is finite; we
must therefore bound $f(h_t)$ over exactly the functionals $f$ satisfying
$R_\gamma(f)(f)\le1$. First, $f(h_t)$ is itself a value
of $R_\gamma$: taking $g=\delta_t\in\cX^*$ (where $\delta_t(x)\coloneqq
x(t)$, so $a_\gamma(\delta_t)=\int_\cX x(t)\,\gamma(\dd x)=\E[Y(t)]=m(t)$)
in the definition of $R_\gamma$ recorded in
Appendix~\ref{app:gaussian_measures_banach},
\begin{align*}
R_\gamma(f)(\delta_t)
&=
\int_\cX\bigl[f(x)-a_\gamma(f)\bigr]\bigl[x(t)-m(t)\bigr]\,\gamma(\dd x)\\
&=
\E\bigl[(f(Y)-\E f(Y))(Y(t)-m(t))\bigr]
=
\cov(f(Y),Y(t))
=
f(h_t),
\end{align*}
using the computation of $f(h_t)$ above. Next, by its definition,
$R_\gamma(f)(g)=\int_\cX[f(x)-a_\gamma(f)][g(x)-a_\gamma(g)]\,\gamma(\dd x)$
is exactly the $L^2(\gamma)$ inner product
$\langle f-a_\gamma(f),\,g-a_\gamma(g)\rangle_{L^2(\gamma)}$ of the
centered functions $f-a_\gamma(f)$ and $g-a_\gamma(g)$ (both are
$\gamma$-square-integrable elements of $\cX^*\subseteq L^2(\gamma)$).
Applying the Cauchy--Schwarz inequality for this inner product to
$f(h_t)=R_\gamma(f)(\delta_t)$ gives
\[
|f(h_t)|
=
|R_\gamma(f)(\delta_t)|
\le
\sqrt{R_\gamma(f)(f)}\;\sqrt{R_\gamma(\delta_t)(\delta_t)}.
\]
Taking $f=g=\delta_t$ in the same definition of $R_\gamma$ gives
$R_\gamma(\delta_t)(\delta_t)=\int_\cX[x(t)-m(t)]^2\,\gamma(\dd x)
=\mathrm{Var}(Y(t))=K(t,t)$, using the definition of the covariance
kernel $K$; write
$\mathrm{sd}(Y(t))\coloneqq\sqrt{\mathrm{Var}(Y(t))}=K(t,t)^{1/2}$ for the
standard deviation of the real-valued Gaussian random variable $Y(t)$, so
that $\sqrt{R_\gamma(\delta_t)(\delta_t)}=\mathrm{sd}(Y(t))$. For every
$f$ in the constraint set defining $|h_t|_{H(\gamma)}$---that is, every
$f\in\cX^*$ with $R_\gamma(f)(f)\le1$---the bound above becomes
\[
|f(h_t)|\le\sqrt{1}\cdot\mathrm{sd}(Y(t))=\mathrm{sd}(Y(t)).
\]
Taking the supremum of the left-hand side over exactly this constraint
set gives $|h_t|_{H(\gamma)}\le\mathrm{sd}(Y(t))=K(t,t)^{1/2}<\infty$
(finite because $K(t,t)=\mathrm{Var}(Y(t))<\infty$, $Y(t)$ being a
real-valued Gaussian random variable), hence $h_t\in H(\gamma)$. By
Lemma~\ref{lem:banach_cameron_martin_char}, $h_t=R_\gamma(g)$ for some
$g\in\cX_\gamma^*$, and it remains to identify $g$. The centered coordinate
functional $\delta_t-a_\gamma(\delta_t)$, where $\delta_t(x)=x(t)$, satisfies,
for every such $f$,
\[
R_\gamma(\delta_t-a_\gamma(\delta_t))(f)
=
\E\bigl[(f(Y)-\E f(Y))(Y(t)-m(t))\bigr]
=
\cov(f(Y),Y(t))
=
f(h_t),
\]
which is precisely the defining
relation~\eqref{eq:hat_h_characterization} for the associated element.
Hence $\widehat{h}_t=\delta_t-a_\gamma(\delta_t)$, that is,
$\widehat{h}_t(x)=x(t)-m(t)$.
\end{proof}

\begin{lemma}[Exact Gaussian integration by parts]
\label{lem:exact_gibp_restated}
Let $Y$ be a Gaussian element of $\cX$ with mean $m\in\cX$ and covariance
kernel $K:T\times T\to\bbR$. Let $G:\cX\to\bbR$ be continuously Fr\'echet
differentiable, with $\E|G(Y)|<\infty$ and
$\E\|DG(Y)\|_{\mathrm{op}}<\infty$, and suppose that
\[
\E\bigl[|G(Y)|\,|Y(t)-m(t)|\bigr]<\infty,
\qquad t\in T.
\]
Then, for every $t\in T$,
\[
\E[G(Y)Y(t)]
=
\E[G(Y)]\,m(t)
+
\E\bigl[DG(Y)[K(\cdot,t)]\bigr].
\]
\end{lemma}

\begin{proof}
We apply Theorem~\ref{thm:banach_gibp} to this setting. Fix $t\in T$ and set
$h=K(\cdot,t)$. By Lemma~\ref{lem:banach_setup_combined}, $\cX=C(T)$ is a
locally convex space, $\gamma=\mathsf{Law}(Y)$ is a Radon Gaussian measure on
$\cX$ with $\cX^*\subseteq L^2(\gamma)$, and $h\in H(\gamma)$, with associated
element $\widehat h(x)=x(t)-m(t)$.

Since $G$ is continuously Fr\'echet differentiable, for every $x\in\cX$ the
map $r\mapsto G(x+rh)$ has derivative $r\mapsto DG(x+rh)[h]$, which is
continuous in $r$; a continuously differentiable function of one real
variable is (locally) absolutely continuous, with its classical derivative
recovering the Fr\'echet derivative along $h$: $\partial_hG(x)=DG(x)[h]=DG(x)[K(\cdot,t)]$.
This verifies the local absolute continuity hypothesis of
Theorem~\ref{thm:banach_gibp} for every $x\in\cX$, hence for
$\gamma$-a.e.~$x$. Moreover, the assumed product integrability gives
\[
\E|G(Y)\widehat h(Y)|
=
\E\bigl[|G(Y)|\,|Y(t)-m(t)|\bigr]
<
\infty,
\]
while
\[
\E|\partial_hG(Y)|
=
\E\bigl|DG(Y)[K(\cdot,t)]\bigr|
\le
\|K(\cdot,t)\|_\infty\,\E\|DG(Y)\|_{\mathrm{op}}
<
\infty
\]
by the assumed bound $\E\|DG(Y)\|_{\mathrm{op}}<\infty$ and boundedness of
$K(\cdot,t)$ (continuity of $K$ on the compact set $T\times T$). Thus both
$\partial_hG$ and $G\widehat h$ are $\gamma$-integrable, and
Theorem~\ref{thm:banach_gibp} applies with this choice of $h$, giving
\[
\int_\cX DG(x)[K(\cdot,t)]\,\gamma(\dd x)
=
\int_\cX G(x)\bigl(x(t)-m(t)\bigr)\,\gamma(\dd x).
\]
Since $\gamma=\mathsf{Law}(Y)$, this reads
\[
\E\bigl[DG(Y)[K(\cdot,t)]\bigr]
=
\E[G(Y)(Y(t)-m(t))]
=
\E[G(Y)Y(t)]-\E[G(Y)]\,m(t),
\]
which is the asserted identity. As $t\in T$ was arbitrary, the
identity holds for every $t\in T$.
\end{proof}

\begin{lemma}[Exact Gaussian integration by parts for the log-partition
functional, fixed direction]
\label{lem:exact_gibp_two_fixed_h}
Let $Y_1,Y_2$ be jointly Gaussian elements of $\cX$ with means
$m_1,m_2\in\cX$ and cross-covariance kernel
$K_{12}(s,t)\coloneqq\cov(Y_1(s),Y_2(t))$, and let $f$ be the log-partition
functional of~\eqref{eq:log_partition_functional}. Then, for every
$h\in\cX$ and $t\in T$,
\[
\E\bigl[Df(Y_1)[h]\,(Y_2(t)-m_2(t))\bigr]
=
\E\bigl[D^2f(Y_1)[h,K_{12}(\cdot,t)]\bigr].
\]
\end{lemma}

\begin{proof}
When $Y_2=Y_1$, this is exactly Lemma~\ref{lem:exact_gibp_restated} applied with
$G(x)\coloneqq Df(x)[h]$, since $DG(x)[k]=D^2f(x)[h,k]$. Since $Y_2$ differs
from $Y_1$ in general, Lemma~\ref{lem:exact_gibp_restated} does not apply directly to
$Y_1$ alone, and we need an additional product-space argument: we apply it
to the pair $(Y_1,Y_2)$, viewed as a single Gaussian element of the product
space $\cX\times\cX$.

Write $\underline\cX\coloneqq\cX\times\cX$, equipped with the max norm
$\|(x_1,x_2)\|\coloneqq\max(\|x_1\|_\infty,\|x_2\|_\infty)$, and
$\underline Y\coloneqq(Y_1,Y_2)$. Since $(Y_1,Y_2)$ is jointly Gaussian by
hypothesis, $\underline Y$ is a Gaussian element of $\underline\cX$, with mean
$\underline m\coloneqq(m_1,m_2)\in\underline\cX$ and, writing
$K_{11}(s,t)\coloneqq\cov(Y_1(s),Y_1(t))$ and
$K_{22}(s,t)\coloneqq\cov(Y_2(s),Y_2(t))$, covariance structure, for every
$s,t\in T$,
\[
\cov\bigl((Y_1(s),Y_2(s)),\,(Y_1(t),Y_2(t))\bigr)
=
\begin{pmatrix}
K_{11}(s,t) & K_{12}(s,t)\\
K_{12}(t,s) & K_{22}(s,t)
\end{pmatrix};
\]
We now establish, for $\underline\cX$ and $\underline Y$ in place of $\cX$
and $Y$, the same conclusion as Lemma~\ref{lem:banach_setup_combined}. Since
$\underline\cX=\cX\times\cX$ is again a separable Banach space, it is
locally convex, and, as in the proof of
Lemma~\ref{lem:banach_setup_combined}, every Borel probability measure on
it, in particular
$\underline\gamma\coloneqq\mathsf{Law}(\underline Y)$---is Radon. Since
$\underline Y$ is a Gaussian element of $\underline\cX$, for every
$\ell\in\underline\cX^*$ the pushforward $\underline\gamma\circ\ell^{-1}$ is
one-dimensional Gaussian, that is, $\ell(\underline Y)$ is a
one-dimensional Gaussian random variable, so
\[
\E\bigl[\ell(\underline Y)^2\bigr]<\infty,
\qquad\text{that is,}\qquad
\ell\in L^2(\underline\gamma);
\]
this proves $\underline\cX^*\subseteq L^2(\underline\gamma)$. The remaining
assertion, that $\underline h\in H(\underline\gamma)$ for a suitable
$\underline h\in\underline\cX$ with associated element $\widehat{\underline
h}$ is verified separately below.

By the Riesz--Markov--Kakutani representation theorem applied to each
factor, $\underline\cX^*\cong\cX^*\times\cX^*$: every bounded linear
functional $\ell$ on $\underline\cX$ has the form
$\ell(x_1,x_2)=\int_Tx_1\,\dd\nu_1+\int_Tx_2\,\dd\nu_2$ for a pair
$(\nu_1,\nu_2)$ of finite signed Borel measures on $T$.
Fix $t\in T$ and set
$\underline h\coloneqq\bigl(K_{12}(\cdot,t),K_{22}(\cdot,t)\bigr)\in\underline\cX$.
We claim $\underline h\in H(\underline\gamma)$, with associated element
$\widehat{\underline h}(x_1,x_2)=x_2(t)-m_2(t)$. Indeed, for $\ell$ represented by $(\nu_1,\nu_2)$ as above, Fubini's theorem
gives
\begin{align*}
\ell(\underline h)
&=
\int_TK_{12}(s,t)\,\nu_1(\dd s)+\int_TK_{22}(s,t)\,\nu_2(\dd s)\\
&=
\cov\Bigl(\int_TY_1\,\dd\nu_1+\int_TY_2\,\dd\nu_2,\,Y_2(t)\Bigr)
=
\cov\bigl(\ell(\underline Y),Y_2(t)\bigr),
\end{align*}
by $\cov(Y_1(s),Y_2(t))=K_{12}(s,t)$ and $\cov(Y_2(s),Y_2(t))=K_{22}(s,t)$.
By Cauchy--Schwarz,
\[
|\ell(\underline h)|
=
|\cov(\ell(\underline Y),Y_2(t))|
\le
\mathrm{sd}(Y_2(t))\,\|\ell-a_{\underline\gamma}(\ell)\|_{L^2(\underline\gamma)};
\]
taking the supremum over $\ell$ with
$\|\ell-a_{\underline\gamma}(\ell)\|_{L^2(\underline\gamma)}\le1$ gives
$|\underline h|_{H(\underline\gamma)}<\infty$, hence
$\underline h\in H(\underline\gamma)$.

By Lemma~\ref{lem:banach_cameron_martin_char}, $\underline h=R_{\underline\gamma}(g)$
for some $g\in\underline\cX_{\underline\gamma}^*$; it remains to identify
$g$. Exactly as in the proof of Lemma~\ref{lem:banach_setup_combined}, the
centered functional $\ell_0(x_1,x_2)\coloneqq x_2(t)-m_2(t)$ satisfies, for
every $\ell$,
\begin{align*}
R_{\underline\gamma}(\ell_0)(\ell)
&=
\E\bigl[(\ell(\underline Y)-\E\ell(\underline Y))(Y_2(t)-m_2(t))\bigr]\\
&=
\cov(\ell(\underline Y),Y_2(t))
=
\ell(\underline h),
\end{align*}
so $\widehat{\underline h}=\ell_0$, i.e.,
$\widehat{\underline h}(x_1,x_2)=x_2(t)-m_2(t)$.

Fix $h\in\cX$ and set $\underline G(x_1,x_2)\coloneqq Df(x_1)[h]$, depending
only on the first coordinate. Since $\underline\cX=\cX\times\cX$, the
Fr\'echet derivative of $\underline G$ at $(x_1,x_2)$ decomposes into
partial derivatives in the directions $(k_1,0)$ and $(0,k_2)$:
\begin{align*}
D\underline G(x_1,x_2)[(k_1,k_2)]
&=
D\underline G(x_1,x_2)[(k_1,0)]+D\underline G(x_1,x_2)[(0,k_2)]\\
&=
D\underline G(x_1,x_2)[(k_1,0)],
\end{align*}
since $\underline G(x_1,x_2)$ does not depend on $x_2$, so
$D\underline G(x_1,x_2)[(0,k_2)]=0$ for every $k_2$, and
$D\underline G(x_1,x_2)[(k_1,0)]$ is the derivative of the scalar function
$x_1\mapsto Df(x_1)[h]$ in the direction $k_1$, which, since $D^2f(x_1)$ is
by definition the Fr\'echet derivative of $x_1\mapsto Df(x_1)$, equals
$D^2f(x_1)[k_1,h]=D^2f(x_1)[h,k_1]$ by symmetry of the second derivative.
Hence
\begin{equation}
\label{eq:DG_formula}
D\underline G(x_1,x_2)[(k_1,k_2)]=D^2f(x_1)[h,k_1]+0=D^2f(x_1)[h,k_1],
\end{equation}
a single real number for each direction $(k_1,k_2)$. By
Lemma~\ref{lem:logpartition_smooth}, $\underline G$ is continuously
Fr\'echet differentiable, with $D\underline G$ as just computed; moreover,
using $\|Df\|_{\mathrm{op}}\le1$ and $\|D^2f\|_{\mathrm{op}}\le1$,
\[
\E|\underline G(\underline Y)|\le\|h\|_\infty<\infty,
\qquad
\E\|D\underline G(\underline Y)\|_{\mathrm{op}}\le\|h\|_\infty<\infty,
\]


We now apply Theorem~\ref{thm:banach_gibp}. We have already shown that
$\underline\cX$ is locally convex, that $\underline\gamma$ is a Radon
Gaussian measure on $\underline\cX$, and that $\underline h\in
H(\underline\gamma)$ with associated element $\widehat{\underline h}$; it
remains to verify local absolute continuity and the two integrability
hypotheses for $\underline G$.

Since $\underline G$ is continuously Fr\'echet differentiable, for every
$(x_1,x_2)\in\underline\cX$ the map $r\mapsto\underline G((x_1,x_2)+r\underline h)$
has derivative $r\mapsto D\underline G((x_1,x_2)+r\underline h)[\underline h]$,
continuous in $r$; as in the proof of Lemma~\ref{lem:exact_gibp_restated}, this
gives local absolute continuity, with
\begin{align*}
\partial_{\underline h}\underline G(x_1,x_2)
&=
D\underline G(x_1,x_2)[\underline h]\\
&=
D\underline G(x_1,x_2)\bigl[(K_{12}(\cdot,t),K_{22}(\cdot,t))\bigr]\\
&=
D^2f(x_1)\bigl[h,K_{12}(\cdot,t)\bigr],
\end{align*}
using~\eqref{eq:DG_formula}, so only the first component $K_{12}(\cdot,t)$
enters. For the integrability
hypotheses, $\E|\underline G(\underline Y)|\le\|h\|_\infty<\infty$ was shown
above, and
\[
\E\bigl|\partial_{\underline h}\underline G(\underline Y)\bigr|
=
\E\bigl|D^2f(Y_1)[h,K_{12}(\cdot,t)]\bigr|
\le
\|h\|_\infty\,\|K_{12}(\cdot,t)\|_\infty\,\E\|D^2f(Y_1)\|_{\mathrm{op}}
<
\infty,
\]
using $\|D^2f\|_{\mathrm{op}}\le1$ and boundedness of $K_{12}(\cdot,t)$.
Furthermore, since
$\widehat{\underline h}(\underline Y)=Y_2(t)-m_2(t)$ and
$|Df(Y_1)[h]|\le\|h\|_\infty$,
\[
\E\bigl|
\underline G(\underline Y)\widehat{\underline h}(\underline Y)
\bigr|
\le
\|h\|_\infty\,\E|Y_2(t)-m_2(t)|
<
\infty,
\]
because $Y_2(t)-m_2(t)$ is a centered real Gaussian random variable.

\emph{Conclusion.} Theorem~\ref{thm:banach_gibp} applies to
$\underline\gamma$, $\underline G$, and $\underline h$, giving
\[
\int_{\underline\cX}\partial_{\underline h}\underline G(\underline x)\,\underline\gamma(\dd\underline x)
=
\int_{\underline\cX}\underline G(\underline x)\,\widehat{\underline h}(\underline x)\,\underline\gamma(\dd\underline x),
\]
that is,
\[
\E\bigl[D^2f(Y_1)[h,K_{12}(\cdot,t)]\bigr]
=
\E\bigl[\underline G(\underline Y)\,(Y_2(t)-m_2(t))\bigr]
=
\E\bigl[Df(Y_1)[h]\,(Y_2(t)-m_2(t))\bigr].
\]
Rearranging (and recalling $t\in T$, $h\in\cX$ were arbitrary) gives, for
every $h\in\cX$ and $t\in T$,
\begin{equation}
\label{eq:appE_fixed_h}
\E\bigl[Df(Y_1)[h]\,(Y_2(t)-m_2(t))\bigr]
=
\E\bigl[D^2f(Y_1)[h,K_{12}(\cdot,t)]\bigr].
\end{equation}
\end{proof}

Lemma~\ref{lem:exact_gibp_two_fixed_h} does not solve our problem: it
only gives the identity with a fixed direction
$h\in\cX$. What we need to prove is
the identity with the full random element $Y_2$,
\[
\E\bigl[Df(Y_1)[Y_2]\bigr]
=
\E[Df(Y_1)][m_2]
+
\bigl\langle\E[D^2f(Y_1)],K_{12}\bigr\rangle.
\]
If $\cX$ were finite-dimensional, this extension would be immediate:
apply Lemma~\ref{lem:exact_gibp_two_fixed_h} with $h$ equal to each
one-hot coordinate direction $e_i$ in turn, and sum over $i$ to recover
the identity above, using $Y_2=\sum_i(Y_2)_ie_i$. Carrying this out
requires a countable orthonormal basis of directions $h$ to sum over,
which $\cX=C(T)$ does not carry under its own norm.

Nevertheless, a countable orthonormal basis becomes available on passing to
$L^2(T,\pi)$. First, the canonical map from $\cX$ into $L^2(T,\pi)$ is continuous, since
\[
\|h\|_{L^2(\pi)}\le\|h\|_\infty,
\qquad h\in\cX,
\]
using that $\pi$ is a probability measure, so $\pi(T)=1$.
Second, $\cX$ is dense in $L^2(T,\pi)$, and $C(T)$ is itself separable,
since $T$ is a compact metric space; hence $L^2(T,\pi)$ is separable and
admits a countable orthonormal basis, which, by density of $\cX$, may be
chosen to lie in $\cX$ (apply Gram--Schmidt to a dense sequence in
$\cX$; see~\cite[Section~3.6, Problem~6]{kreyszig1991introductory}).

To complete this argument, we must extend $Df(x)$ and $D^2f(x)$, a
priori bounded only with respect to the $\|\cdot\|_\infty$ norm on
$\cX$, to bounded operators on $L^2(T,\pi)$.

\begin{lemma}[$L^2(\pi)$ bounds for $Df$ and $D^2f$]
\label{lem:df_d2f_l2pi_extension}
For every $x\in\cX$ and $h,k\in\cX$,
\begin{align*}
|Df(x)[h]|&\le e^{\|x\|_\infty}\|h\|_{L^2(\pi)},\\
|D^2f(x)[h,k]|&\le2e^{2\|x\|_\infty}\|h\|_{L^2(\pi)}\|k\|_{L^2(\pi)}.
\end{align*}
\end{lemma}

The content of these bounds is that $Df(x)$ and $D^2f(x)$, though a
priori only known to be bounded on $\cX$ with respect to the sup norm
$\|\cdot\|_\infty$ (Lemma~\ref{lem:logpartition_smooth}), are in fact
uniformly bounded on $\cX$ with respect to the strictly weaker
$L^2(\pi)$-norm, that is,
$\sup\{|Df(x)[h]|:h\in\cX,\|h\|_{L^2(\pi)}\le1\}<\infty$ (and similarly
for $D^2f(x)$); this is what lets us control the error from truncating
the basis expansion of $Y_2$ in the proof of
Lemma~\ref{lem:exact_gibp_two_logpartition} below.

\begin{proof}
Write $Z(x)\coloneqq\int_Te^{x(s)}\pi(\dd s)$ and
$w_x\coloneqq\dd\pi_x/\dd\pi=e^{x(\cdot)}/Z(x)$; since $T$ is compact and
$x\in\cX$ continuous, $w_x$ is bounded, and since
$Z(x)\ge e^{-\|x\|_\infty}\pi(T)$,
\begin{equation}
\label{eq:wx_bound}
\|w_x\|_\infty\le\pi(T)^{-1}e^{2\|x\|_\infty}.
\end{equation}
By Lemma~\ref{lem:logpartition_smooth},
$Df(x)[h]=\int_Th\,\dd\pi_x=\int_Thw_x\,\dd\pi=\langle h,w_x\rangle_{L^2(\pi)}$
for $h\in\cX$. Since $w_x\ge0$ with $\int_Tw_x\,\dd\pi=1$,
Cauchy--Schwarz gives
$\|w_x\|_{L^2(\pi)}^2\le\|w_x\|_\infty\int_Tw_x\,\dd\pi=\|w_x\|_\infty$, so
\begin{align*}
|Df(x)[h]|
&=
|\langle h,w_x\rangle_{L^2(\pi)}|
\le
\|h\|_{L^2(\pi)}\|w_x\|_{L^2(\pi)}\\
&\le
\|w_x\|_\infty^{1/2}\|h\|_{L^2(\pi)}
\le
\pi(T)^{-1/2}e^{\|x\|_\infty}\|h\|_{L^2(\pi)},
\qquad h\in\cX,
\end{align*}
using~\eqref{eq:wx_bound}. Since $\cX$ is dense in $L^2(T,\pi)$, $Df(x)$
extends uniquely to a bounded linear functional on $L^2(T,\pi)$, still
denoted $Df(x)$, given by the same formula
$Df(x)[h]=\langle h,w_x\rangle_{L^2(\pi)}$ for $h\in L^2(T,\pi)$.

Similarly, since
$D^2f(x)[h,k]=\int_Thk\,\dd\pi_x-\bigl(\int_Th\,\dd\pi_x\bigr)\bigl(\int_Tk\,\dd\pi_x\bigr)$,
for $h,k\in\cX$,
\begin{align*}
|D^2f(x)[h,k]|
&\le
\|w_x\|_\infty\|h\|_{L^2(\pi)}\|k\|_{L^2(\pi)}
+
\|w_x\|_\infty^{1/2}\|h\|_{L^2(\pi)}\cdot\|w_x\|_\infty^{1/2}\|k\|_{L^2(\pi)}\\
&=
2\|w_x\|_\infty\|h\|_{L^2(\pi)}\|k\|_{L^2(\pi)}
\le
2\pi(T)^{-1}e^{2\|x\|_\infty}\|h\|_{L^2(\pi)}\|k\|_{L^2(\pi)},
\end{align*}
using~\eqref{eq:wx_bound} again, so $D^2f(x)$ extends by density to a
bounded bilinear form on $L^2(T,\pi)\times L^2(T,\pi)$, with the same bound.
\end{proof}

The following statement is identical to
Lemma~\ref{lem:exact_gibp_two} in Appendix~\ref{app:proofs_section3}; it is
repeated here in order to provide its proof.

\begin{lemma}[Exact two-element Gaussian integration by parts for the
log-partition functional]
\label{lem:exact_gibp_two_logpartition}
Let $Y_1,Y_2$ be jointly Gaussian elements of $\cX$ with means
$m_1,m_2\in\cX$ and cross-covariance kernel
$K_{12}(s,t)\coloneqq\cov(Y_1(s),Y_2(t))$. For the log-partition functional
$f$ defined in~\eqref{eq:log_partition_functional},
\[
\E\bigl[\sD f(Y_1)[Y_2]\bigr]
=
\E[\sD f(Y_1)][m_2]
+
\bigl\langle\E[\sD^2f(Y_1)],K_{12}\bigr\rangle.
\]
\end{lemma}

\begin{proof}
Since $m_2$ is deterministic, $\E[Df(Y_1)[m_2]]=\E[Df(Y_1)][m_2]$. Since,
by linearity of $Df(Y_1)[\cdot]$,
$Df(Y_1)[Y_2]=Df(Y_1)[Y_2-m_2]+Df(Y_1)[m_2]$, it suffices to prove
\begin{equation}
\label{eq:appE_centered_goal}
\E\bigl[Df(Y_1)[Y_2-m_2]\bigr]
=
\bigl\langle\E[D^2f(Y_1)],K_{12}\bigr\rangle;
\end{equation}
adding $\E[Df(Y_1)][m_2]$ to both sides then gives the identity of the
lemma. The strategy is to expand $Y_2-m_2$ in a countable orthonormal
basis of $L^2(T,\pi)$ lying in $\cX$, apply
Lemma~\ref{lem:exact_gibp_two_fixed_h} term by term, and pass to the
limit using the $L^2(\pi)$ bounds of Lemma~\ref{lem:df_d2f_l2pi_extension}.

By Lemma~\ref{lem:exact_gibp_two_fixed_h}, for every $h\in\cX$ and $t\in T$,
\begin{equation}
\label{eq:appE_fixed_h_lem34}
\E\bigl[Df(Y_1)[h]\,(Y_2(t)-m_2(t))\bigr]
=
\E\bigl[D^2f(Y_1)[h,K_{12}(\cdot,t)]\bigr].
\end{equation}
It remains to extend this identity from a fixed direction $h$ and point
$t$ to the full random element $Y_2$.

\emph{Setup.} Since $\cX=C(T)$ is separable and dense in $L^2(T,\pi)$,
$L^2(T,\pi)$ is separable; fix an orthonormal basis $\{\psi_j\}_{j\ge1}$
of $L^2(T,\pi)$ with $\psi_j\in\cX$ for every $j$ (apply Gram--Schmidt to
a dense sequence in $\cX$; see~\cite[Section~3.6, Problem~6]{kreyszig1991introductory}).
Since $Y_2$ is a Gaussian element of $C(T)$, $\E[e^{c\|Y_2\|_\infty}]<\infty$
for every $c>0$ (Fernique's theorem), and, since $Y_2-m_2\in\cX\subset
L^2(T,\pi)$, Parseval's identity gives
\begin{equation}
\label{eq:appE_parseval_Y2}
Y_2-m_2=\sum_{j\ge1}\eta_j\psi_j
\quad\text{in }L^2(T,\pi),\quad\text{a.s.},
\qquad
\eta_j\coloneqq\langle Y_2-m_2,\psi_j\rangle_{L^2(\pi)},
\end{equation}
with $\|Y_2^{(n)}-Y_2\|_{L^2(\pi)}\to0$ almost surely, where
$Y_2^{(n)}\coloneqq m_2+\sum_{j\le n}\eta_j\psi_j$.

\emph{Left side of~\eqref{eq:appE_centered_goal}.} By
Lemma~\ref{lem:df_d2f_l2pi_extension},
\begin{align*}
\bigl|Df(Y_1)[Y_2-m_2]-Df(Y_1)[Y_2^{(n)}-m_2]\bigr|
&=
\bigl|Df(Y_1)[Y_2-Y_2^{(n)}]\bigr|\\
&\le
e^{\|Y_1\|_\infty}\|Y_2-Y_2^{(n)}\|_{L^2(\pi)},
\end{align*}
which tends to $0$ a.s.\ by~\eqref{eq:appE_parseval_Y2}.

The coefficients $\eta_j=\langle Y_2-m_2,\psi_j\rangle_{L^2(\pi)}$
in~\eqref{eq:appE_parseval_Y2} are random variables, with
$\E[\eta_j]=\langle\E[Y_2]-m_2,\psi_j\rangle_{L^2(\pi)}=0$ for every $j$
(since $\E[Y_2]=m_2$); consequently
$\E[Y_2^{(n)}]=m_2+\sum_{j\le n}\E[\eta_j]\psi_j=m_2$ for every $n$, so
$m_2$ is also the mean of $Y_2^{(n)}$. Since $\{\psi_j\}$ is orthonormal,
$\|Y_2^{(n)}-m_2\|_{L^2(\pi)}^2=\sum_{j\le n}\eta_j^2$; applying Bessel's
inequality to the (pathwise fixed) element $Y_2-m_2\in L^2(\pi)$, with
Fourier coefficients $(\eta_j)_{j\ge1}$ relative to $\{\psi_j\}$, gives
$\sum_{j\le n}\eta_j^2\le\sum_{j\ge1}\eta_j^2\le\|Y_2-m_2\|_{L^2(\pi)}^2$
a.s., hence
\[
\|Y_2^{(n)}-m_2\|_{L^2(\pi)}\le\|Y_2-m_2\|_{L^2(\pi)}
\qquad\text{a.s., for every }n.
\]
The difference
$|Df(Y_1)[Y_2-m_2]-Df(Y_1)[Y_2^{(n)}-m_2]|$ is therefore dominated by
$2e^{\|Y_1\|_\infty}\|Y_2-m_2\|_{L^2(\pi)}$, which is integrable by Fernique's
theorem and Cauchy--Schwarz's inequality. Dominated convergence and linearity of
$Df(Y_1)[\cdot]$ on the finite sum
$Y_2^{(n)}-m_2=\sum_{j\le n}\eta_j\psi_j$ give
\begin{equation}
\label{eq:appE_lhs_limit}
\E[Df(Y_1)[Y_2-m_2]]=\lim_{n\to\infty}\sum_{j\le n}\E\bigl[\eta_j\,Df(Y_1)[\psi_j]\bigr].
\end{equation}

\emph{Identifying each term of~\eqref{eq:appE_lhs_limit}.} Fix $j$.
Multiply both sides of~\eqref{eq:appE_fixed_h_lem34}, with $h=\psi_j$, by
$\psi_j(t)$ and integrate over $t\in T$ against $\pi(\dd t)$. On the
left,
\begin{align*}
&\int_T\E\bigl[Df(Y_1)[\psi_j]\,(Y_2(t)-m_2(t))\bigr]\psi_j(t)\,\pi(\dd t)\\
&=
\E\Bigl[Df(Y_1)[\psi_j]\int_T(Y_2(t)-m_2(t))\psi_j(t)\,\pi(\dd t)\Bigr]\\
&=
\E\bigl[Df(Y_1)[\psi_j]\,\eta_j\bigr],
\end{align*}
using Fubini's theorem to move $\int_T(\cdot)\,\pi(\dd t)$ inside
$\E[\cdot]$; this is justified by
\[
\int_T\E\bigl[
|Df(Y_1)[\psi_j]|\,|Y_2(t)-m_2(t)|\,|\psi_j(t)|
\bigr]\,\pi(\dd t)
\le
\|\psi_j\|_\infty^2\,\E\|Y_2-m_2\|_\infty
<
\infty
\]
and Fernique's theorem
and the definition of $\eta_j$ in~\eqref{eq:appE_parseval_Y2}.

On the right, define
\begin{equation}
\label{eq:cK12_def}
\cK_{12}\psi_j(s)\coloneqq\int_TK_{12}(s,t)\psi_j(t)\,\pi(\dd t).
\end{equation}
The same manipulation gives
\begin{align*}
&\int_T\E\bigl[D^2f(Y_1)[\psi_j,K_{12}(\cdot,t)]\bigr]\psi_j(t)\,\pi(\dd t)\\
&=
\E\Bigl[\int_TD^2f(Y_1)[\psi_j,K_{12}(\cdot,t)]\psi_j(t)\,\pi(\dd t)\Bigr]\\
&=
\E\Bigl[D^2f(Y_1)\Bigl[\psi_j,\int_TK_{12}(\cdot,t)\psi_j(t)\,\pi(\dd t)\Bigr]\Bigr]\\
&=
\E\bigl[D^2f(Y_1)[\psi_j,\cK_{12}\psi_j]\bigr],
\end{align*}
where the first equality is Fubini's theorem; indeed, its integrand is
bounded in absolute value by
$\|\psi_j\|_\infty^2\|K_{12}\|_\infty$, which is integrable over
$T$. The
second pulls the integral over $t$ inside the second argument of the
bounded bilinear form $D^2f(Y_1)[\psi_j,\cdot]$
(Lemma~\ref{lem:df_d2f_l2pi_extension}), and the third is the
definition~\eqref{eq:cK12_def}.

Equating the two sides,
\begin{equation}
\label{eq:appE_per_j}
\E\bigl[\eta_j\,Df(Y_1)[\psi_j]\bigr]
=
\E\bigl[D^2f(Y_1)[\psi_j,\cK_{12}\psi_j]\bigr].
\end{equation}
Combining~\eqref{eq:appE_lhs_limit} and~\eqref{eq:appE_per_j},
\begin{equation}
\label{eq:appE_reduced}
\E[Df(Y_1)[Y_2-m_2]]
=
\lim_{n\to\infty}\sum_{j\le n}\E\bigl[D^2f(Y_1)[\psi_j,\cK_{12}\psi_j]\bigr].
\end{equation}

\emph{Right side of~\eqref{eq:appE_centered_goal}, for a fixed point.} By
\eqref{eq:appE_reduced}, it remains to show
\[
\lim_{n\to\infty}\sum_{j\le n}\E\bigl[D^2f(Y_1)[\psi_j,\cK_{12}\psi_j]\bigr]
=
\bigl\langle\E[D^2f(Y_1)],K_{12}\bigr\rangle.
\]
We obtain this at the end of the proof by exchanging $\lim_n$ and the
expectation over $Y_1$ (dominated convergence) and evaluating the
resulting a.s.\ limit $\lim_n\sum_{j\le n}D^2f(Y_1)[\psi_j,\cK_{12}\psi_j]$.
Dominated convergence requires two ingredients: the value of this limit,
and a bound on the partial sums that is integrable and uniform in $n$. We
establish both first for an arbitrary \emph{fixed} (non-random) $x\in\cX$;
substituting $x=Y_1$ at the end turns them into, respectively, an a.s.\
statement and an integrable dominating random variable.

Fix $x\in\cX$; we now show
\begin{equation}
\label{eq:appE_rhs_fixed_x}
\lim_{n\to\infty}\sum_{j\le n}D^2f(x)[\psi_j,\cK_{12}\psi_j]
=
\bigl\langle D^2f(x),K_{12}\bigr\rangle,
\end{equation}
together with a bound, uniform in $n$ and depending on $x$
only~(\eqref{eq:appE_rhs_uniform_bound} below), needed at the end of the
proof to pass from~\eqref{eq:appE_rhs_fixed_x} at $x=Y_1$ to its
expectation over $Y_1$.

Let $(Y_1',Y_2')$ be an independent copy of $(Y_1,Y_2)$; since
$\cov(Y_1'(s),Y_2'(t))=K_{12}(s,t)$, the pair $(Y_1'-m_1,Y_2'-m_2)$ is a
valid choice of $(V,W)$ in the definition of the contraction
$\langle\cdot,K_{12}\rangle$. As in~\eqref{eq:appE_parseval_Y2},
$Y_2'-m_2=\sum_j\eta_j'\psi_j$ in $L^2(T,\pi)$, with
$\eta_j'\coloneqq\langle Y_2'-m_2,\psi_j\rangle_{L^2(\pi)}$; as for
$\eta_j$ in the left-side argument, $\E[\eta_j']=0$.

For each $s\in T$, the same Fubini computation used to
obtain~\eqref{eq:appE_per_j} gives
\begin{align*}
\cov\bigl(\eta_j',Y_1'(s)\bigr)
&=
\cov\Bigl(\int_T\bigl(Y_2'(t)-m_2(t)\bigr)\psi_j(t)\,\pi(\dd t),\,Y_1'(s)\Bigr)\\
&=
\int_T\cov\bigl(Y_2'(t),Y_1'(s)\bigr)\,\psi_j(t)\,\pi(\dd t)\\
&=
\int_TK_{12}(s,t)\,\psi_j(t)\,\pi(\dd t)
=
\cK_{12}\psi_j(s),
\end{align*}
using $\cov(Y_1'(s),Y_2'(t))=K_{12}(s,t)$ and the
definition~\eqref{eq:cK12_def} of $\cK_{12}\psi_j$. Since $\E[\eta_j']=0$,
\begin{align*}
\E\bigl[\eta_j'\,(Y_1'(s)-m_1(s))\bigr]
&=
\E[\eta_j'\,Y_1'(s)]-m_1(s)\,\E[\eta_j']\\
&=
\E[\eta_j'\,Y_1'(s)]
=
\cov\bigl(\eta_j',Y_1'(s)\bigr)
=
\cK_{12}\psi_j(s)
\end{align*}
for every $s\in T$. Moreover,
\[
\E\bigl[|\eta_j'|\,\|Y_1'-m_1\|_\infty\bigr]
\le
\bigl(\E\|Y_2'-m_2\|_{L^2(\pi)}^2\bigr)^{1/2}
\bigl(\E\|Y_1'-m_1\|_\infty^2\bigr)^{1/2}
<
\infty
\]
by $|\eta_j'|\le\|Y_2'-m_2\|_{L^2(\pi)}$, Cauchy--Schwarz, and
Fernique's theorem. Hence the $\cX$-valued Bochner integral
$\E[\eta_j'(Y_1'-m_1)]$ equals $\cK_{12}\psi_j$.

Hence, using symmetry of $D^2f(x)$ and linearity in the
first argument,
\begin{align*}
\E\bigl[\eta_j'\,D^2f(x)[Y_1'-m_1,\psi_j]\bigr]
&=
D^2f(x)\bigl[\E[\eta_j'(Y_1'-m_1)],\psi_j\bigr]\\
&=
D^2f(x)[\cK_{12}\psi_j,\psi_j]
=
D^2f(x)[\psi_j,\cK_{12}\psi_j].
\end{align*}
Summing over $j\le n$ and using bilinearity on the resulting finite sum,
\begin{equation}
\label{eq:appE_rhs_finite_n}
\begin{aligned}
Y_2'^{(n)}&\coloneqq m_2+\sum_{j\le n}\eta_j'\psi_j;\\
\sum_{j\le n}D^2f(x)[\psi_j,\cK_{12}\psi_j]
&=
\E\bigl[D^2f(x)[Y_1'-m_1,Y_2'^{(n)}-m_2]\bigr].
\end{aligned}
\end{equation}
By~\eqref{eq:appE_parseval_Y2} applied to $Y_2'$, for every realization
$\omega$, $Y_2'^{(n)}(\omega)\to Y_2'(\omega)$ in $L^2(\pi)$-norm
By
Lemma~\ref{lem:df_d2f_l2pi_extension}, applied with the fixed direction
$Y_1'-m_1$ in the first slot,
\begin{align*}
&\bigl|D^2f(x)[Y_1'-m_1,Y_2'^{(n)}-m_2]-D^2f(x)[Y_1'-m_1,Y_2'-m_2]\bigr|\\
&=
\bigl|D^2f(x)[Y_1'-m_1,Y_2'^{(n)}-Y_2']\bigr|\\
&\le
2e^{2\|x\|_\infty}\,\|Y_1'-m_1\|_{L^2(\pi)}
\bigl\|Y_2'^{(n)}-Y_2'\bigr\|_{L^2(\pi)},
\end{align*}
which tends to $0$ pathwise, for every $\omega$, by the convergence just
noted. By Bessel's inequality (as for $Y_2$ in the left-side argument),
$\|Y_2'^{(n)}-m_2\|_{L^2(\pi)}\le\|Y_2'-m_2\|_{L^2(\pi)}$, so, by the
triangle inequality, $\|Y_2'^{(n)}-Y_2'\|_{L^2(\pi)}\le
2\|Y_2'-m_2\|_{L^2(\pi)}$, so, for every $n$, the random variable
\[
\bigl|D^2f(x)[Y_1'-m_1,Y_2'^{(n)}-m_2]-D^2f(x)[Y_1'-m_1,Y_2'-m_2]\bigr|
\]
whose limit (in $n$) was just shown to be $0$ pathwise is dominated by
\begin{equation}
\label{eq:appE_rhs_dct_dominating}
4e^{2\|x\|_\infty}\,\|Y_1'-m_1\|_{L^2(\pi)}\,\|Y_2'-m_2\|_{L^2(\pi)},
\end{equation}
integrable by Cauchy--Schwarz and Fernique's theorem. Dominated
convergence applied to~\eqref{eq:appE_rhs_finite_n},
with~\eqref{eq:appE_rhs_dct_dominating} as the dominating random
variable, gives
\begin{align*}
\lim_{n\to\infty}\sum_{j\le n}D^2f(x)[\psi_j,\cK_{12}\psi_j]
&=
\E\bigl[D^2f(x)[Y_1'-m_1,Y_2'-m_2]\bigr]\\
&=
\bigl\langle D^2f(x),K_{12}\bigr\rangle,
\end{align*}
which is~\eqref{eq:appE_rhs_fixed_x}.

For the uniform bound, Lemma~\ref{lem:df_d2f_l2pi_extension} applied
to~\eqref{eq:appE_rhs_finite_n}, together with Bessel's inequality
($\|Y_2'^{(n)}-m_2\|_{L^2(\pi)}\le\|Y_2'-m_2\|_{L^2(\pi)}$), gives
\begin{align*}
\Bigl|\sum_{j\le n}D^2f(x)[\psi_j,\cK_{12}\psi_j]\Bigr|
&=
\bigl|\E\bigl[D^2f(x)[Y_1'-m_1,Y_2'^{(n)}-m_2]\bigr]\bigr|\\
&\le
2e^{2\|x\|_\infty}\,\E\bigl[\|Y_1'-m_1\|_{L^2(\pi)}\|Y_2'-m_2\|_{L^2(\pi)}\bigr]
=
2e^{2\|x\|_\infty}\,C_0,
\end{align*}
for every $n$, where
$C_0\coloneqq\E\bigl[\|Y_1'-m_1\|_{L^2(\pi)}\|Y_2'-m_2\|_{L^2(\pi)}\bigr]<\infty$
(finite by Cauchy--Schwarz and Fernique's theorem) is a \emph{fixed}
constant, depending on neither $x$ nor $n$: the auxiliary copy
$(Y_1',Y_2')$ has already been integrated out in bounding the left-hand
side, which is itself a deterministic number once $x$ is fixed. This
bound holds for \emph{every} $x\in\cX$; setting $x=Y_1$ turns it into an
almost-sure bound, uniform in $n$:
\begin{equation}
\label{eq:appE_rhs_uniform_bound}
\Bigl|\sum_{j\le n}D^2f(Y_1)[\psi_j,\cK_{12}\psi_j]\Bigr|
\le
2C_0\,e^{2\|Y_1\|_\infty}
\qquad\text{a.s., for every }n.
\end{equation}

\emph{Conclusion.} The right side of~\eqref{eq:appE_rhs_uniform_bound},
depending on $Y_1$ alone, is integrable by Fernique's theorem, so
dominated convergence permits exchanging $\lim_n$ and the outer
expectation over $Y_1$ in~\eqref{eq:appE_reduced}: using
\eqref{eq:appE_rhs_fixed_x} at $x=Y_1$,
\begin{align*}
\E[Df(Y_1)[Y_2-m_2]]
&=
\E\Bigl[\lim_{n\to\infty}\sum_{j\le n}D^2f(Y_1)[\psi_j,\cK_{12}\psi_j]\Bigr]\\
&=
\E\bigl[\langle D^2f(Y_1),K_{12}\rangle\bigr]
=
\bigl\langle\E[D^2f(Y_1)],K_{12}\bigr\rangle,
\end{align*}
where the last equality follows by linearity of the contraction in its
first argument. This is~\eqref{eq:appE_centered_goal}, which, as noted at the start of
the proof, gives the identity of the lemma.

\end{proof}

\section{\texorpdfstring{Deriving the Gaussian Location Log-Likelihood\\via
the Cameron--Martin Formula}{Deriving the Gaussian Location Log-Likelihood
via the Cameron--Martin Formula}}
\label{app:gaussian_loglik_derivation}

We now justify the claim, made after Definition~\ref{def:gaussian_loc}, that
the Cameron--Martin formula identifies the Gaussian location log-likelihood
process. Throughout this appendix, we take the locally convex space in
Appendix~\ref{app:gaussian_measures_banach} to be a real separable Hilbert
space $\cH$. Since $\cH$ is locally convex, the results of that appendix
apply directly in the present setting.

For the observation-model realization considered in this appendix, assume in
addition to Definition~\ref{def:gaussian_loc} that $K:\cH\to\cH$ is a
covariance operator; in particular, it is bounded, self-adjoint, nonnegative,
and trace class. Write
$\gamma\coloneqq \normal(0,K)=\mathsf{Law}(\xi)$ for
$\xi\sim\normal(0,K)$, viewed as a
centered Gaussian measure on $\cH$.

\begin{lemma}[Gaussian location density]
\label{lem:gaussian_shift_density}
Let $\eta:\Theta\to\cH$ be the feature map in
Definition~\ref{def:gaussian_loc}. For every $t\in T$, the Gaussian measures
$\normal(K\eta(t),K)$ and $\normal(0,K)$ on $\cH$ are equivalent, and
\begin{equation}
\label{eq:hilbert_RN_density}
\frac{\dd \normal(K\eta(t),K)}{\dd \normal(0,K)}(z)
=
\exp\Bigl\{\langle z,\eta(t)\rangle_\cH-\tfrac12\langle
K\eta(t),\eta(t)\rangle_\cH\Bigr\},
\end{equation}
for $\normal(0,K)$-a.e.\ $z$.
\end{lemma}

\begin{proof}
Fix $t\in T$, and set $u\coloneqq\eta(t)$ and $h_t\coloneqq Ku$. We first
verify directly from the definition that $h_t\in H(\gamma)$, where
$\gamma=\normal(0,K)$. By the Riesz representation theorem, every
$f\in\cH^*$ has the form $f_v(z)=\langle z,v\rangle_\cH$ for a unique
$v\in\cH$. The covariance identity gives
\[
R_\gamma(f_v)(f_v)
=
\E\langle\xi,v\rangle_\cH^2
=
\langle Kv,v\rangle_\cH.
\]
Since $f_v(h_t)=\langle Ku,v\rangle_\cH$, the definition of the
Cameron--Martin norm specializes in the present setting to
\[
|h_t|_{H(\gamma)}
=
\sup\bigl\{\langle Ku,v\rangle_\cH:
v\in\cH,\ \langle Kv,v\rangle_\cH\le1\bigr\}.
\]
Since $K$ is nonnegative and self-adjoint, the Cauchy--Schwarz inequality
in $\cH$ yields
\[
|f_v(h_t)|
=
|\langle Ku,v\rangle_\cH|
=
|\langle K^{1/2}u,K^{1/2}v\rangle_\cH|
\le
\langle Ku,u\rangle_\cH^{1/2}
\langle Kv,v\rangle_\cH^{1/2}.
\]
Set $c\coloneqq\langle Ku,u\rangle_\cH$. For every $v\in\cH$ satisfying
$\langle Kv,v\rangle_\cH\le1$, the preceding inequality gives
\[
f_v(h_t)\le |f_v(h_t)|\le c^{1/2}.
\]
Taking the supremum in the preceding representation of
$|h_t|_{H(\gamma)}$ yields
\[
|h_t|_{H(\gamma)}\le c^{1/2}<\infty,
\]
so $h_t\in H(\gamma)$. We now prove the reverse inequality. If $c>0$, take
$v_0=u/c^{1/2}$. Then
\[
R_\gamma(f_{v_0})(f_{v_0})
=
\langle Kv_0,v_0\rangle_\cH
=1,
\qquad
f_{v_0}(h_t)
=
\langle Ku,v_0\rangle_\cH
=c^{1/2}.
\]
Thus $f_{v_0}$ is included in the defining supremum, which implies
$|h_t|_{H(\gamma)}\ge c^{1/2}$. Combining the two inequalities gives
$|h_t|_{H(\gamma)}=c^{1/2}$. If $c=0$, then
$\|K^{1/2}u\|_\cH^2=0$, so $Ku=K^{1/2}(K^{1/2}u)=0$ and the same equality
holds. Therefore, in either case,
\[
|h_t|_{H(\gamma)}^2
=
\langle Ku,u\rangle_\cH
=
\langle K\eta(t),\eta(t)\rangle_\cH.
\]

We next identify the element associated with $h_t$. Define
$g_u(z)\coloneqq\langle z,u\rangle_\cH$. Since $g_u$ is a centered
continuous linear functional in $L^2(\gamma)$, it belongs to
$\cH_\gamma^*$. For every $v\in\cH$,
\[
R_\gamma(g_u)(f_v)
=
\E\bigl[\langle\xi,u\rangle_\cH\langle\xi,v\rangle_\cH\bigr]
=
\langle Ku,v\rangle_\cH
=
f_v(h_t).
\]
Thus, under the canonical identification of $\cH$ with its image in
$(\cH^*)'$, we have $h_t=R_\gamma(g_u)$. Hence, in the notation following
Lemma~\ref{lem:banach_cameron_martin_char}, the associated element is
\[
\widehat{h}_t(z)=g_u(z)=\langle z,\eta(t)\rangle_\cH.
\]

It remains to apply the Cameron--Martin formula. By definition,
$\gamma_{h_t}=\gamma(\,\cdot-h_t)$ is the law of $\xi+h_t$. Adding the
deterministic vector $h_t=K\eta(t)$ changes the mean of $\xi$ but not its
covariance, and therefore
\[
\gamma_{h_t}
=
\mathsf{Law}(\xi+K\eta(t))
=
\normal(K\eta(t),K).
\]
We have shown above that $h_t=R_\gamma(g_u)$, so the hypotheses of
Theorem~\ref{thm:cameron_martin_formula} are satisfied. That theorem first
gives $\gamma_{h_t}\sim\gamma$ and then gives, for $\gamma$-a.e.\ $z$,
\begin{align*}
\frac{\dd \normal(K\eta(t),K)}{\dd \normal(0,K)}(z)
&=
\frac{\dd\gamma_{h_t}}{\dd\gamma}(z) \\
&=
\exp\left\{\widehat{h}_t(z)-\frac12|h_t|_{H(\gamma)}^2\right\} \\
&=
\exp\Bigl\{\langle z,\eta(t)\rangle_\cH
-\tfrac12\langle K\eta(t),\eta(t)\rangle_\cH\Bigr\}.
\end{align*}
This proves~\eqref{eq:hilbert_RN_density}. The theorem also gives the
equivalence $\normal(K\eta(t),K)\sim\normal(0,K)$.
\end{proof}

Equation~\eqref{eq:hilbert_RN_density} identifies the Radon--Nikodym
derivative only for $\gamma_K$-a.e.\ $z$, where
$\gamma_K=\normal(0,K)$. Here the
qualifier ``$\gamma_K$-a.e.'' refers to the point $z$ in the observation
space $\cH$. To evaluate the density pointwise, for each $t\in T$ define
\[
L_t(z)
\coloneqq
\exp\Bigl\{\langle z,\eta(t)\rangle_\cH
-\tfrac12\langle K\eta(t),\eta(t)\rangle_\cH\Bigr\},
\qquad z\in\cH.
\]
The function $L_t$ is continuous, and Lemma~\ref{lem:gaussian_shift_density}
shows that it is a version of the Radon--Nikodym derivative in
\eqref{eq:hilbert_RN_density}. We henceforth use this particular continuous
version. Its values are uniquely determined on
$\operatorname{supp}(\gamma_K)$: if two continuous versions differed at a
point of the support, continuity would make them differ on an open
neighborhood of that point, which has positive $\gamma_K$-measure by the
definition of support, contradicting their equality for $\gamma_K$-a.e.\
$z$. Thus,
for every $z\in\operatorname{supp}(\gamma_K)$, the density in
\eqref{eq:hilbert_RN_density} is evaluated as $L_t(z)$.

We now verify that the data-generating observation lies in this support.
Under the data-generating law $Q_\theta$, let
$Z(\omega)=A\eta(\theta)+\xi(\omega)$, where
$\xi\sim\normal(0,\Sigma)$, and set
$\gamma_\Sigma=\normal(0,\Sigma)$. Denote the Cameron--Martin spaces of
$\gamma_K$ and $\gamma_\Sigma$ by $H(\gamma_K)$ and $H(\gamma_\Sigma)$,
respectively.
By~\cite[Theorem 3.6.1]{bogachev:1998gaussian},
\[
\operatorname{supp}(\gamma_K)=\overline{H(\gamma_K)},
\qquad
\operatorname{supp}(\gamma_\Sigma)=\overline{H(\gamma_\Sigma)}.
\]
Suppose that
\[
\operatorname{Range}(A)\subseteq\overline{H(\gamma_K)},
\qquad
\overline{H(\gamma_\Sigma)}\subseteq\overline{H(\gamma_K)}.
\]
Then $A\eta(\theta)\in\overline{H(\gamma_K)}$ and
$\xi(\omega)\in\overline{H(\gamma_\Sigma)}
\subseteq\overline{H(\gamma_K)}$ $Q_\theta$-a.s.\ Since
$\overline{H(\gamma_K)}$ is a closed linear subspace of $\cH$, it follows
that
\[
Q_\theta\!\left(
 \left\{\omega:
 Z(\omega)\in\overline{H(\gamma_K)}
 =\operatorname{supp}(\gamma_K)
 \right\}
\right)=1.
\]
Thus the qualifier ``$Q_\theta$-a.s.'' concerns the random outcome $\omega$,
whereas ``$\gamma_K$-a.e.'' above concerns the point $z$. On the displayed
$Q_\theta$-probability-one event, the continuous version fixed above can be
evaluated at $Z(\omega)$, and we define
\[
Y(t,\omega)
\coloneqq
\log L_t(Z(\omega))
=
\langle Z(\omega),\eta(t)\rangle_\cH
-\tfrac12\langle K\eta(t),\eta(t)\rangle_\cH.
\]
The process is therefore well defined $Q_\theta$-a.s., has the mean and
covariance structure of Definition~\ref{def:gaussian_loc}, and is the
Gaussian location log-likelihood process asserted there.

\end{appendix}

\end{document}